\documentclass[10pt,a4paper]{article}
\usepackage[utf8]{inputenc}
\usepackage{microtype}
\usepackage{amsmath,amsthm,amssymb}
\usepackage{mathrsfs} 
\usepackage[T1]{fontenc}
\usepackage{ae,aecompl}
\usepackage{times}
\usepackage[british]{babel}
\usepackage{aliascnt}
\usepackage[colorlinks,pagebackref,hypertexnames=false]{hyperref} 
\usepackage{url}
\usepackage[alphabetic,backrefs]{amsrefs}
\usepackage{nth}
\usepackage{float}
\usepackage[all]{xy}
\usepackage{tikz-cd} 
\usepackage{bookmark}
\usepackage[margin=1.7cm]{geometry}
\usepackage{enumitem} 

\newcommand{\rG}{{\rm G}}

\newcommand{\be}{{\bf e}}

\newcommand{\sX}{\mathscr{X}}

\newcommand{\R}{\mathbb{R}}

\newcommand{\so}{\mathfrak{so}}

\newcommand{\SO}{{\rm SO}}
\newcommand{\Sp}{{\rm Sp}}
\newcommand{\Spin}{{\rm Spin}}
\newcommand{\SU}{{\rm SU}}
\newcommand{\GL}{\mathrm{GL}}

\newcommand{\U}{{\rm U}}

\newcommand{\Rm}{{\rm Rm}}

\newcommand{\Diff}{\mathrm{Diff}}
\newcommand{\End}{{\mathrm{End}}}

\renewcommand{\epsilon}{\varepsilon}

\newcommand{\Ric}{{\rm Ric}}

\renewcommand{\Im}{\mathop{\mathrm{Im}}}

\newcommand{\tr}{\mathop{\mathrm{tr}}\nolimits}

\newcommand{\vol}{\mathrm{vol}}

\newcommand{\qandq}{\quad\text{and}\quad}
\newcommand{\qforq}{\quad\text{for}\quad}
\newcommand{\qwithq}{\quad\text{with}\quad}

\def\<{\mathopen{}\left<}
\def\>{\right>\mathclose{}}
\def\({\mathopen{}\left(}
\def\){\right)\mathclose{}}

\usepackage{multicol, color}

\definecolor{gold}{rgb}{0.85,.66,0}
\definecolor{cherry}{rgb}{0.9,.1,.2}
\definecolor{burgundy}{rgb}{0.8,.2,.2}
\definecolor{orangered}{rgb}{0.85,.3,0}
\definecolor{orange}{rgb}{0.85,.4,0}
\definecolor{olive}{rgb}{.45,.4,0}
\definecolor{lime}{rgb}{.6,.9,0}
\definecolor{green}{rgb}{.2,.7,0}
\definecolor{grey}{rgb}{.4,.4,.2}
\definecolor{brown}{rgb}{.4,.3,.1}

\newtheorem{theorem}{Theorem}
\newtheorem{proposition}[theorem]{Proposition}
\newtheorem{corollary}[theorem]{Corollary}
\newtheorem{lemma}[theorem]{Lemma}

\theoremstyle{remark} 
\newtheorem{remark}{Remark}

\theoremstyle{definition}
\newtheorem{definition}[theorem]{Definition}

\def\al{\alpha}
\def\be{\beta}

\def\w{\wedge}
\def\R{\mathbb{R}}

\def\Lm{\Lambda}
\def\lm{\lambda}
\def\om{\omega}
\def\Om{\Omega}
\def\vp{\varphi}
\def\ip{\raise1pt\hbox{\large$\lrcorner\ \!$}} 

\DeclareMathOperator{\divergence}{div}

\usepackage{enumitem}
\usepackage{mathtools}

\usepackage[textwidth=1.7cm,textsize=scriptsize]{todonotes}
\usepackage{comment}
\numberwithin{theorem}{section}
\numberwithin{remark}{section}

\newcommand{\lot}{\mathrm{l.o.t.}}

\allowdisplaybreaks

\usepackage{titling}
\usepackage{changepage}

\usepackage{authblk}
\author[1,2]{Udhav Fowdar}
\author[2]{Henrique N. S{\'a} Earp}
\affil[1]{Universit\`a di Torino (UniTo)}
\affil[2]{Universidade Estadual de Campinas (Unicamp)}

\begin{document}
\title{Flows of $\SU(2)$-structures}
\date{\today}
\maketitle
	

\begin{abstract}
This paper initiates a classification programme of flows of $\SU(2)$-structures on $4$-manifolds which have short-time existence and uniqueness. 
Our approach adapts a representation-theoretic method originally due to Bryant in the context of $\mathrm{G}_2$ geometry. We show how this strategy can also be used to deduce the number of geometric flows of a given $H$-structure; we illustrate this in the $\mathrm{G}_2$, $\mathrm{Spin}(7)$ and $\mathrm{SU}(3)$ cases. Our investigation also leads us to derive explicit expressions for the Ricci and self-dual Weyl curvature in terms of the intrinsic torsion of the underlying $\SU(2)$-structure. We compute the first variation formulae of all the quadratic functionals in the torsion; these provide natural building blocks for $\SU(2)$ gradient flows. In particular, our results demonstrate that both the negative gradient flow of the Dirichlet energy of the intrinsic torsion and the harmonic Ricci flow are parabolic after a modified DeTurck's trick.
\end{abstract}	
\begin{adjustwidth}{0.95cm}{0.95cm}
    \tableofcontents
\end{adjustwidth}


\section{Introduction}

This article is a thorough study of variations of $\SU(2)$-structures on a connected and oriented $4$-manifold $M$, as an instance of the abstract theory of flows of geometric $H$-structures, i.e., flows of tensor fields defining $H$-reductions of the frame bundle, for closed and connected subgroups $H \subset \mathrm{SO}(n)$, developed in \cites{LoubeauSaEarp2019,Fadel2022}. The subject currently attracts substantial interest, most notably in the cases of $H=\rG_2,\Spin(7)$, see references below and others therein. 

In particular, the recent work by Dwivedi-Gianniotis-Karigiannis \cite{Dwivedi2023} brings forth a systematic study of the second-order differential invariants of a $\rG_2$-structure, as tentative building blocks for \emph{all possible} such geometric flows, up to lower-order terms, which could be `heat-like' [ibid.], i.e. parabolic. By sorting out those  invariants which correspond to symmetric ($2$-tensors) and skew-symmetric (vector) objects, the authors are able to identify essentially all the \emph{independent} such degrees of freedom for the construction of short-time flows. This in turn is a concrete expression of the representation-theoretic perspective on $H$-structure torsion introduced by Bryant in \cite{Bryant06someremarks}, which allows one to determine beforehand \emph{how many} such independent invariants one ought to find, beyond which additional invariants will necessarily be determined by relations. In this context, our approach explicitly combines both narratives as a coherent programme, exploiting from the fact that  $\SU(2)$ is `sufficiently small' that its representation theory is still completely tractable, and yet `sufficiently nontrivial' to exhibit several interesting phenomena tying in geometric torsion and Riemannian geometry. 
It is worth pointing out that the smaller the group $H$, the greater the number of invariants, and hence the more possible ways of evolving the underlying $H$-structure; thereby making a full classification of parabolic flows the more difficult. 

\subsection{\texorpdfstring{$\SU(2)$}{}-flows}

Let $M$ be a $4$-manifold endowed with an $\SU(2)$-structure determined by the triple of $2$-forms $\om_1,\om_2,\om_3$. Using the induced Riemannian metric $g$, we have the well-known decomposition\footnote{We denote the space of symmetric $2$-tensors by $\Sigma^2(M)$ and reserve the notation $S^2$ for projection map of $2$-tensors in $\Sigma^2(M)$ } 
\[
\End(TM) \cong T^*M \otimes T^*M \cong \langle g \rangle \oplus \Sigma^2_0(M) \oplus \Lm^2_+(M)\oplus \Lm^2_-(M).
\]
While the latter decomposition is irreducible as a sum of $\SO(4)$-modules, it can be further split into $\SU(2)$-modules:
\[
\End(TM) \cong \langle g \rangle \oplus 
\bigoplus_{i=1}^3
\Sigma^2_{0,i}(M) \oplus 
\bigoplus_{i=1}^3 \langle\om_i\rangle\oplus 
\Lm^2_-(M),
\]
where the irreducible spaces $\Sigma^2_{0,i}(M)$ are defined by (\ref{def: S_{0,i}}). 
Thus, given a general $2$-tensor $A$, 
we can write $A= B + C \in \Sigma^2(M) \oplus \Lm^2(M)$ as
\begin{align*}
    B &= f_0 g + B_{0,1} + B_{0,2} + B_{0,3},\\
    C &= f_1 \om_1 + f_2 \om_2 + f_3 \om_3 + C^2_-,
\end{align*}
where $f_i$ are functions, each $B_{0,i}$ corresponds to the projection of the symmetric $2$-tensor $B$
onto $\Sigma^2_{0,i}(M)$, and $C^2_-$ is the anti-self-dual component of the $2$-form $C$. In the above notation, following the general theory developed in \cite{Fadel2022},  any variation of the $\SU(2)$-triple $\boldsymbol{\om}:=(\om_1,\om_2,\om_3)$ can be expressed as
\begin{equation}
    \frac{d}{dt} \boldsymbol{\om} = A \diamond \boldsymbol{\om},\label{eq: compact flow}
\end{equation}
where $\diamond$ denotes the endomorphism action, see (\ref{definitionofdiamondtensors}) for the precise definition. More explicitly, we can express (\ref{eq: compact flow}) as 
    \begin{gather}
        \frac{d}{dt} \om_1  = +f_0 \om_1 + f_3 \om_2 - f_2 \om_3 + B_{0,1} \diamond \om_1,\label{eq:evolutionom1}\\
        \frac{d}{dt} \om_2 = - f_3 \om_1 + f_0 \om_2 + f_1 \om_3 + B_{0,2} \diamond \om_2,\label{eq:evolutionom2}\\
        \frac{d}{dt} \om_3 = + f_2 \om_1 - f_1 \om_2 +f_0 \om_3 + B_{0,3} \diamond \om_3,\label{eq:evolutionom3}
    \end{gather}
    where each $B_{0,i} \diamond \om_i$ corresponds to an anti-self-dual $2$-form on $M$.
Since the term $C^2_-$ lies in the kernel of $\diamond\!\!\ \ \boldsymbol{\om}$ it does not occur in (\ref{eq:evolutionom1})-(\ref{eq:evolutionom3}) and hence, without loss of generality, we shall henceforth set it to zero. It follows that the infinitesimal deformations of $\SU(2)$-structures are given by sections of a rank $13$ vector bundle; flows of $\SU(2)$-structures are therefore determined by particular choices of such  sections. 

This abstract framework for $\SU(2)$-flows bears relations with several well-known geometric flows in the literature:
\begin{itemize}
    \item The $\SU(2)$-flow (\ref{eq:evolutionom1})-(\ref{eq:evolutionom3}) can be interpreted as a $\mathrm{U}(2)$-flow, say for the multi-tensor  $(g,J_1,\om_1)$, simply by setting $f_1=0$ and viewing $\om_2$ and $\om_3$ as being locally defined; or equivalently, by substituting the pair (\ref{eq:evolutionom2})-(\ref{eq:evolutionom3}) by  (\ref{eq:evolutionofJi2}), for $(i,j,k)=(1,2,3)$. In particular, the \emph{K\"ahler-Ricci flow} is obtained  when the almost complex structure $J_1$ is integrable  and $d\om_1=0$, so that  $f_2=f_3=0=B_{0,2}=B_{0,3}$ and $B= f_0 g + B_{0,1}= -2 \Ric(g)$. By contrast, the \emph{Ricci flow} is obtained when $f_1=f_2=f_3=0$ and $B=-2\Ric(g)$; this is called the \emph{Ricci $H$-flow} in \cite{Fadel2022}*{Example 1.26}.
    \item In \cite{LoubeauSaEarp2019} Loubeau and the second-named author introduced the \emph{harmonic $H$-flow}. When $H=\SU(2)$, the harmonic flow is obtained when $B=0$ and $f_i= \mathrm{div}(a_i)$, where $a_i$ are the torsion $1$-forms defined in Proposition \ref{torsionproposition}. Stationary points of this flow correspond to $H$-structures which are in certain sense in Coulomb gauge, cf.  \cite{Grigorian2017} and also Corollary \ref{Levicivitacorollary} below.
    \item It is worth pointing out that the geometric flow considered here is different from the \emph{hypersymplectic flow}  introduced by Fine and Yao in \cite{Fine2018}. In the latter framework, one starts with a triple of closed $2$-forms and evolve them within their cohomology classes; the closed triple in \cite{Fine2018} do \textit{not} constitute to an $\SU(2)$-triple unless they are critical points. By contrast, in our setup we already start with an $\SU(2)$-triple and instead evolve them without a priori preserving any differential or cohomological conditions.
\end{itemize}

\subsection{Main results}
In this article, we classify all the second-order invariants of an $\SU(2)$-structure. By choosing the deformation $A$ among  suitable linear combinations of these invariants, we are able to classify all `heat-like' flows. We then establish short-time existence and uniqueness for a large class of these $\SU(2)$-flows, using a modified version of DeTurck's trick akin to the Ricci flow case in  \cite{DeTurck83}. While the theory is formally similar to the $\mathrm{G}_2$ and $\Spin(7)$ cases recently investigated in \cites{Dwivedi2023, Dwivedi2024}, a few important differences emerge. First, our classification of these second-order invariants is entirely based on the representation-theoretic method of \cite{Bryant06someremarks}; indeed in many ways this paper is an adaptation of Bryant's ideas to the $\SU(2)$ framework. 
Second, unlike with those groups, certain $\SU(2)$-modules occur with multiplicities, and this gives rise to many more ways of using the same second-order invariants to define different flows. This article is an attempt to unify the ideas of \cite{Bryant06someremarks}, \cite{Dwivedi2023} and \cite{Fadel2022}, by proposing a systematic classification procedure for second-order flows of $H$-structures and their parabolicity regimes, which we carry out in detail for $H=\SU(2)$. We highlight the following results, by relevance:
\begin{itemize}

    \item We show in Theorem \ref{theorem: classification of second-order invariants} that there are 22 (independent) second-order $\SU(2)$-invariants. 
    With the exception of the anti-self-dual Weyl curvature, all the remaining 21 invariants can be used to evolve the underlying $\SU(2)$-structure. In full generality this leads to a 72-parameter family of $\SU(2)$-flows. 
    
    \item Our method of derivation provides a systematic way of classifying geometric flows of $H$-structures which evolve by its second-order invariants. We explain in \S\ref{sec: comparison with g2 and spin7} how one can already deduce the number of possible second-order $H$-flows only from the knowledge of its representation. We illustrate this concretely in the $\SU(3)$, $\mathrm{G}_2$ and $\Spin(7)$ cases recovering some recent results of \cites{Dwivedi2023, Dwivedi2024}.
    
    \item Using the second-order invariants, we derive an explicit expression for the Ricci curvature in terms of the intrinsic torsion of the $\SU(2)$-structure, see Proposition \ref{prop: ricci in terms of torsion}. This is analogous to 
    Bryant's derivation of the Ricci curvature in terms of the intrinsic torsion of the $\mathrm{G}_2$-structure in \cite{Bryant06someremarks}*{\S 4.5.3}. Moreover, we also derive a corresponding expression for the self-dual Weyl curvature in Proposition \ref{proposition sd weyl curvature}.
    
    \item There are 15 $\SU(2)$-invariant functionals which are quadratic in the intrinsic torsion. Arguably the most natural one is the $L^2$-norm of the intrinsic torsion. In Theorem \ref{thm: short-time existence} we prove short time existence and uniqueness of a 2-parameter family of $\SU(2)$-flow which includes its negative gradient flow and the coupled 
    Ricci harmonic $\SU(2)$-flow. These provide natural $\SU(2)$-flows  to study to find Einstein metrics, see also the recent article \cite{Kirill2024} for another such functional. 
    
    \item In \S \ref{section: more general flows} we show how one can extend Theorem \ref{thm: short-time existence} to
    more general $\SU(2)$-flows. In particular, we show that one can incorporate self-dual Weyl curvature terms in flows of $\SU(2)$-structures which is rather different from flows of merely $\mathrm{O}(4)$-structures, i.e. Riemannian metrics, which by contrast can only involve the Ricci curvature. We also demonstrate short time existence for the negative gradient flow of the weighted energy functional of the intrinsic torsion for suitable parameters.

    \item In \S\ref{sec: general quadratic functionals} we compute the first variational formulae of the aforementioned 15 quadratic functionals (which by symmetry it suffices to consider only 5). Their negative gradient flows provide natural candidates of geometric flows to study, for instance, these can be used to evolve an almost complex structure by the negative gradient flow of the $L^2$-norm of its Nijenhuis tensor. Thus, these provide building blocks of natural $\SU(2)$ functionals and their corresponding gradient flows. We expect these to be useful for constructing geometric flows beyond those considered in this article.
    
\end{itemize}

\bigskip

\noindent\textbf{Acknowledgements:} 
The authors are grateful to Shubham Dwivedi, Daniel Fadel, Spiro Karigiannis, Kirill Krasnov, and Andr\'es Moreno for several interesting discussions. The authors would also like to thank the organisers and participants of the \emph{Special Holonomy and Geometric Structures on Complex Manifolds} workshop held at IMPA in March 2024, for their feedback on this work.

UF was funded by the S\~ao Paulo Research Foundation (Fapesp) [2021/07249-0] from November 2021 to January 2024, in a postdoctoral grant subordinate to the thematic project \emph{Gauge theory and algebraic geometry} [2018/21391-1], led by Marcos Jardim.
HSE was supported by the Fapesp \mbox{[2021/04065-6]} \emph{BRIDGES collaboration} and a Brazilian National Council for Scientific and Technological Development (CNPq) \mbox{[311128/2020-3]} Productivity Grant level 1D.

\subsection{Notation and conventions}
\label{sec: Notation and conventions}

We define the \emph{diamond operator} $\diamond$, induced by the infinitesimal action of an endomorphism $v\otimes \al \in \mathfrak{gl}(4,\R)=\R^4 \otimes (\R^4)^*$ on $k$-tensors, by
    \begin{equation}
    (v\otimes \al) \diamond (\be_1 \otimes...\otimes \be_k) :=\frac{1}{2}\big(\be_1(v) \al \otimes \be_2 \otimes...\otimes \be_k+\cdots +\be_k(v) \be_1\otimes ... \otimes \be_{k-1}\otimes \al\big). \label{definitionofdiamondtensors}
    \end{equation}
This differs from the convention in \cite{Fadel2022}, see Remark \ref{remark:differentdiamondaction} below. Given $k$-forms $\alpha,\beta$, we define their inner product by
\[
g(\al,\be) := \frac{1}{k!} g^{i_1j_1}...g^{i_kj_k}\al_{i_1...i_k}\be_{j_1...j_k}.
\]
With this convention,  $|dx_{i_1...i_k}|^2=1$ with respect to the Euclidean metric; this contrasts with  the conventions of \cite{Fadel2022}, where  the factor of $k!$ above is absent. On the other hand,  for tensors we shall use the standard inner product. 
As a consequence of these differences in convention, the reader should be careful when comparing formulae in this article with other related papers in the literature, such as  \cites{Karigiannis2009, Fadel2022, Dwivedi2023,DwivediLoubeauSaEarp2021}.

Denoting cyclic permutation by $\sim$, we shall often use the Levi-Civita symbols 
\[   
\epsilon_{ijk} := 
     \begin{cases}
       +1 &\quad\text{if $(i,j,k) \sim (1,2,3)$}, \\
      -1 &\quad\text{if $(i,j,k) \sim (1,3,2)$}, \\
       \ \ 0 &\quad\text{otherwise.}  
     \end{cases}
\]
For a given $1$-form $\alpha$ and a vector field $X$, we define $J(\alpha)(X):=\alpha(JX)$, so that on $1$-forms we have the relation 
\[
J_i \circ J_j = - \epsilon_{ijk} J_k.
\]
Also, given a smooth function $f$, we define $d^{c_i} f := J_i(df)$. These conventions might differ from others in the literature by a sign. 

For simplicity, we shall often identify the $\SU(2)$-modules $\otimes^k \R^4 \otimes^l (\R^4)^*$ (as well as their symmetrisations and skew-symmetrisations) with their associated vector bundles $\otimes^k TM \otimes^l T^*M$ and also with their space of sections; this should not cause any confusion for our purposes.
We shall write $\lot(x)$ to mean lower-order differential terms in $x$.

\section{Background on \texorpdfstring{$\SU(2)$}{}-structures}
\subsection{Preliminaries 
}
\label{sec: preliminary}

The main purpose of this section is to gather some algebraic facts about $\SU(2)$-structures that we shall use throughout this article and to fix the notation. We begin by recalling the basic definition.

\begin{definition}
\label{def:  sp1-structure 1}
    An $\SU(2)$-structure on a $4$-manifold $M$ is determined by a triple of $2$-forms $\boldsymbol{\om}=(\om_1,\om_2,\om_3)$ which at each point can be expressed as 
\begin{align*}
    \om_1 &= dx_{12}+dx_{34},\\
    \om_2 &= dx_{13}+dx_{42},\\
    \om_3 &= dx_{14}+dx_{23},
\end{align*}
    for some local coordinates $\{x_i\}$. 
\end{definition}
The triple $ \om_1,\om_2,\om_3$ in turn defines a Riemannian metric $g$ and a triple of almost complex structures $J_1,J_2,J_3$ on $M^4$ as follows: 
\begin{equation}
    J_iX \ip \om_j := X \ip \om_k,
    \qforq (i,j,k)\sim (1,2,3)
    \qandq
    X\in \sX(M).
    \label{complex structure from omi}
\end{equation}
The metric $g$ is then given by
\begin{equation}
    \om_i(\cdot,\cdot)=g(J_i\cdot,\cdot).
    \label{symplecticformtocomplexstructure}
\end{equation} 
We also define a volume form on $M$ by $\vol:=\frac{1}{2}\om_1 \w \om_1$ and the Hodge star operator $*$ by
\begin{equation}
    \alpha \w *\beta = g(\alpha,\beta) \vol,
    \qforq
    \alpha,\beta \in \Lm^k(M)
\end{equation}
In particular, note that $*$ and $J_i$ commute on differential forms. 
It is also worth pointing out there are other equivalent definitions of $\SU(2)$-structures in the literature  cf. \cites{ContiSalamon, Fine2018}. 
Probably the most common one is:
\begin{definition}[alternative]
    An $\SU(2)$-structure on a $4$-manifold $M$ is given by a triple of non-degenerate $2$-forms $\om_1, \om_2, \om_3$ satisfying the conditions
\[\om_i \w \om_j = \delta_{ij} \om_1 \w \om_1,\]
and that $\om_3(X,Y)\geq 0$, for every pair of vector fields $X,Y$ such that $X\ip \om_1=Y\ip \om_2$. Here \emph{non-degenerate} means that $\om_i \w \om_i$ is nowhere-vanishing and the second condition corresponds to an orientation of the triple.
\end{definition}
Note that one could equivalently work with the data $(g,J_1,J_2,J_3)$ instead of $(\om_1,\om_2,\om_3)$, but the latter description fits better with the approach in the $\mathrm{G}_2$ and $\Spin(7)$ cases, as we shall see later on. We shall use the following frequently in computations:
\begin{lemma}
\label{lemma:relations}
    Given arbitrary $1$-forms $\alpha, \beta$ and $2$-form $\gamma$, the following identities hold:
    \begin{enumerate}
        \item $\alpha \w \om_i = -J_k(\alpha) \w \om_j = J_j(\alpha) \w \om_k$, \quad for $(i,j,k)\sim (1,2,3)$.
        \item $*(\alpha \w \om_i)=-J_i(\alpha)$.
        \item $\gamma \w \om_i =g(\gamma,\om_i) \vol$.
        \item $g(\alpha,J_i\beta)=g(\alpha\w \beta,\om_i)$.
    \end{enumerate}
\end{lemma}

Given an oriented Riemannian $4$-manifold $(M^4,g,\vol)$, there is a natural orthogonal splitting of $2$-forms into Hodge self-dual and anti-self-dual forms:
\begin{equation}
\Lm^2(M)=\Lm^2_+(M) \oplus \Lm^2_-(M) \cong \mathfrak{su}(2)^+ \oplus \mathfrak{su}(2)^-.
\end{equation}
Geometrically the data $(g,\vol)$ correspond to a structure group reduction of the frame bundle from $\GL(4,\R)$ to $\SO(4)$. The existence of a global orthonormal basis of $\Lm^2_+(M)$ allows one to further reduce the structure group from $\SO(4)$ to $\SU(2):=\SU(2)^-$. Thus, as $\SU(2)$-modules, we have the irreducible splitting
\begin{equation}
    \Lm^2(M) = \langle \om_1, \om_2 ,\om_3 \rangle \oplus \Lm^2_-(M),
\end{equation}
where $\SU(2)$ acts trivially on each $\om_i$ and $\Lm^2_-(M)\cong \R^3$ corresponds to its adjoint representation. We define the projection map $\pi^2_-:\Lm^2(M) \to \Lm^2_-(M)$ by 
\begin{equation}
\pi^2_-(\al) := \frac{1}{2}(\alpha - *\alpha).
\end{equation}
As $\SU(2)$-modules, the spaces $\Lm^1(M)$ and $\Lm^3(M)$ are both irreducible and correspond to the standard representation of $\Sp(1)\cong\SU(2)$ on $\mathbb{H}\cong \mathbb{C}^2$.

Next we consider the space of symmetric $2$-tensors $\Sigma^2(M)$. As an $\SO(4)$-module, we have that 
\begin{equation}
\Sigma^2(M) =  \langle g \rangle \oplus \Sigma^2_{0}(M), 
\end{equation}
where $\Sigma^2_0(M)$ denotes the space of traceless symmetric $2$-tensors. Under the action of $\SU(2)$, we can further decompose $\Sigma^2_0(M)$ into three irreducible representations $\Sigma^2_{0,1}(M)$, $\Sigma^2_{0,2}(M)$ and $\Sigma^2_{0,3}(M)$ -- all isomorphic to the adjoint representation $\mathfrak{su}(2)\cong \R^3$. Explicitly, in terms of our conventions in \S\ref{sec: Notation and conventions},
\begin{equation}
    \Sigma^2_{0,i}(M)
    =\{\al \in \Sigma^2(M) \mid J_i(\al)
    =\al,\  J_j(\al)=J_k(\al)=-\al\},
    \qforq (i,j,k)\sim (1,2,3).
    \label{def: S_{0,i}}
\end{equation}
Thus a symmetric $2$-tensor $\al \odot \be := \al \otimes \be + \be \otimes \al \in \Sigma^2(M) = \langle g \rangle \oplus \Sigma^2_{0,1}(M) \oplus \Sigma^2_{0,2}(M) \oplus \Sigma^2_{0,3}(M)$ decomposes into four irreducible components:
\begin{align*}
    \al \odot \be = \ &+\frac{1}{4}(\al \odot \be+J_1\al \odot J_1\be+J_2\al \odot J_2\be+J_3\al \odot J_3\be)\\
    &+\frac{1}{4}(\al \odot \be+J_1\al \odot J_1\be-J_2\al \odot J_2\be-J_3\al \odot J_3\be)\\
    &+\frac{1}{4}(\al \odot \be-J_1\al \odot J_1\be+J_2\al \odot J_2\be-J_3\al \odot J_3\be)\\
    &+\frac{1}{4}(\al \odot \be-J_1\al \odot J_1\be-J_2\al \odot J_2\be+J_3\al \odot J_3\be).
\end{align*}
Let us define the projection maps $S^2_{0,i}$, given on a simple $2$-tensor $\alpha \otimes \beta$ by
\begin{gather*}
    S^2_{0,1}: T^*(M) \otimes T^*(M) \to \Sigma^2_{0,1}(M)\\
    S^2_{0,1}(\alpha \otimes \beta) 
    :=\frac{1}{8}(\al \odot \be+J_1\al \odot J_1\be-J_2\al \odot J_2\be-J_3\al \odot J_3\be),
\end{gather*}
and likewise for $S^2_{0,2}(\alpha \otimes \beta)$ and $S^2_{0,3}(\alpha \otimes \beta)$. 

As mentioned above, the spaces $\Lm^2_-(M)$ and $\Sigma^2_{0,i}(M)$ are all isomorphic as $\SU(2)$-modules. In what follows we will need this identification explicitly, and
to this end we define linear maps $P_i$, given on a simple $2$-form $\al=\al_1 \w \al_2$ by
\begin{gather*}
    P_i:\Lm^2_-(M) \to \Sigma^2_{0,i}(M),\\
    P_i(\al):= \alpha_1 \otimes (\alpha_2^\sharp \ip \om_i) -\alpha_2 \otimes (\alpha_1^\sharp \ip \om_i).
\end{gather*}
The above maps are manifestly $\SU(2)$-equivariant, hence by Schur's lemma each $P_i$ is either an isomorphism or the zero map. One easily checks on the Euclidean model that $P_1(dx_{12}-dx_{34})= -dx_1^2-dx_2^2+dx_3^2+dx_4^2$, and thus we deduce that $P_1$ is indeed an isomorphism; similarly for $P_2$ and $P_3$. 
\begin{lemma}
\label{lemma: identifications}
    Given $\alpha\in \Lm^2_-(M)$ and $\beta \in \Sigma^2_{0,i}(M)$, the following identities hold:
    \begin{enumerate}
        \item $P_i(\alpha)(\cdot,\cdot) = \alpha (J_i\cdot ,\cdot)$.
        \item $P_i(\alpha) \diamond \om_j = - \delta_{ij} \alpha$.
        \item $P_i(\beta \diamond \om_i) = - \beta.$
    \end{enumerate}
\end{lemma}
\begin{proof}
    Identity 1. follows from the fact that $X \ip \om_i  = g(J_iX,\cdot) = (J_iX)^\flat.$ The proofs of 2. and 3. are simple applications of Schur's lemma. Since $P_i$ and $\diamond \om_i$ are both $\SU(2)$-equivariant maps, we know that $P_i(\alpha) \diamond \om_j = \lm \alpha$ for some constant $\lm$, hence we can find  $\lambda$ in each case by checking on a simple example. 
\end{proof}
Since the triple $(\om_1,\om_2,\om_3)$ completely determines $g,\vol,J_1,J_2,J_3$, it follows that the evolutions of latter tensors are determined by (\ref{eq:evolutionom1})-(\ref{eq:evolutionom3}). Indeed, explicitly:
\begin{corollary}
    Under a variation of the $\SU(2)$-structure given by (\ref{eq:evolutionom1})-(\ref{eq:evolutionom3}), the underlying metric, volume form and almost complex structures evolve by:
\begin{align}
    \frac{d}{dt} g 
    &= f_0 g + B_{0,1} + B_{0,2} + B_{0,3}
    \label{eq:evolutiong},\\
    \frac{d}{dt} \vol 
    &= 2 f_0 \vol
    \label{eq:evolutionvol},\\
    \frac{d}{dt}(J_i) 
    &= f_k J_j - f_j J_k - B_{0,j}^i - B_{0,k}^i,
    \label{eq:evolutionofJi2}
\end{align}
 where $(i,j,k)\sim (1,2,3)$ and the endomorphisms $B_{0,i}^{j}$ are defined by $B_{0,i}(J_j\cdot,\cdot)=:g(B_{0,i}^{j}\cdot ,\cdot)$.
\end{corollary}
\begin{proof}
    As (\ref{eq:evolutiong}) and (\ref{eq:evolutionvol}) are straightforward, we only prove (\ref{eq:evolutionofJi2}).
    It suffices to consider the case $(i,j,k)=(1,2,3)$, since the same argument holds for the other cyclic permutations of indices.
    Differentiating  $g(J_1\cdot,\cdot)=\om_1(\cdot,\cdot)$ and using (\ref{eq:evolutionom1}), we get
    \begin{equation*}
         g(\frac{d}{dt}(J_1)\cdot,\cdot) = f_0 \om_1 + B_{0,1} \diamond \om_1 + f_3 \om_2 - f_2 \om_3 - B(J_1 \cdot,\cdot).
    \end{equation*}
    Observing that 
    $$B(J_1 \cdot,\cdot) = f_0 \om_1 + \sum_{i=1}^3  B_{0,i}(J_1\cdot ,\cdot),
    $$
    we know moreover from Lemma \ref{lemma: identifications} that $(B_{0,1} \diamond \om_1)(\cdot,\cdot) = B_{0,1}(J_1\cdot ,\cdot)$, and the result follows. 
\end{proof}
It is clear from (\ref{eq:evolutionofJi2}) that the almost complex structures are invariant under conformal rescalings, which correspond to the function $f_0$. We also see that the variations $f_i$ and $B_{0,i}$ leave the $J_i$ unchanged.

Let $\nabla$ denote the Levi-Civita connection of $g$. Given an arbitrary $1$-form $\alpha$, we can decompose the $2$-tensor $\nabla \alpha$ in terms of its  $\SU(2)$-irreducible components as
\begin{equation}
    \nabla \alpha 
    = -\frac{1}{4} (\delta \alpha) g + \sum_{i=1}^3 S^2_{0,i}(\nabla \alpha) + \frac{1}{4} \sum_{i=1}^3 g(d\alpha, \om_i) \om_i + \frac{1}{2} \pi^2_-(d\alpha), 
\label{eq: nabla of 1 form}
\end{equation}
where $\delta \alpha = - * d * \alpha $ denotes the codifferential. Appealing to Lemma \ref{lemma: identifications} and using the definition of $\Sigma^2_{0,i}(M)$,
we have the following identifications:
\begin{align}
    S^{2}_{0,1}(\nabla \alpha) (J_1, \cdot) 
    &= S^{2}_{0,1}(\nabla \alpha) \diamond \om_1 = 
    S^{2}_{0,1}(\nabla \alpha) (J_3, J_2), 
\label{eq: identification1}
\\
    S^{2}_{0,1}(\nabla \alpha) (J_2, \cdot) 
    &= 
    -P_3(S^2_{0,1}(\nabla \alpha) \diamond \om_1) = S^{2}_{0,1}(\nabla \alpha) (J_3, J_1),
\label{eq: identification2}
\\
    S^{2}_{0,1}(\nabla \alpha) (J_3, \cdot) 
    &= 
    +P_2(S^2_{0,1}(\nabla \alpha) \diamond \om_1) =
    -S^{2}_{0,1}(\nabla \alpha) (J_1, J_2).
\label{eq: identification3}
\end{align}
The first expression identifies $S^{2}_{0,1}(\nabla \alpha) \in \Sigma^2_{0,1}(M)$ with a $2$-form in $\Lm^2_-(M)$, the second with a symmetric $2$-tensor in $\Sigma^2_{0,3}(M)$ and the third with a symmetric $2$-tensor in $\Sigma^2_{0,2}(M)$; analogous statements hold for the other $S^{2}_{0,i}$, by cyclic permutation. For instance, from (\ref{eq: identification1}) we also deduce that
\[
S^{2}_{0,2}(\nabla \alpha) (J_2, ) = 
S^{2}_{0,2}(\nabla \alpha) \diamond \om_2 =
S^{2}_{0,2}(\nabla \alpha) (J_1, J_3) \in \Lm^2_-(M).
\]
The identifications (\ref{eq: identification1})-(\ref{eq: identification3}) will be particularly useful in \S\ref{section: energy functionals} and \S\ref{section: second order}. 
Lastly we recall the following simple lemma, certainly known to most readers, that we shall use frequently in \S\ref{section: flows} and as such we also give a proof:
\begin{lemma}
\label{lemma: lie derivative of om}
    Given a vector field $X$ and a $2$-form $\om$ on $M$, we have that
    \begin{equation}
     \mathcal{L}_X \om = \nabla_X \om + (\mathcal{L}_X g - dX^\flat) \diamond \om. \label{eq: lie derivative of om}
    \end{equation}
\end{lemma}
\begin{proof}
    Given vector fields $X,Y,Z\in\sX(M)$, from the relation
    $\mathcal{L}_X (\om(Y,Z)) = \nabla_X (\om(Y,Z))$
    we derive
\[
(\mathcal{L}_X \om)(Y,Z) = (\nabla_X \om)(Y,Z) + \om (\nabla_Y X, Z) + \om(Y, \nabla_Z X).
\]
    i.e. $(\mathcal{L}_X\om)_{ij}= X^k \nabla_k \om_{ij} + \nabla_i X^k \om_{kj}+ \nabla_j X^k \om_{ik}$. On the other hand, we know that $$(\mathcal{L}_Xg)_{ij}= \nabla_i X^k g_{kj}+ \nabla_j X^k g_{ki}.$$ 
    Acting with the latter on a $2$-form $\omega$ under $\diamond$, as defined in \S\ref{sec: Notation and conventions}, we get
\[
\big((\mathcal{L}_X g )\diamond \om\big)_{ij} = \frac{1}{2}( 
\nabla_a X^k g_{ki}g^{ab}\om_{bj} + 
\nabla_i X^b \om_{bj} + \nabla_a X^k g_{kj} g^{ab} \om_{ib} + \nabla_j X^b \om_{ib}
),
\]
and the result follows from the fact that 
$\nabla X^\flat = \frac{1}{2}(\mathcal{L}_X g + dX^\flat)$.
\end{proof}
One way to interpret the above formula, say when $\om=\om_1$, is that to highest order in $X$, the anti-self-dual part and the $\om_1$-component of $\mathcal{L}_X \om_1$ are determined by the deformation of the metric, i.e. by $\mathcal{L}_X g$, while the isometric deformation is determined by $dX^\flat$. Observe that the lower-order term $\nabla_X \om_1$ in $X$ lies in the linear span $\langle \om_2, \om_3\rangle$. This interpretation of (\ref{eq: lie derivative of om}) will be useful when considering DeTurck's trick to prove parabolicity, in \S\ref{section: flows}. Another immediate application of the above lemma is the following:
\begin{corollary}
    If $X\in\sX(M)$ is a Killing vector field, then $$\mathcal{L}_X \om_i \in \langle \om_j , \om_k \rangle,
    \qforq (i,j,k)\sim (1,2,3).
    $$
    Conversely, if $\mathcal{L}_X \om_1=\mathcal{L}_X \om_2=\mathcal{L}_X \om_3$, then $X$ is a Killing vector field.
\end{corollary}
Having covered the algebraic preliminaries of $\SU(2)$-structures, we next consider their first-order invariants, namely, the intrinsic torsion.

\subsection{Intrinsic torsion of \texorpdfstring{$\SU(2)$}{}-structures}
\label{sec: intrinsic torsion}
 
We begin by considering the more general set up, as described in \cite{Salamon1989}*{Chapter 2}. Given an $H$-structure on the manifold $M^n$, such that $H \subset \SO(n)$, we have a natural inclusion $\mathfrak{h}\subset \mathfrak{so}(n) \cong \Lm^2(M)$ and hence an orthogonal splitting $\Lm^2(M) = \mathfrak{h}\oplus \mathfrak{h}^\perp$. The intrinsic torsion of the $H$-structure is determined by a tensor $T \in \Lm^1(M) \otimes \mathfrak{h}^\perp$ which measures the failure of the holonomy algebra of the Levi-Civita connection $\nabla$ to reduce from $\so(n)$ to $\mathfrak{h}$. 
If $H$ can be defined as the stabiliser of a (multi-)tensor $\xi$, then $T$ can be identified with $\nabla \xi$. 

In our setup, we have $H=\SU(2)$, with $\mathfrak{h}=\mathfrak{su}(2)^-$ and $\mathfrak{h}^\perp=\mathfrak{su}(2)^+$, and the stabilised multi-tensor is $\xi=(\om_1,\om_2,\om_3)$. Note that $\mathfrak{h}^\perp$ is not in general a Lie subalgebra, this occurs in dimension $4$ because  $\SO(4)\cong (\SU(2)\times \SU(2))/\mathbb{Z}_2$.
Thus, the intrinsic torsion of an $\SU(2)$-structure is given by a tensor 
\begin{equation}
T \in \Lm^1(M) \otimes \mathfrak{h}^\perp \cong \Lm^1(M) \otimes \Lm^2_+(M).\label{eq:splittingofT}
\end{equation}
To describe the latter explicitly, 
we first note that $\diamond$ induces an action of $\Lm^2 \cong \mathfrak{so}(4)$ on itself. 
In particular:
\begin{lemma}
\label{lemma: action of diamond on om_i}
    The operator $\diamond$ acts on $\Lm^2_+(M) = \langle\om_1,\om_2,\om_3\rangle$ by
\begin{equation*}
    \om_i \diamond \om_j =- \om_j \diamond \om_i = \om_k,
    \qforq 
    (i,j,k)\sim (1,2,3),
\end{equation*}
    and $\om_i \diamond \om_i = 0$.
\end{lemma}
The above lemma is not really surprising, since the operator $\diamond$ corresponds (up to a constant factor) to the Lie bracket on $\mathfrak{su}(2)^+\subset \Lm^2$, and hence Lemma \ref{lemma: action of diamond on om_i} identifies  $\langle \om_1,\om_2,\om_3\rangle$ with the algebra of imaginary quaternions $\langle \text{i, j, k} \rangle$.
More generally, we have an action of $\mathfrak{gl}(4,\R)=\R^4 \otimes (\R^4)^*$ on tensors given by (\ref{definitionofdiamondtensors}). 
In local coordinates, given $A=A^i_j \partial_{x_i}\otimes dx_{j}$ and a $2$-tensor $\al=\al_{ij}dx_{i}\otimes dx_{j}$, we have  $$(A\diamond \al)_{ij}=\frac{1}{2}(A^{b}_i \al_{bj}+A^{b}_j\al_{ib})
.$$
Using the metric, we have the identification $\mathfrak{gl}(4,\R) \cong (\R^4)^* \otimes (\R^4)^* = \langle g\rangle\oplus \Sigma^2_0(M)\oplus \Lm^2(M)$, and thus shall also view the $\diamond$ operator as an action of $2$-tensors, i.e for $A=A_{ij}dx_i \otimes dx_j$,
$$(A\diamond \al)_{ij}
=\frac{1}{2}(A_{ai}g^{ab}\al_{bj}+A_{aj}g^{ab}\al_{ib}).
$$
With this convention, notice that $g \diamond \om_i = \om_i$. 

\begin{remark}
\label{remark:differentdiamondaction}
    The infinitesimal action $\diamond$ defined here differs from the one given in \cite{Fadel2022} by a factor of one half;  furthermore, the latter  authors view $\mathfrak{gl}(4,\R)$ as $(\R^4)^* \otimes \R^4$, and as a result there is a $-1/2$ factor difference for the diamond action of $2$-forms, and only a factor of $+1/2$ for action of symmetric 2-tensors. In local coordinates, $A^a_jg_{ai}=A_{ji}$ in the convention of \cite{Fadel2022}, whereas here $A^a_jg_{ai}=A_{ij}$; this does not matter if $A$ is a symmetric $2$-tensor, but there is a sign difference  for $2$-forms. 
    So the reader should bear these differences in mind if comparing results; fortunately, we shall not need to rely directly on formulae computed therein for this article. 
\end{remark}
Denoting the $\SU(2)$-triple by the row vector $\boldsymbol{\om}=(\om_1,\om_2,\om_3)$, from \cite{LoubeauSaEarp2019}*{Lemma 1} we know that 
\begin{equation}
	\nabla_X \boldsymbol{\om}= T(X) \diamond \boldsymbol{\om},
    \qforq X\in\sX(M).
 \label{key}
\end{equation}
Under the identification (\ref{eq:splittingofT}), we can view $T(X)$ as a self-dual $2$-form
\begin{equation}
    T= a_1 \otimes \om_1+a_2 \otimes \om_2+a_3 \otimes \om_3,
    \label{intrinsictorsionsu2}
\end{equation}
where $a_1,a_2,a_3$ are \emph{torsion $1$-forms}, which  can be explicitly described as follows:
\begin{proposition} \label{torsionproposition}
    The torsion tensor $T$ can be expressed as 
\begin{equation}
    T =\frac{1}{2}\big(g(\nabla \om_2,\om_3)\otimes \om_1 +g(\nabla \om_3,\om_1)\otimes \om_2+g(\nabla \om_1,\om_2)\otimes \om_3\big).
    \label{eq:torsion in terms of a}
\end{equation}
    Equivalently, the torsion forms \eqref{intrinsictorsionsu2} are given by
\begin{align}
        a_1 &= -\frac{1}{2}\big(*d\om_1+J_2(*d\om_3)-J_3(*d\om_2)\big),\label{eq:a1}\\
        a_2 &= -\frac{1}{2}\big(*d\om_2-J_1(*d\om_3)+J_3(*d\om_1)\big),\label{eq:a2}\\
        a_3 &= -\frac{1}{2}\big(*d\om_3+J_1(*d\om_2)-J_2(*d\om_1)\big).\label{eq:a3}
    \end{align}    
    In particular, 
    $$\nabla \om_i=0  \Leftrightarrow 
    a_j=a_k=0, \qforq (i,j,k)\sim (1,2,3).$$
    Moreover, $T=0 \Leftrightarrow \nabla \boldsymbol{\om}=0 \Leftrightarrow d\boldsymbol{\om}=0$.
\end{proposition}
\begin{proof}
    Expanding (\ref{key}) with Lemma \ref{lemma: action of diamond on om_i}, we get
\begin{align}
    \nabla_X\boldsymbol{\om}^t 
    &= (-a_2(X)\ \! \om_3+a_3(X)\ \! \om_2, a_1(X)\ \!\om_3-a_3(X)\ \!\om_1, -a_1(X)\ \!\om_2+a_2(X)\ \! \om_1)^t\nonumber\\
    &=  \begin{pmatrix}
	0      & +a_3(X) & -a_2(X)  \\
	-a_3(X)       & 0 & +a_1(X) \\
	+a_2(X)    & -a_1(X) & 0 
    \end{pmatrix}
    \begin{pmatrix}
        \om_1   \\  \om_2 \\ \om_3 
    \end{pmatrix}. \label{intrinsictorsionsu2-2}
\end{align}
    It follows that
\begin{gather}
	 a_i(X) = \frac{1}{2}g(\nabla_X \om_j,\om_k) = -\frac{1}{2} 
      g(\nabla_X \om_k,\om_j), 
      \qforq
      (i,j,k)\sim (1,2,3),
\end{gather}
    and this proves (\ref{eq:torsion in terms of a}). 
    
    From (\ref{intrinsictorsionsu2-2}), we also have  
  \begin{gather}
        d\om_i=a_k\w \om_j-a_j\w \om_k.
        \label{eq: exterior derivative of om1}
    \end{gather}
    Using the above together with Lemma \ref{lemma:relations}, a short calculation yields (\ref{eq:a1})-(\ref{eq:a3}).
\end{proof}

Among $\SU(2)$-structures with non-vanishing torsion, one important subclass is that of \emph{hyper-Hermitian} structures i.e when the almost complex structures $J_1,J_2,J_3$ are all integrable. In our setup, 
we can characterise this condition as follows:
\begin{proposition}
\label{hypercomplexproposition}
    For each $(i,j,k)\sim (1,2,3)$, the almost complex structure $J_i$ is integrable if, and only if,
    \[J_i(a_k)=a_j.\]
    In particular, $\{ J_1, J_2, J_3\}$ 
    defines a hypercomplex structure if, and only if,
    $a_2=-J_3(a_1)$ and $a_3=J_2(a_1)$.
\end{proposition}
\begin{proof}
    The almost complex structure $J_1$ is integrable if, and only if, $d(\om_2+i\om_3)\in \Lm^{3,0}\oplus \Lm^{2,1}$, in terms of the $J_1$-type decomposition. It is easy to see this is equivalent to $a_3-ia_2$ being a $(1,0)$-form ($J_1$-type) i.e. $J_1(a_3)=a_2$. The same argument holds for $J_2$ and $J_3$.
\end{proof}

Next we want to relate the derivative of the intrinsic torsion with the Riemann curvature tensor. To this end we recall the alternative characterisation of the intrinsic torsion, found eg. in \cite{Bryant06someremarks}*{\S 4.2} :
\begin{equation}
    g(\nabla_X Y - \nabla^{\mathfrak{h}}_X Y,Z) =  \frac{1}{2} T(X)(Y,Z) \label{equ: torsion plus canonical connection}
\end{equation}
where $\nabla$ denotes the Levi-Civita connection and $\nabla^{\mathfrak{h}}$ denotes the natural $\SU(2)$-connection obtained by projecting $\nabla$ onto $\mathfrak{h}=\mathfrak{su}(2)^-$. In particular,  $\nabla^{\mathfrak{h}} \om_i = 0$, for $i=1,2,3$, and $T=0$ precisely when $\nabla=\nabla^{\mathfrak{h}}$. The connection $\nabla^{\mathfrak{h}}$ is in general not torsion-free, and instead its torsion can be identified with $T$.

Denoting the curvature forms of $\nabla$ and $\nabla^{\mathfrak{h}}$ by $F_{\nabla}=\mathrm{Rm}$ and $F_{\nabla^{\mathfrak{h}}}$, respectively, and writing $d_{\nabla^{\mathfrak{h}}}$ for the exterior covariant derivative, we have 
\begin{equation}
    F_{\nabla}  = F_{\nabla^{\mathfrak{h}}} + \frac{1}{2} d_{\nabla^{\mathfrak{h}}} T + \frac{1}{8}[T\w T],\label{eq:curvaturetorsion}
\end{equation}
where $T$ is viewed as an $\mathfrak{h}^\perp=\mathfrak{sp}(1)^+$-valued $1$-form 
and $[\cdot \w \cdot ]$ corresponds to simultaneously wedging the differential forms and taking the Lie bracket, cf. \cite{Donaldson1990}*{\S2.1.2}. 
Note that the relations (\ref{equ: torsion plus canonical connection}) and (\ref{eq:curvaturetorsion}) hold more generally for any $H$-structure with $H\subset \SO(n)$, and $\nabla^{\mathfrak{h}}$ is again understood to be the projection of the Levi-Civita connection onto $\mathfrak{h} \subset \so(n)$ cf. \cites{Bryant06someremarks, Salamon1989}. 

A unique feature of our setup, due to the splitting $\mathfrak{so}(4)=\mathfrak{sp}(1)^+\oplus \mathfrak{sp}(1)^-$, is that ${\nabla^{\mathfrak{h}}}$ induces the flat connection on $\Lm^2_+(M)$, and hence $d_{\nabla^{\mathfrak{h}}} T =d T$. Moreover, just like the diffeomorphism-invariance of the Riemann curvature yields the Bianchi identity, cf. \cite{Chow2004}*{\S2.2}, the diffeomorphism-invariance of the intrinsic torsion of an $H$-structure also yields a Bianchi-type identity; see for instance \cite{Karigiannis2009}*{Theorem 4.2} for the $\mathrm{G}_2$ case, \cite{DwivediLoubeauSaEarp2021}*{Theorem 3.8} for the $\mathrm{Spin}(7)$ case, \cite{FowdarSaEarp2023}*{Proposition 2.5} for the $\Sp(2)\Sp(1)$ case 
and \cite{Fadel2022}*{Corollary 1.38} for the general case.
This identity expresses the $\Lm^2(M)\otimes \mathfrak{h}^{\perp} \subset \Lm^2(M) \otimes 
\Lm^2(M)$ component of the Riemann curvature in terms of derivatives of the intrinsic torsion, and it is obtained by the endomorphism action of  (\ref{eq:curvaturetorsion}) on the tensor $\xi$ stabilised by $H$, see also \S\ref{subsect: rep theory} below. 
\begin{corollary}[Bianchi-type identity]
\label{cor:bianchitypeidentity}
    The diffeomorphism-invariance of the intrinsic torsion $T$ is expressed by:
\begin{equation*}
    2  F_{\nabla} \diamond \boldsymbol{\om} =\big( 
    (da_1 + a_2 \w a_3)\otimes \om_1 + 
    (da_2 + a_3 \w a_1)\otimes \om_2 +
    (da_3 + a_1 \w a_2)\otimes \om_3 \big)\diamond \boldsymbol{\om}.
\end{equation*}    
\end{corollary}
\begin{proof}
    The claim follows directly from a simple computation using (\ref{eq:curvaturetorsion}), along with  $F_{\nabla^{\mathfrak{h}}}\diamond \boldsymbol{\om}=0$ and $d_{\nabla^{\mathfrak{h}}} T =d T$, as discussed above. 
\end{proof}
\begin{remark}
     Corollary \ref{cor:bianchitypeidentity} implies that the $\mathfrak{su}(2)^\perp$ component of the curvature form $F_{\nabla}$ is determined, up to lower-order terms, by $da_i$. Based on this observation, we can already deduce that, to highest order, the Ricci and self-dual Weyl curvature tensors can be expressed in terms of these second-order quantities -- we use on this in \S\ref{section: second order} to derive explicit formulae. 
\end{remark}

Next we characterise \emph{harmonic} $\SU(2)$-structures, in the sense of \cite{LoubeauSaEarp2019}. We recall that an $\SU(2)$-structure defined by the triple $\boldsymbol{\om}$ is called harmonic if $$\mathrm{div}(T):=g^{ij}\nabla_i T_{j,kl}=0$$ (this holds true for more general $H$-structures with $H\subset \SO(n)$, cf. \cite{LoubeauSaEarp2019}*{Lemma 2}). In our setup we can state the condition more explicitly  as follows:
\begin{proposition}
\label{prop: harmonic condition}
    The divergence of the intrinsic torsion $T$ can be expressed as
\begin{equation}
    \mathrm{div}(T) = \sum_{i=1}^3 \mathrm{div}(a_i)\ \! \om_i. \label{eq: blah1}
\end{equation}
Equivalently, in terms of $\nabla$ and the triple $\boldsymbol{\om}$ alone,
\begin{align*}
	\mathrm{div}(T) = \frac{1}{2}\Big(&\big(g(\nabla^*\nabla \om_2,\om_3)+g(\nabla \om_2,\nabla \om_3) \big)\ \om_1
	+\big( g(\nabla^*\nabla \om_3,\om_1) + g(\nabla \om_3,\nabla \om_1)\big)\ \om_2\ +\\ 
	&\big( g(\nabla^*\nabla \om_1,\om_2)+g(\nabla \om_1 ,\nabla \om_2) \big)\ \om_3 \Big),
\end{align*}
where $\nabla^*\nabla \om_i:=\tr(\nabla^2 \om_i)$.
Thus, $\mathrm{div}(T)=0$ if, and only if, the $1$-forms $a_i$ are all coclosed.
\end{proposition}
\begin{proof}
    Let $\{e_i\}$ denote a local $\SU(2)$-frame, then using (\ref{eq:torsion in terms of a}) we compute
\begin{align*}
	2\mathrm{div}(T) 
	=\ &g(\nabla^*\nabla \om_1,\om_2)\ \! \om_3+g(\nabla^*\nabla \om_3,\om_1)\ \! \om_2+g(\nabla^*\nabla \om_2,\om_3)\ \! \om_1 \\
	 &+ \sum_{i=1}^4 g(\nabla_{e_i} \om_1,\nabla_{e_i} \om_2)\ \! \om_3+g(\nabla_{e_i} \om_3,\nabla_{e_i} \om_1)\ \! \om_2+g(\nabla_{e_i} \om_2,\nabla_{e_i}\om_3)\ \! \om_1 \\
	&+ \sum_{i=1}^4 g(\nabla_{e_i} \om_1,\om_2)\ \! \nabla_{e_i}\om_3+g(\nabla_{e_i} \om_3,\om_1)\ \! \nabla_{e_i}\om_2+g(\nabla_{e_i} \om_2,\om_3)\ \! \nabla_{e_i}\om_1\\
	=\  &g(\nabla^*\nabla \om_1,\om_2)\ \! \om_3+g(\nabla^*\nabla \om_3,\om_1)\ \! \om_2+g(\nabla^*\nabla \om_2,\om_3)\ \! \om_1\\
	&- 2g(a_1,a_2)\ \! \om_3 -2 g(a_3,a_1) \ \!\om_2 - 2g(a_2,a_3) \ \!\om_1\\
	=\ & \big( g(\nabla^*\nabla \om_1,\om_2)+g(\nabla \om_1 ,\nabla \om_2) \big)\ \! \om_3\\ 
	&+\big(g(\nabla^*\nabla \om_2,\om_3)+g(\nabla \om_2,\nabla \om_3) \big)\ \! \om_1\\  &+
	\big( g(\nabla^*\nabla \om_3,\om_1) + g(\nabla \om_3,\nabla \om_1)\big)\ \! \om_2.
\end{align*}
The equivalent expression (\ref{eq: blah1}) can be deduced easily from the proof of Proposition \ref{prop:divTintermsofa} below.
\end{proof}
In \cite{Grigorian2017} Grigorian showed that harmonic $\mathrm{G}_2$-structures 
can be interpreted as $\mathrm{G}_2$-structures
in  `octonionic Coulomb gauge', where the gauge group corresponds to the octonions $\mathbb{O}$ and hence is non-associative. This interpretation is also valid in the $\SU(2)$ case, see the discussion in \S\ref{sec: quaternionic bundle}, but since the quaternions $\mathbb{H}$ are associative we instead get the usual notion of Coulomb gauge:
\begin{corollary}
\label{Levicivitacorollary}
    The $\SU(2)$-structure determined by $\langle \om_1,\om_2,\om_3\rangle$ is harmonic if, and only if, the induced Levi-Civita connection $\nabla$ on $\Lm^2_+(M)$ is in Coulomb gauge for the frame $\langle \om_1,\om_2,\om_3\rangle \subset\Lm^2_+(M)$. 
\end{corollary}
\begin{proof}
    The $\mathfrak{so}(3)$ matrix of a $1$-form $\mathbf{a}^t=-\mathbf{a}$, as defined in (\ref{intrinsictorsionsu2-2}), 
    simply corresponds to the induced Levi-Civita connection on $\Lm^2_+$, viewed here as a matrix of $1$-forms with respect to the frame $\langle \om_1,\om_2,\om_3\rangle$. Thus, the harmonic condition $\mathrm{div}(T)=0$ is equivalent to $\delta \mathbf{a}=0$.
\end{proof}
\begin{remark}
From the general theory of harmonic structures \cite{LoubeauSaEarp2019}, we already know that harmonic $\SU(2)$-structures are critical points of the Dirichlet energy functional
\begin{equation}
    E(\boldsymbol{\om}):=\int_M |T(\boldsymbol{\om})|^2 \vol =  2\int_M |\mathbf{a}|^2 \vol,
\end{equation}
whereby we only consider isometric variations of the $\SU(2)$-structure i.e. $g$ is kept fixed while the triple $\boldsymbol{\om}$ are allowed to vary as $g$-compatible $2$-forms. 
Thus, Corollary \ref{Levicivitacorollary} asserts that Coulomb gauge minimises (or rather is the critical gauge for) the $L^2$-norm of the connection $1$-form $\mathbf{a}$; this was already suggested in \cite{Donaldson1990}*{p. 56} motivated by the theory of Abelian connections. In \S\ref{section: energy functionals} we shall consider more general functionals, for which the metric is also allowed to vary.
\end{remark}

Since throughout this article we shall express everything in terms of the torsion $1$-forms $a_i$, we conclude this section by identifying several well-known classes of $\SU(2)$-structures in terms of intrinsic torsion: 
\begin{definition}
\label{definitionHKetal}
    In the above framework,  $(M^4,\om_1,\om_2,\om_3)$ is
    \begin{enumerate}
        \item hyperK\"ahler, if $a_1=a_2=a_3=0$.
        \item K\"ahler with respect to $\om_i$, if $a_j=a_k=0$, for $(i,j,k)\sim(1,2,3)$.
        \item Hermitian with respect to $J_i$, if $J_j(a_j)=J_k(a_k)$, for $(i,j,k)\sim(1,2,3)$.
        \item hyperHermitian, if $J_1a_1=J_2a_2=J_3a_3$.
        \item a harmonic $\SU(2)$-structure, if $\delta(a_1)=\delta(a_2)=\delta(a_3)=0$.
        \item a harmonic $\U(2)$-structure with respect to $(J_i,\om_i)$, if $\delta(a_j)=g(a_k,a_i)$ and $\delta(a_k)=-g(a_j,a_i)$, for $(i,j,k)\sim(1,2,3)$.
    \end{enumerate}
\end{definition}
Note that the definition of harmonic $\U(n)$-structures in the literature is usually formulated in terms of the almost complex structure $J$, namely as $[\nabla^*\nabla J,J]=0$, where  $\nabla^*\nabla J:=\tr(\nabla^2 J)$ cf. \cites{Wood1993,Wood1995}
and also compare with Proposition \ref{prop: harmonic condition} above. In Definition \ref{definitionHKetal}-(6) above we have 
instead expressed the harmonic condition in terms of the torsion $1$-forms.

\section{Variations and Functionals}
\label{section: energy functionals}

We will now compute the evolution of the intrinsic torsion under (\ref{eq:evolutionom1})-(\ref{eq:evolutionom3}), thence deriving the first variational formulae of all the quadratic functionals in the intrinsic torsion. The corresponding gradient flow provides natural candidates of $\SU(2)$-flows, among which we will have interest primary interest in the negative gradient flow of the Dirichlet energy \eqref{def:energyfunctionalfulltorsion}. Throughout this section we shall use the Einstein summation conventions i.e. we sum over repeated indices. We should mention that some of the computations here have already appeared in \cite{Fadel2022} but owing to the differences in conventions, see \S \ref{sec: Notation and conventions} and Remark \ref{remark:differentdiamondaction}, we find it safer to redo the computations.

\subsection{Variation of intrinsic torsion}
In this section we compute the evolution equation for the intrinsic torsion when the $\SU(2)$-structure evolves under  (\ref{eq:evolutionom1})-(\ref{eq:evolutionom3}). 

\begin{lemma}
\label{lemma:evolutionofchristoffel}
    Under the variation of the $\SU(2)$-structure given by (\ref{eq:evolutionom1})-(\ref{eq:evolutionom3}),
    \begin{enumerate}
        \item the Christoffel symbols $\Gamma_{ij}^k$ evolve by
        \begin{equation}
            \frac{d}{dt} (\Gamma_{ij}^k) = \frac{1}{2}g^{lk}(\nabla_jB_{il}+\nabla_iB_{jl}-\nabla_{l}B_{ij}); \label{eq:evolutionofchristoffel}
        \end{equation}
        \item the covariant derivative of an arbitrary $2$-tensor $\xi$ evolves by
        \begin{equation}
            \frac{d}{dt} (\nabla_l\xi_{ij}) = \nabla_l(\frac{d}{dt}\xi_{ij}) -\frac{1}{2}\xi_{bj}g^{ab}(\nabla_i B_{la}+ \nabla_l B_{ia} -\nabla_a B_{li})-\frac{1}{2}\xi_{ib}g^{ab}(\nabla_j B_{la}+ \nabla_l B_{ja} -\nabla_a B_{lj}).\label{eq:evolutionofxi}
        \end{equation}
    \end{enumerate}
    In particular,  $\boldsymbol{\om}=(\om_1,\om_2,\om_3)$ evolves by
        \begin{equation}
            \frac{d}{dt} (\nabla_l\boldsymbol{\om}) = A \diamond \nabla_l \boldsymbol{\om} + (\nabla_l C + \Lambda\nabla B_l) \diamond \boldsymbol{\om},\label{eq:evolutionombold}
        \end{equation}
    where $\Lambda\nabla B_l$ is the 2-form defined by $(\Lambda\nabla B_l)_{ij} := \nabla_{i}B_{lj}-\nabla_jB_{li}$.
\end{lemma}
\begin{proof}
    The evolution of the Christoffel symbols is a standard result, cf. \cite{Topping2006}*{Prop 2.3.1} or \cite{Chow2004}*{Lemma 3.2}. Since $\nabla_l\xi_{ij}=\partial_l(\xi_{ij})-\xi_{bj}\Gamma_{li}^b-\xi_{ib}\Gamma_{lj}^b$, we can easily deduce (\ref{eq:evolutionofxi}) from (\ref{eq:evolutionofchristoffel}). So it only remains to prove (\ref{eq:evolutionombold}).
    
    First, since $\frac{d}{dt}\boldsymbol{\om}=A \diamond \boldsymbol{\om}$ we have  
    \begin{align*}
        \nabla_l(\frac{d}{dt}\boldsymbol{\om}) &= A \diamond \nabla_l\boldsymbol{\om} + \nabla_l A \diamond \boldsymbol{\om} \\
        &= A \diamond \nabla_l\boldsymbol{\om} + \nabla_l B \diamond \boldsymbol{\om} + \nabla_l C \diamond \boldsymbol{\om} \\
        &= A \diamond \nabla_l\boldsymbol{\om} + \frac{1}{2}g^{ab}(\nabla_l B_{ai}  \boldsymbol{\om}_{bj}+\nabla_l B_{aj}  \boldsymbol{\om}_{ib})dx_{i}\otimes dx_j + \nabla_l C \diamond \boldsymbol{\om}.
    \end{align*}
    Here we used the fact that the operator $\diamond$ depends only on the metric and hence is preserved by $\nabla$. 
    Setting $\xi=\boldsymbol{\om}$ in (\ref{eq:evolutionofxi}) and using the above now yields
    \begin{align*}
            \frac{d}{dt} (\nabla_l \boldsymbol{\om}) &= \nabla_l(\frac{d}{dt}\boldsymbol{\om}) -\frac{1}{2}\boldsymbol{\om}_{bj}g^{ab}(\nabla_i B_{la}+ \nabla_l B_{ai} -\nabla_a B_{li})-\frac{1}{2}\boldsymbol{\om}_{ib}g^{ab}(\nabla_j B_{la}+ \nabla_l B_{aj} -\nabla_a B_{lj})\\
            &= A \diamond \nabla_l\boldsymbol{\om} + \nabla_l C \diamond \boldsymbol{\om} +\frac{1}{2}\big(( \nabla_a B_{li}-\nabla_i B_{la})g^{ab}\boldsymbol{\om}_{bj}+(\nabla_a B_{lj}-\nabla_j B_{la})g^{ab}\boldsymbol{\om}_{ib}\big)dx_i \otimes dx_j
        \end{align*}
    and this proves (\ref{eq:evolutionombold}).
\end{proof}
Since the intrinsic torsion $T$ is defined by (\ref{key}), we can now derive the evolution of $T$ using (\ref{eq:evolutionombold}). 
In what follows it will also be important to distinguish the endomorphism-valued $1$-form $T_{l,j}^i$ from the $3$-tensor $T_{l,ij}$; to emphasise this subtlety, we shall denote the associated endomorphism as indeed a matrix by $[T_l]$, i.e $$[T_l]_{j}^i=T_{l,aj}g^{ai}.$$
and likewise we denote by $[P]$ the endomorphism associated to a $2$-tensor $P$.
\begin{proposition}\label{propostion:evolutionoftorsionendomorphism}
    Under the variation of the $\SU(2)$-structure given by (\ref{eq:evolutionom1})-(\ref{eq:evolutionom3}), the endomorphism $[T_l]$ evolves by
    \begin{equation}
      (\frac{d}{dt}[T_l] - [\nabla_l C] - [\Lambda\nabla B_l] + \frac{1}{2}\big[[A],[T_l]\big]) \diamond \boldsymbol{\om}  =0,\label{eq:qqqq}
  \end{equation} 
  where $\big[[A],[T_l]\big]^i_{j}= A^i_a T^a_{l,j} - T^i_{l,a} A^a_j$ denotes the usual Lie bracket on matrices. 
In particular,  
 \begin{equation}
      g(\frac{d}{dt}[T_l] - [\nabla_l C] - [\Lambda\nabla B_l],[T_l] )=0,
      \label{eq:pppp}
  \end{equation}
  where $[\nabla_l C]$ and $[\Lambda\nabla B_l]$ are viewed as endomorphisms.     
\end{proposition}
\begin{proof}
Differentiating (\ref{key}),
\begin{align*}
     \frac{d}{dt}(\nabla \boldsymbol{\om}) &= \frac{d}{dt}([T]) \diamond \boldsymbol{\om} + [T] \diamond \frac{d}{dt}\boldsymbol{\om} \\
     &= \frac{d}{dt}([T]) \diamond \boldsymbol{\om} + [T] \diamond (A \diamond \boldsymbol{\om}).
\end{align*}
Note that if we had considered $T_l$ as a $2$-tensor, instead of the endomorphism $[T_l]$, we would have had to differentiate $\diamond$ as well, since it would depend on $g$.  
Using (\ref{eq:evolutionombold}), we can rewrite the above as
  \begin{equation}
      (\frac{d}{dt}([T_l]) - [\nabla_l C] - [\Lambda\nabla B_l]) \diamond \boldsymbol{\om} + [T_l] \diamond (A \diamond \boldsymbol{\om}) - A\diamond ([T_l] \diamond \boldsymbol{\om}) =0.\label{eq:evolutionTE}
  \end{equation}  
  Writing $[A]$ for the endomorphism associated to the $2$-tensor $A$, one can easily verify from the definition of $\diamond$ that $$[T_l] \diamond (A \diamond \boldsymbol{\om}) - A \diamond ([T_l] \diamond \boldsymbol{\om})= \frac{1}{2}\big[[A],[T_l]\big] \diamond \boldsymbol{\om},$$ which proves (\ref{eq:qqqq}).
  In particular, using the metric we see that  $\big[[B],[T_l]\big]$ corresponds to a symmetric $2$-tensor i.e. $\big[[B],[T_l]\big]_{ij}=\big[[B],[T_l]\big]_{ji}$, while $\big[[C],[T_l]\big]$ corresponds to a $2$-form. As a consequence,
  \begin{align*}
  g(\big[[A],[T_l]\big],[T_l]) &= g(\big[[C],[T_l]\big],[T_l])\\
  &= [C]^a_k[T_l]^k_b [T_l]^i_j g_{ia}g^{bj}
  - [T_l]^a_k[C]^k_b[T_l]^i_j g_{ia}g^{bj}\\
  &= [C]^a_k([T_l]^k_b [T_l]^i_j g_{ia}g^{bj}- 
  [T_l]^b_a[T_l]^i_j g_{ib}g^{kj})\\
  &= [C]^a_k([T_l]^{kj} [T_l]_{aj}- 
  [T_l]_{ia}[T_l]^{ik})\\
  &= 0 \qedhere
  \end{align*}
  \end{proof}
Finally we derive the evolution of each irreducible component of the intrinsic torsion, namely the $\{a_i\}$.
\begin{proposition}
\label{prop:evolutionofai}
    Under the variation of the $\SU(2)$-structure given (\ref{eq:evolutionom1})-(\ref{eq:evolutionom3}),
    the torsion $1$-forms $\{a_i\}$ evolve by
    \begin{equation}
    \frac{d}{dt} (a_i)_l = \nabla_l f_i + f_k (a_j)_l - f_j (a_k)_l +\frac{1}{2}g(\Lm \nabla B_l,\om_i),
    \qforq 
    (i,j,k)\sim(1,2,3),
    \label{eq:evolutionofai}
    \end{equation}
    where $(a_i)_l:=a_i(\partial_{x_l})$.
\end{proposition}
  \begin{proof}
      Differentiating $\nabla_l \om_1 =  (a_3)_l \om_2  -(a_2)_l \om_3$, we have 
      \begin{equation*}
          \frac{d}{dt}(\nabla_l \om_1) =  \frac{d}{dt}((a_3)_l) \om_2 -\frac{d}{dt}((a_2)_l) \om_3  + (a_3)_l A \diamond \om_2 -(a_2)_l A \diamond \om_3.
      \end{equation*}
      Now, from Lemma \ref{lemma:evolutionofchristoffel}, we know that
      \begin{equation*}
         \frac{d}{dt}(\nabla_l \om_1) = A \diamond \nabla_l \om_1 + (\nabla_l C + \Lm \nabla B_l) \diamond \om_1.
      \end{equation*}
     Combining both assertions we get
      \begin{equation}
          \frac{d}{dt}\big((a_3)_l\big) \om_2 -\frac{d}{dt}\big((a_2)_l\big) \om_3 = (\nabla_l C + \Lm \nabla B_l) \diamond \om_1. 
      \end{equation}
    Since $C= f_1 \om_1 + f_2 \om_2 + f_3 \om_3$, we also have 
\begin{equation*}
    \nabla_l C = \big(\nabla_lf_1 - f_2 (a_3)_l+f_3 (a_2)_l\big) \om_1  + 
    \big(\nabla_lf_2 - f_3 (a_1)_l+f_1 (a_3)_l\big) \om_2 +\big(\nabla_lf_3 - f_1 (a_2)_l+f_2 (a_1)_l\big) \om_3,
\end{equation*}
    and, in view of Lemma \ref{lemma: action of diamond on om_i}, we get
\begin{align*}
    \frac{d}{dt}\big((a_3)_l\big) \om_2 -\frac{d}{dt}\big((a_2)_l\big) \om_3 =&\   
    \big(\nabla_lf_3 - f_1 (a_2)_l+f_2 (a_1)_l\big) \om_2
    - \big(\nabla_lf_2 - f_3 (a_1)_l+f_1 (a_3)_l\big) \om_3 \\
    &+
    \Lm \nabla B_l \diamond \om_1. 
\end{align*}
    The fact that $\Lm \nabla B_l \diamond \om_1 = \frac{1}{2}g(\Lm \nabla B_l,\om_3) \om_2 - \frac{1}{2}g(\Lm \nabla B_l,\om_2) \om_3$ gives us the evolution of $(a_2)_l$ and $(a_3)_l$, and applying the same argument to $(a_1)_l$ yields the result.
\end{proof}

\subsection{Energy of the intrinsic torsion}

We are now in position to compute the first variation formula of the Dirichlet energy functional:
\begin{equation}
    E(\boldsymbol{\om})
    := \int_M |T|^2 \vol = 2\int_M \sum_{i=1}^3|a_i|^2 \vol.
    \label{def:energyfunctionalfulltorsion}
\end{equation}
We should point out that  (\ref{def:energyfunctionalfulltorsion}) corresponds to a special choice among
many other possible $\SU(2)$-functionals which are quadratic in the intrinsic torsion; these will be considered in \S\ref{sec: general quadratic functionals}. Our particular interest stems from the fact that its negative gradient flow provides a natural tool to find torsion-free $\SU(2)$-structures.

\begin{proposition}
\label{proposition:firstvariation}
    The first variation of the Dirichlet functional (\ref{def:energyfunctionalfulltorsion}) is given by
\begin{equation}
    \frac{d}{dt}E(\boldsymbol{\om})
    =-\int_M 2 g(\mathrm{div}(T),C)\vol -\int_M g\big( \mathrm{sym}(\mathrm{div}(T^t)) + T*T - \frac{1}{2}|T|^2g,B\big) \vol
\end{equation}
    where
\begin{gather*}
    \mathrm{div}(T)_{ij}
    :=g^{ab}\nabla_a T_{b,ij},
    \quad
    \mathrm{div}(T^t)_{ij}
    :=g^{ab}\nabla_a T_{i,bj},
    \quad
    \mathrm{sym}(A)_{ij}
    :=A_{ij}+A_{ji},
\end{gather*}
     and $$(T*T)_{ab}:=\frac{1}{2}g^{cd}g^{ef}T_{a,ce}T_{b,df}$$ is a symmetric $2$-tensor.
  \end{proposition}
  \begin{proof}
   Recall our convention to take the inner product of $2$-forms on the last two indices of $T$, so that $$|T|^2=\frac{1}{2}g^{ab}g^{cd}g^{ef} T_{a,ce}T_{b,df}.$$ With this in mind, 
\begin{align*}
    \frac{d}{dt}|T|^2 =&\ g^{ab}g^{cd}g^{ef}\frac{d}{dt}(T_{a,ce})T_{b,df}-\frac{1}{2}B^{ab}g^{cd}g^{ef}T_{a,ce}T_{b,df} -g^{ab}B^{cd}g^{ef}T_{a,ce}T_{b,df}\\
    =&\ g^{ab}g^{cd}g^{ef}\big(\frac{d}{dt}([T_a]_{e}^k)g_{ck}+[T_a]_{e}^kB_{ck}\big)T_{b,df}-\frac{1}{2}B^{ab}g^{cd}g^{ef}T_{a,ce}T_{b,df} -g^{ab}B^{cd}g^{ef}T_{a,ce}T_{b,df}\\
    =&\ g^{ab}g^{cd}g^{ef}\frac{d}{dt}([T_a]_{e}^k)g_{ck}T_{b,df}-\frac{1}{2}B^{ab}g^{cd}g^{ef}T_{a,ce}T_{b,df} -g^{ab}B^{cd}g^{ef}T_{a,ce}T_{b,df} + g^{ab}g^{cd}g^{ef}[T_a]_{e}^kB_{ck}T_{b,df}\\
    =&\ g^{ab}g^{cd}g^{ef}\frac{d}{dt}([T_a]_{e}^k)g_{ck}T_{b,df}-\frac{1}{2}B^{ab}g^{cd}g^{ef}T_{a,ce}T_{b,df}
\end{align*}
where, for the last equality, we used $$g^{ab}g^{cd}g^{ef}[T_a]_{e}^kB_{ck}T_{b,df}= g^{ab}g^{cd}g^{ef}T_{a,ie}g^{ik}B_{ck}T_{b,df} =g^{ab}B^{di}g^{ef}T_{a,ie}T_{b,df}.$$
Using (\ref{eq:pppp}), and integrating by parts,  
\begin{align*}
    \int_M (\frac{d}{dt}|T|^2) \vol =& \int_M \big(g^{ab}g^{cd}g^{ef}\frac{d}{dt}([T_a]_{e}^k)g_{ck}T_{b,df}-\frac{1}{2}B^{ab}g^{cd}g^{ef}T_{a,ce}T_{b,df}\big) \vol\\
    =& \int_M \big(g^{ab}g^{cd}g^{ef}(\nabla_{a}C_{ce}+ (\Lambda\nabla B_a)_{ce})   
    T_{b,df}-\frac{1}{2}B^{ab}g^{cd}g^{ef}T_{a,ce}T_{b,df} \big)\vol\\
    =&-\int_M {2} g\big(\mathrm{div}(T),C\big)\vol+\int_M \big(g^{ab}g^{cd}g^{ef}( 2 \nabla_{c}B_{ae} )   
    T_{b,df}-\frac{1}{2}B^{ab}g^{cd}g^{ef}T_{a,ce}T_{b,df} \big)\vol\\
    =&-\int_M {2} g\big(\mathrm{div}(T),C\big)\vol-\int_M \big(g^{ab}g^{cd}g^{ef}( 2B_{ae} )   
    \nabla_{c}T_{b,df}+\frac{1}{2}B^{ab}g^{cd}g^{ef}T_{a,ce}T_{b,df} \big)\vol\\
    =&-\int_M 2 g\big(\mathrm{div}(T),C\big)\vol-\int_M g\big( 2\mathrm{div}(T^t) + T*T,B\big) \vol.
\end{align*}
    {Note again that the factor `2' in front of the $\divergence(T)$ term is due to our inner product convention on differential forms.}
    The result now follows from the fact that $\frac{d}{dt}\vol = \frac{1}{2}g(g,B)\vol$, from (\ref{eq:evolutionvol}).
\end{proof}
Thus, the critical points of the energy functional $E(\boldsymbol{\om})$ are given by:
\begin{gather}
    \mathrm{div}(T)=0,\label{eq:cptharmonic}\\
    \mathrm{sym}\big(\mathrm{div}(T^t)\big) + T*T - \frac{1}{2}|T|^2g =0.\label{eq:cptsym}
\end{gather}

We now describe these critical points in terms of the $1$-forms $\{a_i\}$.
\begin{proposition}
\label{prop:divTintermsofa}
    In terms of the torsion $1$-forms $\{a_i\}$, as defined by (\ref{intrinsictorsionsu2}), 
    \begin{enumerate}
    \item $\mathrm{div}(T)= \sum_{i=1}\mathrm{div}(a_i)\om_i = - \sum_{i=1}\delta(a_i)\om_i$,
    \item $    \mathrm{div}(T^t) = -\sum_{i,j}(\nabla_{e_i}a_j)\otimes J_j(e^i) -\sum_{i,j,k}\epsilon_{ijk} a_i\otimes J_ja_k$,\label{eq: symdiv T}
    \item $T*T= 2 (a_1 \otimes a_1+ a_2 \otimes a_2+ a_3 \otimes a_3)$,
    \item $|T|^2 = 2 (|a_1|^2+|a_2|^2+|a_3|^2)$.
    \end{enumerate}
\end{proposition}
\begin{proof}
    As usual, let $e_i$ (resp. $e^i$) denote a local $\SU(2)$ (co-)frame, so that $\om_1=e^{12}+e^{34}$, $\om_2=e^{13}-e^{24}$ and  $\om_3=e^{14}+e^{23}$. Using (\ref{intrinsictorsionsu2}) and (\ref{intrinsictorsionsu2-2}), a direct computation yields
\begin{align*}
    \nabla_{e_i}T =\ &\sum_j(\nabla_{e_i}a_j)\otimes\om_j + a_1\otimes (a_3(e_i)\om_2-a_2(e_i)\om_3)+ a_2\otimes (-a_3(e_i)\om_1+a_1(e_i)\om_3)\\ &+ a_3\otimes (a_2(e_i)\om_1-a_1(e_i)\om_2).
\end{align*}
We deduce the first expression easily from the above since $\mathrm{div}(T)=\sum_i e_i \ip \nabla_{e_i} T$. On the other hand, 
using the identity $X\ip\ \om_i= - J_i(X^{\flat})$, 
\begin{align*}
    \mathrm{div}(T^t) =\ &-\sum_{i,j}(\nabla_{e_i}a_j)\otimes J_j(e^i) + a_1\otimes (-J_2a_3+J_3a_2)+ a_2\otimes (J_1a_3-J_3a_1)\\ &+ a_3\otimes (-J_1a_2+J_2a_1)\\
    =\ &-\sum_{i,j}(\nabla_{e_i}a_j)\otimes J_j(e^i) -\sum_{i,j,k}\epsilon_{ijk} a_i\otimes J_ja_k.
\end{align*}
    Finally, the terms $T*T$ and $|T|^2$ are easily computed from (\ref{intrinsictorsionsu2}).
\end{proof}
Note that the expressions in Proposition \ref{prop:divTintermsofa} are not split into $\SU(2)$-irreducible  components.
However, in the sequel it will be important to have an explicit description of such components of the highest-order term of $\mathrm{sym}(\mathrm{div}(T^t))$. To this end, we make the following observation:
\begin{remark}
\label{rem: key remark} 
    Since $\nabla_X a_i = (\nabla a_i)(X,\cdot)= \sum_{j=1}^4 e^j(X) \otimes \nabla_{e_j} a_i$, it follows that $$\nabla_{JX } a_i = \sum_{j=1}^4 J(e^j)(X) \otimes \nabla_{e_j} a_i.$$ 
    Hence, in view of Proposition \ref{prop:divTintermsofa}--(\ref{eq: symdiv T}), up to a factor of $-1$,  the highest-order term of (\ref{eq:cptsym}) is 
\[    \mathrm{sym}\big(\sum_{i,j}(\nabla_{e_i}a_j)\otimes J_j(e^i)\big)= \sum_{i,j}(\nabla_{e_i}a_j)\odot J_j(e^i)  = \sum_{j=1}^3 \mathrm{sym}(\nabla_{J_j} a_j), \]
    where $\mathrm{sym}(\nabla_{J_j} a_j)(X,Y):= (\nabla_{J_jX} a_j) (Y) + (\nabla_{J_jY} a_j) (X)$.
    Using (\ref{eq: nabla of 1 form}) and (\ref{eq: identification1})-(\ref{eq: identification3}), it follows that
\begin{gather}
  \mathrm{sym}(\nabla_{J_1} a_1) = -\frac{1}{2} g(da_1,\om_1) g + P_1\big(\pi^2_-(da_1)\big) + 2 P_3(S^2_{0,2}(\nabla a_1) \diamond \om_2) - 2 P_2(S^2_{0,3}(\nabla a_1) \diamond \om_3),
\end{gather}
and analogous expressions can be deduced for $\mathrm{sym}(\nabla_{J_2} a_2)$ and $\mathrm{sym}(\nabla_{J_2} a_2)$, by permuting indices. This shows that $\mathrm{sym}(\nabla_{J_1} a_1)$ does not depend on the second-order terms $\delta a_1$, $g(da_1,\om_2)$, $g(da_1,\om_3)$ and $S^2_{0,1}(\nabla a_1)$. Note that such a formula was to be expected, since $\mathrm{sym}(\nabla_{J_1} a_1)$ is determined by the $\SU(2)$-invariant components of $\nabla a$; but we have $\mathrm{sym}(\nabla_{J_1} a_1) \in S^2(M) \cong \R \oplus 3 \R^3$, while $\nabla a_1 \in T^* \otimes T^* \cong 4\R \oplus 4\R^3$, so the former term had to be independent of the terms generating a copy of $3\R \oplus \R^3$ which is indeed consistent with the above expression (see also \S \ref{subsect: rep theory} below).
\end{remark}

Note that we can also express $T*T$ in terms of its  $\SU(2)$-irreducible components as 
\begin{equation}
T*T =\frac{1}{2} \sum_i |a_i|^2g + S^2_{0,1}(T*T) + S^2_{0,2}(T*T) + S^2_{0,3}(T*T).
\end{equation}
Together with Proposition \ref{prop:divTintermsofa} it follows that the $1$-dimensional component of (\ref{eq:cptsym}) is given by
\begin{equation}
        g\big(\mathrm{sym}\big(\mathrm{div}(T^t)\big),g\big)  =2\sum_i|a_i|^2.
\end{equation}
Applying the above to the critical condition  (\ref{eq:cptsym}), we deduce:

\begin{corollary}
    If $\boldsymbol{\om}$ is a critical point of $E(\boldsymbol{\om})$ and $\mathrm{div}(T^t)$ is traceless, then $\boldsymbol{\om}$ is an absolute minimum.     
\end{corollary}

Observe that, since $\SU(2)\subset \mathrm{U}(2)$, we are implicitly also considering the almost Hermitian case. In fact, the tensors
\begin{equation}
    T_i:=a_j \otimes \om_j+a_k \otimes \om_k \in \Lm^1 \otimes (\langle \om_i \rangle \oplus \Lm^2_-)^{\perp},
    \qforq  (i,j,k)\sim (1,2,3),
\label{eq: partial energy functional}
\end{equation}
correspond to the intrinsic torsion of the $\U(2)$-structure 
defined by $(g,\om_i,J_i)$ i.e. 
\begin{equation}
T_i \diamond \om_i = \nabla \om_i.
\end{equation} 
Hence it is natural to consider the `partial' energy functionals defined by 
\begin{equation}
E_i(\boldsymbol{\om}):= \int_M |T_i|^2 \vol,
\qforq i=1,2,3,
\label{def:energyfunctionalu(2)torsion}
\end{equation}
each corresponding to the 
Dirichlet energy of the intrinsic torsion of the associated $\mathrm{U}(2)$-structure. In particular, it is not hard to see that the computations in the proofs of Proposition  \ref{propostion:evolutionoftorsionendomorphism} and 
\ref{proposition:firstvariation} also apply to each torsion tensor $T_i$, and we immediately deduce the following:
\begin{proposition}
\label{proposition:firstvariationu(2)}
    The first variation of (\ref{def:energyfunctionalu(2)torsion}) is given by
      \begin{equation}
    \frac{d}{dt}E_i(\boldsymbol{\om})
    =-\int_M g\big(\mathrm{div}(T_i),C\big)\vol-\int_M g\big( \mathrm{sym}\big(\mathrm{div}(T_i^t)\big) + T_i*T_i - \frac{1}{2}|T_i|^2g,B\big) \vol.
\end{equation}
\end{proposition}
It is also worth pointing out that 
\begin{equation}
    E(\boldsymbol{\om})
    =\frac{1}{2}\Big(E_1(\boldsymbol{\om}) +E_2(\boldsymbol{\om}) +E_3(\boldsymbol{\om})\Big),
\end{equation}
and moreover this holds for any $\Sp(n)$-structure, not just $\Sp(1)=\SU(2)$. Next we consider more general $\SU(2)$ functionals which are quadratic in the torsion.

\subsection{General quadratic functionals of the intrinsic torsion}
\label{sec: general quadratic functionals}
In the previous section we considered energy functionals involving the intrinsic torsion and computed their first variation formulae. The negative gradient flows of these functionals provide natural candidates for flows of $\U(2)$- and $\SU(2)$-structures, since the absolute minima correspond precisely to torsion-free structures. In our context, however, there are also other interesting functionals that one can consider; for instance, the functional $\int_M |J_1a_2+a_3|^2\vol$ can be identified with the $L^2$-norm of the Nijenhuis tensor of the almost complex structure $J_1$, see Definition \ref{definitionHKetal}--(3), and hence its negative gradient flow provides a natural way to deform $J_1$ (optimistically) towards a complex structure. Thus, we are led to consider the first variation of functionals of the form
\[
\int_M g(a_i,a_j)\vol \quad \text{ and  } \quad \int_M g\big(a_i,J_j(a_k)\big)\vol.
\]
From (\ref{eq: quadratic terms in torsion}) below, it follows that there are in total 15 such distinct functionals. These can be viewed as building blocks of general functionals of $\SU(2)$-structures which are quadratic in the intrinsic torsion $T$. Note that it suffices to consider the 5 functionals
\[
\int_M g(a_1,a_2)\vol, \quad \int_M g(a_1,a_1)\vol, \quad \int_M g(a_1,J_3a_2)\vol, \quad \int_M g(a_1,J_1a_2)\vol, \quad \int_M g(a_1,J_1a_3)\vol,
\]
since the remaining 10 terms can be obtained by cyclically permuting indices and using the fact that $g$ is compatible with $J_i$; for instance, the functional associated to $g(a_1, J_2 a_3)=-g(J_2a_1,a_3)$ occurs as a cyclic permutation of the third term in the latter expression. Next we compute their first-variation formulae.

\begin{proposition}
\label{prop: first variation 1}
    Under the variation of the $\SU(2)$-structure given by (\ref{eq:evolutionom1})-(\ref{eq:evolutionom3}), 
\begin{align}
    \frac{d}{dt}\int_M g(a_1,a_2)\vol 
    = \frac{1}{2} &\int_M g\big(B, \nabla_{J_1} a_2+ \nabla_{J_2} a_1
    + g(a_1,a_2) g
    - 2 a_1\otimes a_2
     + a_1 \otimes J_3 a_1 - a_1 \otimes J_1 a_3 \nonumber
     \\ &\ \ \ \ \ \ \ \ \ + a_2 \otimes J_2 a_3 - a_2 \otimes J_3 a_2  \big)
    \vol \  \label{eq: variantion a1 a2}
    \\  
     -&\frac{1}{2} \int_M g\big(C,  (\divergence(a_2)-g(a_1, a_3)) \om_1  + (\divergence(a_1)+g(a_2,a_3))\om_2 +
    (|a_1|^2 - |a_2|^2)\om_3\big) \vol. \nonumber
\end{align}
\end{proposition}
\begin{proof}
    Using $\partial_t g^{ij} = - B^{ij}$ and $\partial_t \vol = 2f_0 \vol$, 
\begin{align*}
    \frac{d}{dt}\int_M g(a_1,a_2)\vol 
    &= \frac{d}{dt}\int_M g^{ij}(a_1)_i (a_2)_j \vol    \\
    &= \int_M \Big(-B^{ij}(a_1)_i(a_2)_j+ 2f_0 g(a_1,a_2)+ g^{ij} \frac{d}{dt}(a_1)_i (a_2)_j + g^{ij} (a_1)_i\frac{d}{dt}(a_2)_j \Big)\vol 
\end{align*}
    Using the evolution equations for the torsion forms in Proposition \ref{prop:evolutionofai}, and integrating by parts, we can rewrite the above as
\begin{align*}
    \frac{d}{dt}\int_M g(a_1,a_2)\vol 
    &= \int_M \Big(-B^{ij}(a_1)_i(a_2)_j+ 2f_0 g(a_1,a_2)
    -g^{ij} \nabla_i (a_2)_j f_1 + g^{ij} f_3 (a_2)_i (a_2)_j \\ &\ \ \ \ \ - g^{ij} f_2 (a_2)_j (a_3)_i - g^{ij}f_2 \nabla_j (a_1)_i + g^{ij}f_1 (a_1)_i (a_3)_j - g^{ij} f_3 (a_1)_i (a_1)_j \\ &\ \ \ \ \ + 
    \frac{1}{4}g^{ij}g^{ab}g^{cd}
    (a_2)_j (\om_1)_{bd} (\nabla_a B_{ic}-\nabla_c B_{ia})+ \frac{1}{4}g^{ij}g^{ab}g^{cd}
    (a_1)_i (\om_2)_{bd} (\nabla_a B_{jc}-\nabla_c B_{ja})   
    \Big)\vol.
\end{align*}
    Recall from Proposition \ref{torsionproposition} that $\nabla \om_1 =  - a_2 \otimes \om_3 + a_3 \otimes \om_2$, and similarly for $\nabla \om_2, \nabla \om_3$; on the other hand, for general $1$-form $\alpha$, we have  $(J_i\alpha)_a = - (\om_i)_{ka} g^{bk}\alpha_b$. Using these facts, and integrating by parts again, we get
\begin{align*}
    \frac{d}{dt}\int_M g(a_1,a_2)\vol 
    &= \int_M \Big(-B^{ij}(a_1)_i(a_2)_j+ 2f_0 g(a_1,a_2)
    -g^{ij} \nabla_i (a_2)_j f_1 + g^{ij} f_3 (a_2)_i (a_2)_j \\ 
    &\ \ \ \ \ - g^{ij} f_2 (a_2)_j (a_3)_i - g^{ij}f_2 \nabla_j (a_1)_i + g^{ij}f_1 (a_1)_i (a_3)_j - g^{ij} f_3 (a_1)_i (a_1)_j \\ 
    &\ \ \ \ \ - 
    \frac{1}{2}g^{ij}g^{ab}g^{cd} B_{ic} ( -(a_2)_a (\om_3)_{bd} (a_2)_j + (a_3)_a (\om_2)_{bd} (a_2)_j + (\om_1)_{bd}\nabla_a (a_2)_j)
     \\ &\ \ \ \ \  -  \frac{1}{2}g^{ij}g^{ab}g^{cd} B_{jc} ( -(a_3)_a (\om_1)_{bd} (a_1)_i + (a_1)_a (\om_3)_{bd} (a_1)_i + (\om_2)_{bd}\nabla_a (a_1)_i) 
    \Big)\vol, \\
    &= \int_M \Big(-g(B,a_1\otimes a_2)+ 2f_0 g(a_1,a_2)
    -f_1 \divergence(a_2) - f_2 \divergence(a_1) + f_3 |a_2|^2 -  f_3 |a_1|^2 \\ &\ \ \ \ \ - f_2 g(a_2,a_3)  + f_1 g(a_1, a_3)  - \frac{1}{2} g^{ij}g^{ab}g^{cd} B_{ic} (\om_1)_{bd} \nabla_a (a_2)_j
     - \frac{1}{2} g^{ij}g^{ab}g^{cd} B_{jc} (\om_2)_{bd} \nabla_a (a_1)_i \\ &\ \ \ \ \ + \frac{1}{2} g(B, a_1 \otimes J_3 a_1 - a_1 \otimes J_1 a_3 + a_2 \otimes J_2 a_3 - a_2 \otimes J_3 a_2)  
    \Big)\vol.
    \qedhere
\end{align*} 
\end{proof}
Observe that the highest-order symmetric part in (\ref{eq: variantion a1 a2}) is given by the symmetric components of $\nabla a_2(J_1 ,\cdot ) + \nabla a_1(J_2 , \cdot)$, while the highest-order skew-symmetric part is given by $\divergence(a_2)\om_1  + \divergence(a_1) \om_2$. While the latter term lies in an irreducible $\SU(2)$-module, the former terms can be further decomposed as follows. Using
(\ref{eq: nabla of 1 form}) and (\ref{eq: identification1})-(\ref{eq: identification3}) we can express these as 
\begin{align*}
\nabla a_2(J_1 , ) =  
&-\frac{1}{4}g (da_2,\om_1) g 
+\frac{1}{2} P_1(\pi^2_-(da_2))
- P_2(S^2_{0,3}(\nabla a_2) \diamond \om_3)
+ P_3(S^2_{0,2}(\nabla a_2) \diamond \om_2)\\
&+
\frac{1}{4}\big(  
-(\delta a_2) \om_1 + g(d a_2, \om_3) \om_2 - g(da_2, \om_2) \om_3
\big) 
\end{align*}
and 
\begin{align*}
\nabla a_1(J_2 , ) =
&-\frac{1}{4}g (da_1,\om_2) g 
+\frac{1}{2} P_2(\pi^2_-(da_1))
- P_3(S^2_{0,1}(\nabla a_1) \diamond \om_1)
+ P_1(S^2_{0,3}(\nabla a_1) \diamond \om_3)\\
&+
\frac{1}{4}\big(  
-(\delta a_1) \om_2 + g(d a_1, \om_1) \om_3 - g(da_1, \om_3) \om_1
\big).
\end{align*}
The key point here is that we are considering the inner product with the symmetric endomorphism $B$, so we may disregard the skew-symmetric terms in the above two expressions. By an analogous computation, one can show:
\begin{proposition}
\label{prop: first variation a1 a1}
    Under the variation of the $\SU(2)$-structure given by (\ref{eq:evolutionom1})-(\ref{eq:evolutionom3}), 
\begin{align*}
    \frac{d}{dt}\int_M g(a_1,a_1)\vol = &\int g(B, \nabla_{J_1}a_1 + \frac{1}{2}|a_1|^2g - a_1 \otimes a_1 - a_1 \otimes J_3 a_2 + a_1 \otimes J_2 a_3) \vol \\ &- \int g(C, \divergence(a_1)\om_1 + g(a_1,a_3)\om_2 - g(a_1,a_2)\om_3)  \vol.
\end{align*}
\end{proposition}
From Proposition \ref{prop: first variation a1 a1} we can also deduce the analogous expressions when $a_1$ is replaced by $a_2$ and $a_3$, simply by permuting indices. In particular, the reader can easily verify that the first variation of the Dirichlet functional, as given in Proposition \ref{proposition:firstvariation}, is just a consequence of the latter, in view of (\ref{def:energyfunctionalfulltorsion}). More generally, we can deduce the critical points of the weighted energy functionals
\begin{equation}
   E_{\boldsymbol{\lm}}(\boldsymbol{\om}):=  2\int_M \sum_{i=1}^3 \lm_i |a_i|^2 \vol,\label{eq: weighted energy functional}
\end{equation}
where $\boldsymbol{\lm}=(\lm_1,\lm_2,\lm_3) \in \R^3$; the cases $\boldsymbol{\lm}= (1,1,0), (1,0,1), (0,1,1)$ correspond to the $\U(2)$ energy functionals (\ref{def:energyfunctionalu(2)torsion}), and $\lm=(1,1,1)$ to the Dirichlet energy (\ref{def:energyfunctionalfulltorsion}).

\begin{proposition}
Under the variation of the $\SU(2)$-structure given by (\ref{eq:evolutionom1})-(\ref{eq:evolutionom3}), 
\begin{align}
    \frac{d}{dt}\int_M g(a_1,J_3a_2)\vol 
    = &\int \frac{1}{2}g\big(B, (\nabla a_2)(J_1,J_3)- (\nabla a_1)(J_2,J_3)
    \big) \vol\ \label{eq: variation a1 J3a2}\\
    &+ \int \frac{1}{2}g\big(B,  J_3a_1 \otimes J_1 a_3 - J_3 a_1 \otimes J_3 a_1 + J_2 a_1 \otimes J_2 a_1 - J_2 a_2 \otimes J_1 a_1
    \big) \vol\ \nonumber \\
    &+ \int \frac{1}{2}g\big(B,  J_2a_3 \otimes J_3 a_2 - J_3 a_2 \otimes J_3 a_2 + J_1 a_2 \otimes J_1 a_2 - J_1 a_1 \otimes J_2 a_2
    \big) \vol\ \nonumber \\
    &+ \int g\big(B,  a_2 \otimes J_3 a_1 - a_1 \otimes J_3 a_2 + \frac{3}{4}g(a_1,J_3a_2) g
     + S^2_{0,3}( J_3a_2 \otimes  a_1) \big) \vol\ \nonumber \\
    &- \int \frac{1}{2} g\big(C, \big(g(da_2,\om_3)-g(a_1,J_3a_3)\big)\om_1- \big(g(da_1,\om_3)+g(a_3,J_3a_2)\big)\om_2\big) \vol. \nonumber
\end{align}
\end{proposition}
\begin{proof}
    From the outset, we have
\begin{align*}
        \frac{d}{dt}\int_M g(a_1,J_3a_2)\vol =& \ -\frac{d}{dt}\int_M g^{ij}(a_1)_i(\om_3)_{kj}g^{bk}(a_2)_b\vol \\
        =& -\int_M g(B,  a_1 \otimes J_3 a_2)\vol + \int_M g(a_1, J_3(\frac{d}{dt}a_2))\vol + \int_M g(\frac{d}{dt}a_1, J_3a_2)\vol\\ &+ \int_M 2f_0 g(a_1, J_3a_2)\vol + \int_M g^{ij}B^{bk}(a_1)_i (\om_3)_{kj} (a_2)_b\vol \\ 
        &+ \int_M g^{ij}g^{bk}(a_1)_i \frac{d}{dt}(\om_3)_{kj} (a_2)_b\vol. 
\end{align*} 
    Using (\ref{eq:evolutionom3}) and Proposition \ref{prop:evolutionofai}, we further have
\begin{align*}
        \frac{d}{dt}\int_M g(a_1,J_3a_2)\vol
        =& -\int_M g(B,  a_1 \otimes J_3 a_2)\vol 
        + \int_M 2f_0 g(a_1, J_3a_2)\vol \\
        &- \int_M g^{ij}g^{bk} (\nabla_i f_1 + f_3 (a_2)_i - f_2 (a_3)_i + \frac{1}{2}g(\Lm \nabla B_i ,\om_1)) (\om_3)_{kj} (a_2)_b \vol \\ 
        &- \int_M g^{ij}g^{bk} (a_1)_i (\om_3)_{kj}(\nabla_b f_2 + f_1 (a_3)_b - f_3 (a_1)_b + \frac{1}{2}g(\Lm \nabla B_b ,\om_2)) \vol \\  & + \int_M g^{ij}B^{bk}(a_1)_i (\om_3)_{kj} (a_2)_b\vol \\ 
        &- \int_M g^{ij}g^{bk}(a_1)_i (f_0 (\om_3)_{kj}+ f_2 (\om_1)_{kj} - f_1 (\om_2)_{kj} + (B_{0,3}\diamond \om_3)_{kj} ) (a_2)_b\vol \\
        =& -\int_M g(B,  a_1 \otimes J_3 a_2)\vol 
        + \int_M 2f_0 g(a_1, J_3a_2)\vol \\
        &+ \int_M g^{ij}g^{bk} f_1  (\nabla_i( \om_3)_{kj} (a_2)_b+ (\om_3)_{kj} \nabla_i(a_2)_b) \vol \\ 
        &- \int_M g^{ij}g^{bk} ( f_3 (a_2)_i - f_2 (a_3)_i ) (\om_3)_{kj} (a_2)_b \vol \\ 
        &- \int_M \frac{1}{2} g^{ij}g^{bk} g^{\alpha \beta} g^{cd} \nabla_{\alpha} B_{ic} (\om_1)_{\beta d} (\om_3)_{kj} (a_2)_b \vol \\ 
        &+ \int_M g^{ij}g^{bk}( \nabla_b(a_1)_i (\om_3)_{kj}+ (a_1)_i\nabla_b (\om_3)_{kj}) f_2 \vol \\  
        &- \int_M g^{ij}g^{bk} (a_1)_i (\om_3)_{kj}( f_1 (a_3)_b - f_3 (a_1)_b ) \vol \\
        &- \int_M \frac{1}{2} g^{ij}g^{bk} (a_1)_i (\om_3)_{kj}
        g^{\alpha \beta} g^{cd}
        \nabla_{\alpha} B_{bc} (\om_2)_{\beta d} \vol \\
        & + \int_M g^{ij}B^{bk}(a_1)_i (\om_3)_{kj} (a_2)_b\vol \\ 
        &- \int_M g^{ij}g^{bk}(a_1)_i (f_0 (\om_3)_{kj}+ f_2 (\om_1)_{kj} - f_1 (\om_2)_{kj} + (B_{0,3}\diamond \om_3)_{kj} ) (a_2)_b\vol.
\end{align*} 
    Since $(J_i\alpha)_a = - (\om_i)_{ka} g^{bk}\alpha_b$ for any $1$-form $\alpha$, we then use Proposition \ref{torsionproposition} to compute $\nabla \om_i$ and integrate by parts:
\begin{align*}
        \frac{d}{dt}\int_M g(a_1,J_3a_2)\vol
        =& -\int_M g(B,  a_1 \otimes J_3 a_2)\vol 
        + \int_M 2f_0 g(a_1, J_3a_2)\vol + \int_M g^{ij}g^{bk} f_1  (\om_3)_{kj} \nabla_i(a_2)_b \vol\\
        &+ \int_M g^{ij}g^{bk} f_1  \big(\nabla_i( \om_3)_{kj} (a_2)_b\big) \vol \\ 
        &- \int_M g^{ij}g^{bk} \big( f_3 (a_2)_i - f_2 (a_3)_i \big) (\om_3)_{kj} (a_2)_b \vol \\ 
        &+ \int_M \frac{1}{2} g^{ij}g^{bk} g^{\alpha \beta} g^{cd}  B_{ic} \nabla_{\alpha}\big( (\om_1)_{\beta d} (\om_3)_{kj} (a_2)_b\big) \vol \\ 
        &+ \int_M g^{ij}g^{bk}\big( \nabla_b(a_1)_i (\om_3)_{kj}+ (a_1)_i\nabla_b (\om_3)_{kj}\big) f_2 \vol \\  
        &- \int_M g^{ij}g^{bk} (a_1)_i (\om_3)_{kj}\big( f_1 (a_3)_b - f_3 (a_1)_b \big) \vol \\
        &+ \int_M \frac{1}{2} g^{ij}g^{bk}g^{\alpha \beta} g^{cd} B_{bc}\nabla_{\alpha}\big((a_1)_i (\om_3)_{kj}
        (\om_2)_{\beta d}\big) \vol \\
        & + \int_M g^{ij}B^{bk}(a_1)_i (\om_3)_{kj} (a_2)_b\vol \\ 
        &- \int_M g^{ij}g^{bk}(a_1)_i \big(f_0 (\om_3)_{kj}+ f_2 (\om_1)_{kj} - f_1 (\om_2)_{kj} + (B_{0,3}\diamond \om_3)_{kj} \big) (a_2)_b\vol\\
        =& -\int_M g(B,  a_1 \otimes J_3 a_2)\vol 
        + \int_M 2f_0 g(a_1, J_3a_2)\vol + \int_M g^{ij}g^{bk} f_1  (\om_3)_{kj} \nabla_i(a_2)_b \vol\\
        &+ \int_M g^{ij}g^{bk} f_1  \big( -(a_1)_i (\om_2)_{kj} + (a_2)_i (\om_1)_{kj}\big) (a_2)_b \vol \\ 
        &- \int_M g^{ij}g^{bk} \big( f_3 (a_2)_i - f_2 (a_3)_i \big) (\om_3)_{kj} (a_2)_b \vol \\ 
        &+ \int_M \frac{1}{2} g^{ij}g^{bk} g^{\alpha \beta} g^{cd}  B_{ic} \nabla_{\alpha}(a_2)_b  (\om_1)_{\beta d} (\om_3)_{kj}  \vol \\
        &+ \int_M \frac{1}{2} g^{ij}g^{bk} g^{\alpha \beta} g^{cd}  B_{ic} (a_2)_b \nabla_{\alpha}( (\om_1)_{\beta d} (\om_3)_{kj}) \vol \\ 
        &+ \int_M g^{ij}g^{bk}\big( \nabla_b(a_1)_i (\om_3)_{kj}+ (a_1)_i(-(a_1)_b (\om_2)_{kj}+(a_2)_b(\om_1)_{kj})\big) f_2 \vol \\  
        &- \int_M g^{ij}g^{bk} (a_1)_i (\om_3)_{kj}\big( f_1 (a_3)_b - f_3 (a_1)_b \big) \vol \\
        &+ \int_M \frac{1}{2} g^{ij}g^{bk}g^{\alpha \beta} g^{cd} B_{bc}\big(\nabla_{\alpha}(a_1)_i (\om_3)_{kj}
        (\om_2)_{\beta d}
        \big) \vol \\
        &+ \int_M \frac{1}{2} g^{ij}g^{bk}g^{\alpha \beta} g^{cd} B_{bc}(a_1)_i\big(        
         \nabla_{\alpha}(\om_3)_{kj}
        (\om_2)_{\beta d}
        + (\om_3)_{kj}
        \nabla_{\alpha}(\om_2)_{\beta d}
        \big) \vol \\
        & + \int_M g^{ij}B^{bk}(a_1)_i (\om_3)_{kj} (a_2)_b\vol \\ 
        &- \int_M g^{ij}g^{bk}(a_1)_i \big(f_0 (\om_3)_{kj}+ f_2 (\om_1)_{kj} - f_1 (\om_2)_{kj} + (B_{0,3}\diamond \om_3)_{kj} \big) (a_2)_b\vol.
\end{align*} 
    The assertion follows by simplifying the above and applying 
$g(B_{0,3}, J_3 a_2 \otimes a_1) = -g(B_{0,3}\diamond \om_3, a_2 \otimes a_1)$ which follows from a simple computation using the results in \S \ref{sec: preliminary}.
\end{proof}

The highest-order symmetric part in (\ref{eq: variation a1 J3a2}) is given by the symmetric components of $(\nabla a_2)(J_1,J_3)$ and $(\nabla a_1)(J_2,J_3)$. Using
(\ref{eq: nabla of 1 form}) and (\ref{eq: identification1})-(\ref{eq: identification3}), we can express these as 
\begin{align*}
(\nabla a_2)(J_1,J_3) =  
&-\frac{1}{4}g (da_2,\om_2) g 
- P_1\big(S^2_{0,3}(\nabla a_2) \diamond \om_3\big)
-\frac{1}{2} P_2\big(\pi^2_-(da_2)\big)
- P_3\big(S^2_{0,1}(\nabla a_2) \diamond \om_1\big)\\
&+
\frac{1}{4}\big(  
 g(d a_2, \om_3) \om_1 + (\delta a_2) \om_2  + g(da_2, \om_1) \om_3
\big) + S^2_{0,2}(\nabla a_2) \diamond \om_2 
\end{align*}
and 
\begin{align*}
(\nabla a_1)(J_2,J_3) =
&+\frac{1}{4}g (da_1,\om_1) g 
+\frac{1}{2} P_1\big(\pi^2_-(da_1)\big)
- P_2\big(S^2_{0,3}(\nabla a_1) \diamond \om_3\big)
- P_3\big(S^2_{0,2}(\nabla a_1) \diamond \om_2\big)\\
&+
\frac{1}{4}\big(  
-(\delta a_1) \om_1 + g(d a_1, \om_3) \om_2 + g(da_1, \om_2) \om_3
\big) - S^2_{0,1}(\nabla a_1) \diamond \om_1,
\end{align*}
where as before the skew-symmetric terms do not play a role in the variational formula.

\begin{proposition}
Under the variation of the $\SU(2)$-structure given by (\ref{eq:evolutionom1})-(\ref{eq:evolutionom3}), 
\begin{align}
    \frac{d}{dt}\int_M g(a_1,J_1a_2)\vol 
    =  &\int_M \frac{1}{2} g\big(B, (\nabla a_2)(J_1,J_1) - (\nabla a_1)(J_2,J_1)\big) \vol\label{first variation 4}\\
    &+ \int_M \frac{1}{2}g\big(B, J_1 a_2 \otimes J_2 a _3 + J_2 a_2 \otimes J_1 a_3 - J_1 a_2 \otimes J_3 a_2 - J_3 a_2 \otimes J_1 a_2  
    \big) \vol\ \nonumber \\
    &+ \int_M \frac{1}{2}g\big(B, J_3 a_1 \otimes J_2 a _2 + J_1 a_1 \otimes J_1 a_3 - J_1 a_1 \otimes J_3 a_1 - J_2 a_1 \otimes J_2 a_3  
    \big) \vol\  \nonumber\\
    &+ \int_M g\big(B, a_2 \otimes J_1 a_1 - a_1 \otimes J_1 a_2 + \frac{1}{2} g(a_1,J_1a_2) g - \frac{1}{4} g(a_2, J_1 a_1) g
    \big) \vol\ \nonumber \\
    & - \int_M g\big(B,  S^2_{0,1}(J_1a_1 \otimes  a_2) \big) \vol\ \nonumber\\
    &+\int_M \frac{1}{2} g\big(C, 
    -g(da_2, \om_1) \om_1 + g(da_1,\om_1) \om_2 \big) \vol \nonumber\\
    &+\int_M \frac{1}{2}g\big(C, 
    \big(g(a_2,J_2 a_3) - g(a_3, J_1 a_1)\big)\om_1 +\nonumber \\ &\ \ \ \ \ \ \ \ \ 
     \big(g(a_2,J_3 a_1) + g(a_3, J_2 a_1) - g(a_3,J_1 a_2) - g(a_2, J_3 a_1)\big)\om_2-
      g(a_2,J_2 a_1)\om_3
    \big) \vol.\nonumber
\end{align}
\end{proposition}
\begin{proof}
    By direct computation,
\begin{align*}
        \frac{d}{dt}\int_M g(a_1,J_1a_2)\vol =& \ -\frac{d}{dt}\int_M g^{ij}(a_1)_i(\om_1)_{kj}g^{bk}(a_2)_b\vol \\
        =& -\int_M g(B,  a_1 \otimes J_1 a_2)\vol + \int_M g\big(a_1, J_1(\frac{d}{dt}a_2)\big)\vol + \int_M g(\frac{d}{dt}a_1, J_1a_2)\vol\\ &+ \int_M 2f_0 g(a_1, J_1a_2)\vol + \int_M g^{ij}B^{bk}(a_1)_i (\om_1)_{kj} (a_2)_b\vol \\ 
        &+ \int_M g^{ij}g^{bk}(a_1)_i \frac{d}{dt}(\om_1)_{kj} (a_2)_b\vol. 
\end{align*} 
A long computation as before shows 
\begin{align*}
        \frac{d}{dt}\int_M g(a_1,J_1a_2)\vol 
        =& -\int_M g(B,  a_1 \otimes J_1 a_2)\vol 
        + \int_M 2f_0 g(a_1, J_1a_2)\vol + \int_M g^{ij}g^{bk} f_1  (\om_1)_{kj} \nabla_i(a_2)_b \vol\\
        &+ \int_M g^{ij}g^{bk} f_1  \nabla_i( \om_1)_{kj} (a_2)_b \vol \\ 
        &- \int_M g^{ij}g^{bk} \big( f_3 (a_2)_i - f_2 (a_3)_i \big) (\om_1)_{kj} (a_2)_b \vol \\ 
        &+ \int_M \frac{1}{2} g^{ij}g^{bk} g^{\alpha \beta} g^{cd}  B_{ic} \nabla_{\alpha}\big( (\om_1)_{\beta d} (\om_1)_{kj} (a_2)_b\big) \vol \\ 
        &+ \int_M g^{ij}g^{bk}\big( \nabla_b(a_1)_i (\om_1)_{kj}+ (a_1)_i\nabla_b (\om_1)_{kj}\big) f_2 \vol \\  
        &- \int_M g^{ij}g^{bk} (a_1)_i (\om_1)_{kj}\big( f_1 (a_3)_b - f_3 (a_1)_b \big) \vol \\
        &+ \int_M \frac{1}{2} g^{ij}g^{bk}g^{\alpha \beta} g^{cd} B_{bc}\nabla_{\alpha}\big((a_1)_i (\om_1)_{kj}
        (\om_2)_{\beta d}\big) \vol \\
        & + \int_M g^{ij}B^{bk}(a_1)_i (\om_1)_{kj} (a_2)_b\vol\\ 
        &- \int_M g^{ij}g^{bk}(a_1)_i \big(f_0 (\om_1)_{kj}+ f_3 (\om_2)_{kj} - f_2 (\om_3)_{kj} + (B_{0,1}\diamond \om_1)_{kj} \big) (a_2)_b\vol
\end{align*} 
    and the result follows by simplifying the above. 
\end{proof}
Note that the highest-order symmetric terms in (\ref{first variation 4}) can be expressed as:
\[
\int_M  g\big(B, (\nabla a_2)(J_1,J_1) \big) \vol = \int_M  g\big(B, -\frac{1}{4} (\delta a_2) g 
 + S^2_{0,1}(\nabla a_2)
 - S^2_{0,2}(\nabla a_2)
 - S^2_{0,3}(\nabla a_2)\big) \vol
\]
and 
\[
\int_M  g\big(B, (\nabla a_1)(J_2,J_1) \big) \vol = -\int_M  g\big(B, \frac{1}{4} g(da_1,\om_3) g 
 + P_1( S^2_{0,2}(\nabla a_1)\diamond \om_2)
 + P_2( S^2_{0,1}(\nabla a_1)\diamond \om_1)
 + \frac{1}{2} P_3(\pi^2_-(da_1))\big) \vol.
\]
Finally, by similar computations, we have:

\begin{proposition}
\label{prop: first variation 5}
    Under the variation of the $\SU(2)$-structure given by (\ref{eq:evolutionom1})-(\ref{eq:evolutionom3}), 
\begin{align}
    \frac{d}{dt}\int_M g(a_1,J_1a_3)\vol 
    =  &\int_M \frac{1}{2} g\big(B, (\nabla a_3)(J_1,J_1) - (\nabla a_1)(J_3,J_1)\big) \vol\label{first variation 5} \\
    &+ \int_M \frac{1}{2}g(B, J_1 a_3 \otimes J_2 a _3 + J_3 a_1 \otimes J_3 a_2 - J_1 a_3 \otimes J_3 a_2 - J_3 a_3 \otimes J_1 a_2  
    ) \vol\  \nonumber\\
    &+ \int_M \frac{1}{2}g(B, J_1 a_1 \otimes J_2 a _1 + J_1 a_3 \otimes J_2 a_3 - J_1 a_1 \otimes J_1 a_2 - J_2 a_1 \otimes J_3 a_3  
    ) \vol\ \nonumber \\
    &+ \int_M g\big(B, a_3 \otimes J_1 a_1 - a_1 \otimes J_1 a_3 + \frac{1}{2} g(a_1,J_1a_3) g + \frac{1}{4} g(a_1, J_1 a_3) g
    \big) \vol\  \nonumber\\
    & - \int_M g\big(B,  S^2_{0,1}(J_1a_1 \otimes  a_3) \big) \vol\ \nonumber\\
    &+\int_M \frac{1}{2} g\big(C, 
    -g(da_3, \om_1) \om_1 + g(da_1,\om_1) \om_3 \big) \vol\nonumber \\
    &+\int_M \frac{1}{2}g\big(C, 
    \big(g(a_2,J_1 a_1) - g(a_3, J_3 a_2)\big)\om_1 +\nonumber \\ &\ \ \ \ \ \ \ \ \ 
     \big(g(a_2,J_1 a_3) + g(a_3, J_2 a_1) + g(a_1,J_2 a_3) - g(a_2, J_3 a_1)\big)\om_3-
      g(a_1,J_3 a_3)\om_2
    \big) \vol.\nonumber
\end{align}
\end{proposition}
Note that the highest order symmetric terms in (\ref{first variation 5}) can be expressed as:
\[
\int_M  g\big(B, (\nabla a_3)(J_1,J_1) \big) \vol = \int_M  g\big(B, -\frac{1}{4} (\delta a_3) g 
 + S^2_{0,1}(\nabla a_3)
 - S^2_{0,2}(\nabla a_3)
 - S^2_{0,3}(\nabla a_3)\big) \vol
\]
and 
\[
\int_M  g\big(B, (\nabla a_1)(J_3,J_1) \big) \vol = \int_M  g\big(B, \frac{1}{4} g(da_1,\om_2) g 
 - P_1\big( S^2_{0,3}(\nabla a_1)\diamond \om_3\big)
 + \frac{1}{2} P_2\big(\pi^2_-(da_1)\big)
 - P_3\big( S^2_{0,1}(\nabla a_1)\diamond \om_1\big)\big) \vol.
\]
Having computed the first variation formulae for all the quadratic functionals in the intrinsic torsion we next classify the second-order invariants of $\SU(2)$-structures.

\newpage
\section{Second-order invariants}
\label{section: second order}

It is a classical fact that all the second-order differential invariants of a Riemannian manifold, i.e. tensors defined by second derivatives of the metric, are determined by the Riemann curvature tensor. This is however not the case for general $H$-structures with $H\subset \SO(n)$, see for instance \cite{Bryant06someremarks} for a detailed description of the $\mathrm{G}_2$ case. Furthermore, the first-order invariants given by the components of the intrinsic torsion  determine part of the curvature tensor, see eg. \S\ref{sec: intrinsic torsion}; this is to say that there are relations between $\nabla T$ and $\Rm$, see Corollary \ref{cor:bianchitypeidentity}. In the case $H=\mathrm{G}_2$ considered in \cite{Dwivedi2023}, the authors derive these relations by direct computations from the $\mathrm{G}_2$ Bianchi identity, then use these to extract a basis of second-order invariants, in a process which is computationally quite involved. There is however an alternative approach using representation theory 
due to Bryant [ibid.], which effectively allows us to deduce the number of second-order invariants by purely algebraic methods, then it is just a matter of choosing a basis of such invariants. In this section, we give an explicit description of this phenomenon in the $H=\SU(2)$ case. Later on, in \S\ref{sec: comparison with g2 and spin7}, we systematise 
the corresponding results in the $\SU(3)$, $\mathrm{G}_2$ and $\Spin(7)$ cases.

\subsection{The Riemann curvature tensor}
\label{section: curvature tensor}

On an oriented Riemannian manifold $(M^4,g)$,  the curvature tensor $\Rm$ decomposes into $4$ independent components: the scalar curvature $s$, the traceless Ricci tensor $\Ric_0$, and the self-dual and anti-self-dual parts of the Weyl curvature tensor $W^\pm$. More concretely, we can describe these irreducible $\SO(4)$-modules as follows:
$$
\Rm= \mathrm{Scal} + \Ric_0 + W^+ + W^- \in \R \oplus \Lm^2_+ \otimes \Lm^2_- \oplus S^2_0(\Lm^2_+) \oplus S^2_0(\Lm^2_-) \cong \ker (S^2(\Lm^2) \xrightarrow{\w} \Lm^4 ).
$$
If $M$ admits an $\SU(2)$-structure, say determined by the triple $\{ \om_1,\om_2,\om_3\}$, then we can further decompose $\Rm$ as follows: 
\begin{enumerate}
    \item The scalar curvature $\mathrm{Scal}$ is $1$-dimensional, so it clearly does not further split.
    
    \item $\Ric_0$ splits into $3$ components, according to the splitting of symmetric tensors $\Sigma^2_0=\Sigma^2_{0,1}\oplus \Sigma^2_{0,2} \oplus \Sigma^2_{0,3}$.
    \item $W^+$ can be expressed as a sum of the $5$ terms 
    $$2\om_1 \otimes \om_1-\om_2 \otimes \om_2-\om_3 \otimes \om_3,
    \quad
    \om_2 \otimes \om_2-\om_3 \otimes \om_3, 
    \quad
    \om_1 \odot \om_2,
    \quad
    \om_1 \odot \om_3
    \qandq
    \om_2 \odot \om_3,
    $$
    with suitable coefficient functions. In other words, the latter $5$ terms provide an orthogonal basis for $S^2_0(\Lm^2_+)$.
    \item The anti-self-dual Weyl tensor $W^-$ remains irreducible as a section of $S^2_0(\Lm^2_-)$.
\end{enumerate}

As already seen in Corollary \ref{cor:bianchitypeidentity}, the derivatives of $a_i$ already determine part of the curvature tensor. Furthermore, from the above decomposition, it is not hard to see that these include all the curvature components aside from $W^-$, since the latter lies in the kernel of $\diamond \boldsymbol{\om}$. Before deriving explicit expressions for the Ricci curvature and $W^+$ in terms of $a_i$, we shall next describe all the second-order invariants according to Bryant's perspective.

\subsection{Some representation theory}
\label{subsect: rep theory}

Following \cite{Bryant06someremarks}, we describe the space of first- and second-order invariants of an $\SU(2)$-structure abstractly as irreducible $\SU(2)$-modules. We shall give an explicit basis of these second-order invariants in Theorem \ref{theorem: classification of second-order invariants}. 

We begin by considering the case of general $H$-structures, with $H\subset \mathrm{SO}(n)$. In \cite{Bryant06someremarks}*{\S4.2}, all the $k$-th order invariants of an $H$-structure are obtained as sections of the associated bundle to the $H$-module  $V_k(\mathfrak{h})$, which is implicitly defined by
\begin{equation}
V_k(\mathfrak{h}) \oplus (\Lm^1 \otimes S^{k+1}(\Lm^1))\cong (S^2(\Lm^1)\oplus \mathfrak{h}^\perp )\otimes S^k(\Lm^1).
\label{eq: space of second order invariants}
\end{equation} 
In this article we are only be concerned with $k=1,2$.
Since Bryant does not include the proof of \eqref{eq: space of second order invariants} therein, we sketch the proof for $k=2$; the case $k=1$ is just the space where intrinsic torsion lies, which is already well-known, cf. \cites{Bryant1987, Salamon1989}.
\begin{proof}[Proof (sketch)]
\label{proof of bryant formula}
    Recall from the beginning of \S\ref{sec: intrinsic torsion} that the intrinsic torsion $T$ of a  $H$-structure is given by
\[
\nabla \xi =  T \diamond \xi,
\]
where $\xi$ is the (multi-)tensor defining the $H$-structure. It follows, by differentiation, that
\[
\nabla^2 \xi = \nabla T \diamond \xi + T \diamond ( T \diamond \xi).
\]
Hence, skew-symmetrising on the first two indices and using (\ref{equ: torsion plus canonical connection}),  we have  $$F_{\nabla} \diamond \xi = \mathrm{alt}(\nabla^{\mathfrak{h}} T) \diamond \xi + \mathrm{l.o.t}(\xi).$$ Since $\nabla^{\mathfrak{h}} T \in \Lm^1 \otimes \Lm^1 \otimes \mathfrak{h}^\perp$, it follows  that $\mathrm{alt}(\nabla^{\mathfrak{h}} T) \diamond \xi \in \Lm^2 \otimes \mathfrak{h}^\perp$.
The second-order invariants are obtained from $\nabla T$ and $\mathrm{Rm}$, but these overlap precisely in the space $\Lm^2 \otimes \mathfrak{h}^\perp$, so  $V_2(\mathfrak{h})$ is given by
$S^2(\Lm^1) \otimes \mathfrak{h}^\perp$ (consisting of those invariants not coming from the curvature) together with the space of curvature tensors ($S^2(\Lm^1)\otimes S^2(\Lm^1)$ modulo the terms in $\Lm^1 \otimes S^3(\Lm^1)$),  cf. \cite{BergerBryantGriffiths83}*{\S 5}. This concludes the proof. For the cases $k>2$, the proof is similar but one must also use the differential Bianchi identity.
\end{proof}
We now specialise to the case $H=\SU(2)$. We already know from  \S\ref{sec: intrinsic torsion} that 
$$V_1(\mathfrak{su}(2))\cong 3\Lm^1 \cong 3\R^4$$ 
is precisely the space of intrinsic torsion. On the other hand  $V_2(\mathfrak{su}(2))$ is implicitly defined by 
\begin{equation}
    V_2(\mathfrak{su}(2)) \oplus (\Lm^1 \otimes S^3(\Lm^1))\cong (S^2(\Lm^1)\oplus 3\R)\otimes S^2(\Lm^1).\label{eq: su2 second-order invariants}
\end{equation} 
A straightforward calculation using the Clebsch–Gordan formula, cf. \cite{Salamon1989}, shows that 
\begin{align*}
    (S^2(\Lm^1) \oplus 3\R) \otimes S^2(\Lm^1) &\cong  9 \R^5 \oplus 24 \R^3 \oplus 13 \R,\\
    \Lm^1 \otimes S^3(\Lm^1) &\cong 8 \R^5 \oplus 12 \R^3 \oplus  4 \R.
\end{align*}
Since $\R^3 \cong \mathfrak{su}(2)$ and $\R^5 \cong S^2_0(\R^3)$, as $\SU(2)$-representations, we deduce that
\begin{equation}
    V_2(\mathfrak{su}(2))\cong \R^5 \oplus 12 \R^3 \oplus 9 \R.
    \label{eq: v2 su2}
\end{equation}
On the other hand, as an $\SO(4)$-module, we have 
$$V_2(\mathfrak{so}(4))\cong \R^5_+ \oplus \R^5_- \oplus (\R^3_+\otimes \R^3_-) \oplus \R,$$ 
which is of course just the space of the curvature tensor for an oriented Riemannian $4$-manifold. 
Refining the latter as an $\SU(2)$-module gives
\begin{equation}
 V_2(\mathfrak{so}(4))\cong \R^5 \oplus 3\R^3 \oplus 6 \R.
\end{equation}
The latter is exactly the decomposition given by (1)-(4) in \S\ref{section: curvature tensor}. 
Comparing we see that $V_2(\mathfrak{su}(2))$ has an extra $9\R^3 \oplus 3 \R$-modules, which correspond precisely to the second-order invariants which do not arise from curvature terms. In the next section we shall find an explicit generator for each component of $V_2(\mathfrak{su}(2))$.

\subsection{Second-order invariants explicitly}

Recall that $T \in \Lm^1 \otimes \Lm^2_+ \cong 3 \Lm^1$, since $\Lm^2_+=\langle\om_1,\om_2,\om_3\rangle$, and hence $T$ is identified with the triple of $1$-forms $(a_1,a_2,a_3)$. It follows that $\nabla T$ is determined by $\nabla a_i$:
\begin{equation}
\nabla a_i \in 4\R \oplus \Lm^2_- \oplus \Sigma^2_{0,1} \oplus \Sigma^2_{0,2} \oplus \Sigma^2_{0,3},\label{eq:covariant derivative of ai}
\end{equation}
where for each $i$ the four copies of $\R$ are determined by $\delta a_i$ and $g(da_i, \om_j)$, for $j=1,2,3$; the $\Lm^2_-$-component is just the anti-self-dual part of $da_i$; and the $S^2_{0,j}$ components are the type-decomposition of the symmetric part of $\nabla a_i$. We should point out that these second-order terms are not all independent. We summarise their dependency relations in the next proposition.
\begin{proposition}
\label{prop: second order relations for torsion}
    The following identities hold:
    \begin{enumerate}
        \item $g(da_2,\om_3)=g(da_3,\om_2)+g(a_1,J_2a_2)+g(a_1,J_3a_3)$
        \item $g(da_3,\om_1)=g(da_1,\om_3)+g(a_2,J_3a_3)+g(a_2,J_1a_1)$
        \item $g(da_1,\om_2)=g(da_2,\om_1)+g(a_3,J_1a_1)+g(a_3,J_2a_2)$
    \end{enumerate}
\end{proposition}
\begin{proof}
    The result follows by taking the exterior derivative of the relations 
    (\ref{eq: exterior derivative of om1}) and simplifying. For instance, differentiating $d\om_1=-a_2 \w \om_3 +a_3 \w \om_2$ gives $0=-da_2 \w \om_3 +a_1 \w a_2 \w \om_2+da_3 \w \om_2+a_1\w a_3 \w \om_3$, which yields the first identity using Lemma \ref{lemma:relations}, and likewise for the remaining ones.
\end{proof}
    This shows that the exterior derivatives of the torsion $1$-forms $a_i$ are not completely independent: the $9$ second-order terms $g(da_i,\om_j)$ are related by the above $3$ relations, up to addition of lower-order terms which are quadratic in $T$. Recall from the representation theory above that there are $9$ copies of $\R$ in $V_2(\mathfrak{su}(2))$. On the other hand, we also saw that there are $12$ $\R$-components which can be obtained from $\nabla a_i$, namely, $\delta a_i$ and $g(da_i,\om_j)$. The above proposition shows that there are $3$ relations between the latter terms and hence a total of $12-3=9$ independent $\R$-terms arising from $\nabla a_i$. Thus, all the $\R$-components of $V_2(\mathfrak{su}(2))$ are precisely given by them, and there are no further relations; this would not have been obvious, had one not computed the space $V_2(\mathfrak{su}(2))$ in (\ref{eq: v2 su2}) beforehand. In the $\mathrm{G}_2$ case considered in \cite{Dwivedi2023}, the authors derive all the relations coming from $\nabla T$ and $\mathrm{Rm}$, which is as impressive as it is laborious; our approach circumvents this effort. The conclusion of Proposition \ref{prop: second order relations for torsion}, (\ref{eq: su2 second-order invariants}) and (\ref{eq:covariant derivative of ai}) can be summarised into:
\begin{theorem}
\label{theorem: classification of second-order invariants}
    All second-order invariants of an $\SU(2)$-structure can be expressed from these linearly independent terms: 
    \begin{enumerate}
        \item The 9 functions $\delta a_1,\delta a_2,\delta a_3 $ and $g(da_i, \om_j)$, where $1\leq i\leq j \leq 3$.
        \item The 12 anti-self-dual forms $\pi^2_-(da_i)$, where $i=1,2,3$, and $S^2_{0,j}(\nabla a_i) \diamond \om_j$, where $i,j=1,2,3$.
        \item The anti-self-dual Weyl curvature $W^- \in S^2_0(\Lm^2_-)$.
    \end{enumerate}
\end{theorem}
As we already saw in \S\ref{sec: general quadratic functionals}, all the above second order operators (with the exception of $W^-$) arise naturally when considering the first variational formulae of $\SU(2)$ functionals, see Propositions \ref{prop: first variation 1} - \ref{prop: first variation 5}.
\begin{remark}
    As a point of comparison, the results in \cite{Bryant06someremarks}*{\S 4.3} show that, in the $\mathrm{G}_2$ case, there are in fact 11 second-order invariants; 6 of which were described explicitly in \cite{Dwivedi2023}*{\S 5}, as these can be used to flow the underlying $\mathrm{G}_2$-structure. 
In \S \ref{sec: comparison with g2 and spin7} below, we provide a systematic method for describing the full generality of second-order flows for arbitrary $H$-structures.
\end{remark}

Finally,  we derive formulae to compute certain components of $d(J_ia_j)$ from those of $da_i$, similar to those we found in Proposition \ref{prop: second order relations for torsion}.

\begin{proposition}
\label{proposition: to be named after}
Let $\al$ be an arbitrary $1$-form on $M$. Then the $12$ second-order terms given by $g(d(J_i(\al)), \om_j)$ and $\delta(J_i\al))$ are in fact completely determined by the $4$ second-order terms $g(d\al, \om_i)$ and $\delta \al$, up to addition of lower-order terms. More precisely, 
    \begin{align*}
    \delta \al &= g(d(J_1\al),\om_1)-g(\al,J_2\al_2)-g(\al,J_3\al_3)\\ &= g(d(J_2\al),\om_2)-g(\al,J_3\al_3)-g(\al,J_1\al_1)\\ &= g(d(J_3\al),\om_3)-g(\al,J_1\al_1)-g(\al,J_2\al_2)
    \end{align*}
and
\begin{gather*}
    g(d (J_i \alpha), \om_j )= +g(d\alpha, \om_k)- g(\alpha, J_i(a_j) + a_k ),\\
    g(d (J_i \alpha), \om_k )= -g(d\alpha, \om_j)- g(\alpha, J_i(a_k) - a_j ),
\end{gather*}
for $(i,j,k)\sim (1,2,3)$.
\end{proposition}
\begin{proof}
    From Lemma \ref{lemma:relations}--(1), we know that $ \al \w \om_1 = -J_3(\al) \w \om_2 = J_2(\al) \w \om_3.$ The claim follows by applying the exterior derivative to the latter relation. For instance,
    \begin{align*}
    d(J_3(\al))\w \om_2 = &-d\al \w \om_1 - \al \w a_2 \w \om_3+\al  \w a_3 \w \om_2\\
    &-a_1 \w J_3\al \w \om_3 + a_3 \w J_3\al \w \om_1,
    \end{align*}
    and hence, using Lemma \ref{lemma:relations}, 
    \[
    g(d(J_3\al),\om_2) = - g(d\al,\om_1)-g(\al,J_3a_2)+g(\al, a_1).
    \]
    Likewise, since $$(\delta \al)\vol = d((J_1\al)\w \om_1)= d((J_2\al)\w \om_2)= d((J_3\al)\w \om_3),$$
    we have
    \begin{align*}
    \delta \al = g(d(J_1\al),\om_1)-g(\al,J_2\al_2)-g(\al,J_3\al_3).
    \end{align*}
    The other expressions are obtained  analogously.
\end{proof}

In the next section we express the Ricci curvature and the self-dual Weyl tensor in terms of the second-order invariants
of Theorem \ref{theorem: classification of second-order invariants}
and quadratic terms in the first-order invariants which lie in 
\begin{equation}
S^2\big(V_1(\mathfrak{su}(2))\big) \cong 15 \R \oplus 21 \R^3. \label{eq: quadratic terms in torsion}
\end{equation}
Note that the $15$ functions of (\ref{eq: quadratic terms in torsion}) correspond to $g(a_i,a_j)$ and $g(a_i,J_ja_k)$, whose first variation of the $L^2$-norm we computed in \S\ref{section: energy functionals},  and the $\R^3$ components are spanned by terms of the form $\pi^2_-(J_i a_j \w a_k)$ and $S^2_{0,i}(J_j a_k \otimes a_l)$.

\subsubsection{Ricci tensor in terms of intrinsic torsion}
\begin{proposition}\label{prop: ricci in terms of torsion}
The traceless Ricci tensor of $g$ can be computed explicitly in terms of the intrinsic torsion, by
\[
   \mathrm{Ric_0}(g) = P_1(\pi^2_-(\Phi_1))+P_2(\pi^2_-(\Phi_2))+P_3(\pi^2_-(\Phi_3)) \in \Sigma^2_{0,1}(M)\oplus \Sigma^2_{0,2}(M) \oplus \Sigma^2_{0,3}(M),
\]
where  
\begin{align*}
    \Phi_1 :=\ &da_1+\frac{1}{4}\big(J_2(a_2) \w J_2(a_3)-J_1(a_1) \w J_2(a_3)+J_2(a_2) \w a_1\big)\\ &-\frac{1}{4}\big(J_1(a_1) \w J_3(a_2)+J_3(a_3) \w a_1\big) -\frac{3}{4}\big(J_3(a_3) \w J_3(a_2)\big),
\end{align*}
and $\Phi_2,\Phi_3$ are obtained by cyclically permuting the indices of $\Phi_1$. The scalar curvature is given by
\begin{equation}
{\mathrm{Scal}(g)} = -2\Big(\sum_{i=1}^3 g(da_i,\om_i)+g(J_1(a_1),J_2(a_2))+g(J_1(a_1),J_3(a_3))+g(J_2(a_2),J_3(a_3))\Big).\label{eq: scalar curvature}
\end{equation}
\end{proposition}
\begin{proof}
From 
Theorem \ref{theorem: classification of second-order invariants} we know that there are 12 possible second-order terms to compute $\Ric_0$, whereas from Corollary \ref{cor:bianchitypeidentity} we know that one can restrict to just the $\pi^2_-(da_i)$, so there are only 3 possible terms. On the other hand, the lower-order quadratic terms belong in
\eqref{eq: quadratic terms in torsion}, of which there are 21 possibilities. 
It is now just a matter of identifying the coefficient constants for each of these terms, which can be done by considering suitable examples.
\end{proof}
\begin{remark}
The $\mathrm{G}_2$ analogue of Proposition \ref{prop: ricci in terms of torsion} was derived in \cite{Bryant06someremarks}*{\S 4.5.3}, and the $\mathrm{Spin}(7)$ analogue in 
 \cite{UdhavFowdar}*{Proposition 4.1}. 
\end{remark}
To the best of our knowledge, the above expressions for the irreducible components of the Ricci tensor in terms of the $\SU(2)$ torsion forms do not seem to have appeared in the literature; except under further assumptions on the underlying structure, such as Hermitian or K\"ahler. In particular, we make the following observations:
\begin{enumerate}
    \item $g$ is an Einstein metric if, and only if, $\Phi_i$ are all self-dual $2$-forms.
    \item if $\om_1$ is a K\"ahler form, i.e. $a_2=a_3=0$, then $\Phi_1=-da_1$ and $\Phi_2=\Phi_3=0$. Thus, $\Ric(g) \in \langle g \rangle \oplus \Sigma^2_{0,1}$ i.e. it is of type $(1,1)$, which is of course just the well-known fact that the closed $2$-form $\Phi_1$ is the Ricci form, up to a constant factor, cf. \cite{KobayashiNomizu2}*{Chapter IX. 4}. It is also easy to see from (\ref{eq:evolutionofJi2}) that the Ricci flow preserves the complex structure $J_1$, since in this case $C=0$ and $B=f_0 g + B_{0,1}= -2 \Ric \in \langle g \rangle \oplus \Sigma^2_{0,1}$.
\end{enumerate}
The following result appears to be originally due to Vaisman \cite{Vaisman1982}*{Theorem 3.1}, see also \cite{Boyer1986}. 
\begin{corollary}
    A compact hyper-Hermitian $4$-manifold $M$ has non-negative total scalar curvature. Moreover, the total scalar curvature is zero if, and only if, the hyper-Hermitian structure is hyperK\"ahler. 
\end{corollary}
\begin{proof}
    Using Proposition \ref{proposition: to be named after}, we can equivalently express the scalar curvature (\ref{eq: scalar curvature}) as 
\begin{equation*}
{\mathrm{Scal}(g)} = 2\Big(\sum_{i=1}^3 \delta \big(J_i (a_i)\big)+g\big(J_1(a_1),J_2(a_2)\big)+g\big(J_1(a_1),J_3(a_3)\big)+g\big(J_2(a_2),J_3(a_3)\big)\Big).
\end{equation*}
From Definition \ref{definitionHKetal} it follows that, for a hyper-Hermitian $4$-manifold, 
\begin{equation*}
{\mathrm{Scal}(g)} = 6\big( \delta (J_1 (a_1))+
|a_1|^2\big).
\end{equation*}
Assuming $M^4$ is compact, and integrating the latter expression, Stokes' theorem gives the result.
\end{proof}
We conclude this section by relating the Ricci tensor to the symmetrisation of $\divergence(T^t)$, which is the highest-order term in (\ref{eq:cptsym}). This will be important when we investigate parabolic flows of $\SU(2)$-structures in \S\ref{section: flows}.

\begin{proposition}\label{prop: relation between ricci and div T}
    Denoting by $T$ the torsion of an $\SU(2)$-structure $\boldsymbol{\omega}$, the following holds
    \begin{equation}
        \mathrm{sym}\big(\divergence(T^t)\big) = - 2 \Ric(g) - \sum_{i=1}^3\mathcal{L}_{(J_i a_i)^\sharp}g + \lot (\boldsymbol{\om}),
        \label{eq: relation between ricci and div T}
    \end{equation}
where the lower-order terms are quadratic in the torsion forms $a_i$.
\end{proposition}
\begin{proof}
From Proposition \ref{prop: ricci in terms of torsion}, it follows that 
\begin{equation}
    \Ric(g) = -\frac{1}{2} \sum_{i=1}^3 g(da_i, \om_i) g + P_1(\pi^2_-(da_1)) +  P_2(\pi^2_-(da_2)) + P_3(\pi^2_-(da_3)) +  \lot (\boldsymbol{\om}).
\end{equation}
On the other hand from Proposition \ref{prop:divTintermsofa} and Remark \ref{rem: key remark} we see that 
\begin{align*}
  \mathrm{sym}(\divergence(T^t)) =\ &+\frac{1}{2} \sum_{i=1}^3 g(da_i,\om_i) g - P_1(\pi^2_-(da_1)) - P_2(\pi^2_-(da_2)) - P_3(\pi^2_-(da_3))\\ 
  &- 2 P_3(S^2_{0,2}(\nabla a_1) \diamond \om_2) + 2 P_2(S^2_{0,3}(\nabla a_1) \diamond \om_3)  \\
  & - 2 P_1(S^2_{0,3}(\nabla a_2) \diamond \om_3) + 2 P_3(S^2_{0,1}(\nabla a_2) \diamond \om_1) \\
  & - 2 P_2(S^2_{0,1}(\nabla a_3) \diamond \om_1) + 2 P_1(S^2_{0,2}(\nabla a_3) \diamond \om_2) + \lot (\boldsymbol{\om}).
\end{align*}
    Observe that, for any $1$-form $\alpha$ and vector field $X$, we have  
    $$\nabla_X(J \alpha) = (\nabla_X J) \alpha + J(\nabla_X \alpha) = (\nabla_X J) \alpha + (\nabla \alpha)(X, J \cdot)
    $$ 
    i.e. $\nabla_X(J \alpha) = (\nabla \alpha)(X, J\cdot) + \lot (\alpha)$. 
    Using  (\ref{eq: nabla of 1 form}) and (\ref{eq: identification1})-(\ref{eq: identification3}), we can express the latter in terms of its irreducible components as
\begin{align*}
    (\nabla \alpha)(\cdot, J_1 \cdot) = 
&+\frac{1}{4} g(d\alpha, \om_1) g 
-\frac{1}{2} P_1(\pi^2_-(d \alpha))
- P_2(S^2_{0,3}(\nabla \alpha) \diamond \om_3)
+ P_3(S^2_{0,2}(\nabla \alpha) \diamond \om_2)\\
&+\frac{1}{4} (\delta \alpha) \om_1 
+\frac{1}{4} g(d\alpha, \om_3) \om_2 
-\frac{1}{4} g(d\alpha, \om_2) \om_3
- S^2_{0,1}(\nabla \alpha) \diamond \om_1,
\end{align*}
    and similarly for $J= J_2, J_3$. Applying the above to compute $\sum_{i=1}^3\mathcal{L}_{(J_i a_i)^\sharp}g = \sum_{i=1}^3 \mathrm{sym}(\nabla (J_i a_i))$, we then have
\begin{align*}
    \sum_{i=1}^3\mathcal{L}_{(J_i a_i)^\sharp}g = 
&+\frac{1}{2} g(d a_1, \om_1) g 
- P_1(\pi^2_-(d a_1))
- 2P_2(S^2_{0,3}(\nabla a_1) \diamond \om_3)
+ 2P_3(S^2_{0,2}(\nabla a_1) \diamond \om_2)\\
&+\frac{1}{2} g(d a_2, \om_2) g 
- P_2(\pi^2_-(d a_2))
- 2P_3(S^2_{0,1}(\nabla a_2) \diamond \om_1)
+ 2P_1(S^2_{0,3}(\nabla a_2) \diamond \om_3)\\
&+\frac{1}{2} g(d a_3, \om_3) g 
- P_3(\pi^2_-(d a_3))
- 2P_1(S^2_{0,2}(\nabla a_3) \diamond \om_2)
+ 2P_2(S^2_{0,1}(\nabla a_3) \diamond \om_1)\\
&+ \lot (\boldsymbol{\om})\\
=\ &+\frac{1}{2} \sum_{i=1}^3 g(da_i,\om_i) g - P_1(\pi^2_-(da_1)) - P_2(\pi^2_-(da_2)) - P_3(\pi^2_-(da_3))\\ 
  &+ 2 P_3(S^2_{0,2}(\nabla a_1) \diamond \om_2) - 2 P_2(S^2_{0,3}(\nabla a_1) \diamond \om_3) \  \\
  & + 2 P_1(S^2_{0,3}(\nabla a_2) \diamond \om_3) - 2 P_3(S^2_{0,1}(\nabla a_2) \diamond \om_1)\  \\
  & + 2 P_2(S^2_{0,1}(\nabla a_3) \diamond \om_1) - 2 P_1(S^2_{0,2}(\nabla a_3) \diamond \om_2) +  \lot (\boldsymbol{\om}).
  \qedhere
\end{align*}
\end{proof}
This proposition shows that the symmetric part of the negative gradient flow of (\ref{def:energyfunctionalfulltorsion}) is not the Ricci flow to highest order, but rather there emerges another second-order term coming from the Lie derivative of $g$, by a certain combination of the torsion forms. As such, the short-time existence of negative gradient flow of (\ref{def:energyfunctionalfulltorsion}) is not immediate; we shall investigate this in \S\ref{section: flows}.

\subsubsection{Self-dual Weyl in terms of intrinsic torsion}


\begin{proposition}\label{proposition sd weyl curvature}
The self-dual Weyl curvature is explicitly given in terms of the intrinsic torsion by
    \begin{align*}
        W^+ = &-\frac{1}{24}\Big(2g(da_1,\om_1)-g(da_2,\om_2)-g(da_3,\om_3)\\ &+g(J_3a_1,a_2)-g(J_2a_1,a_3)-2g(J_1a_2,a_3)\Big)(2 \om_1 \otimes \om_1-\om_2 \otimes \om_2-\om_3 \otimes \om_3) \\
        &-\frac{1}{8}\Big(g(da_2,\om_2)-g(da_3,\om_3)+g(J_3a_1,a_2)+g(J_2a_1,a_3)\Big)(\om_2 \otimes \om_2 - \om_3 \otimes \om_3)\\
        &-\frac{1}{4}(g(da_2,\om_1)+g(J_1a_1,a_3))(\om_1 \odot \om_2)\\
        &-\frac{1}{4}(g(da_3,\om_2)+g(J_2a_2,a_1))(\om_2 \odot \om_3)\\
        &-\frac{1}{4}(g(da_1,\om_3)+g(J_3a_3,a_2))(\om_3 \odot \om_1),
    \end{align*}
where $\om_i \odot \om_j := \om_i \otimes \om_j + \om_j \otimes \om_i \in S^2_0(\Lm^2_+)$ for $i \neq j$.
\end{proposition}
\begin{proof}
The proof is again similar to that of the Ricci case. First, from Theorem \ref{theorem: classification of second-order invariants} and Corollary \ref{cor:bianchitypeidentity}, we know that there are only 6 possible second-order terms, namely $g(da_i,\om_j)$, with $i\leq j$. For the quadratic terms, we have 15 possibilities in view of (\ref{eq: quadratic terms in torsion}). So it is again a simple matter of determining the coefficient constants, which are obtained by verifying in suitable examples.
\end{proof}

\begin{corollary}
    If $J_1$ is an integrable almost complex structure, then $W^+$ belongs to a $3$-dimensional subspace of $S^2_0(\Lm^2_+)$, and it is explicitly given by
    \begin{align*}
        W^+ = &-\frac{1}{24}\Big(2g(da_1,\om_1)-g(da_2,\om_2)-g(da_3,\om_3)\\ &+g(J_3a_1,a_2)-g(J_2a_1,a_3)-2g(J_1a_2,a_3)\Big)(2 \om_1 \otimes \om_1-\om_2 \otimes \om_2-\om_3 \otimes \om_3) \\
        &-\frac{1}{4}(g(da_2,\om_1)+g(J_1a_1,a_3))(\om_1 \odot \om_2)\\
        &-\frac{1}{4}(g(da_1,\om_3)+g(J_3a_3,a_2))(\om_3 \odot \om_1).
    \end{align*}
    In particular, if $M$ is hyper-Hermitian, then $W^+=0$. On the other hand, if $\nabla J_1=0$ i.e. $(M,g,J_1,\om_1)$ is K\"ahler, then $W^+$ belongs to a $1$-dimensional subspace of $S^2_0(\Lm^2_+)$, and it is explicitly given by
    \begin{equation}
        W^+=\frac{{\mathrm{Scal}(g)}}{24}\Big(2\om_1 \otimes \om_1-\om_2 \otimes \om_2-\om_3 \otimes \om_3\Big).\label{eq: W+ in terms of scalar}
    \end{equation}
\end{corollary}
\begin{proof}
    Recall from Definition \ref{definitionHKetal} that $N_{J_1}=0$ means $J_2a_2=J_3a_3$, i.e. $a_2=J_1a_3$. To prove the first result, we take the exterior derivative of the relation $a_2 \w \om_2 =J_1 a_3 \w \om_2 = a_3 \w \om_3$ to get $da_2 \w \om_2 = da_3 \w \om_3$. The second identity we need is obtained by differentiating $d\om_1= -2 a_2 \w \om_3 = 2 a_3 \w \om_2$. Together with Proposition \ref{proposition sd weyl curvature}, this proves the first expression. The last assertion follows by  the same argument, replacing $J_1$ with $J_2$ and $J_3$. In the K\"ahler situation we have $a_2=a_3=0$, and hence $da_1 \w \om_2=da_1 \w \om_3=0$ from (\ref{eq: exterior derivative of om1}). Finally, from Proposition \ref{prop: ricci in terms of torsion},  $\mathrm{Scal}(g)=2\delta(J_1a_1)=-2g(da_1,\om_1)$, and this concludes the proof.
\end{proof}
In the K\"ahler case, the expression (\ref{eq: W+ in terms of scalar}) for $W^+$ appears to be originally due to Derdzi\'nski in \cite{Derdzinski1983}*{Proposition 2}, and the fact that hyper-Hermitian $4$-manifolds have $W^+=0$ appears to be originally due to Boyer in \cite{Boyer1988}*{Theorem 2}.

\begin{remark}
    While the anti-self-dual Weyl curvature $W^-$ cannot be computed from the $\SU(2)$ torsion forms, we can nonetheless express it as follows. If we choose a local $\SU(2)$ coframing $\{e^1,e^2,e^3,e^4\}$ on $M^4$, so that $\Lm^2$ can be trivialised by 
\[
\begin{matrix}
\om_1=e^{12}+e^{34}, & \om_2=e^{13}+e^{42}, & \om_3=e^{14}+e^{23}, \\
\sigma_1 = e^{12}-e^{34}, &
\sigma_2 = e^{13}-e^{42}, &
\sigma_3 = e^{14}-e^{23},  
\end{matrix}
\]
    then as in Proposition \ref{torsionproposition} we can define ``anti-self-dual torsion $1$-forms'' $b_i$ by 
$$
d\sigma_1= -b_2 \w \sigma_3 + b_3 \w \sigma_2,\qquad 
d\sigma_2= -b_3 \w \sigma_1 + b_1 \w \sigma_3,\qquad
d\sigma_3= -b_1 \w \sigma_2 + b_2 \w \sigma_1.
$$
    The anti-self-dual Weyl curvature $W^-$ can then be expressed as in Proposition \ref{proposition sd weyl curvature}, by replacing $(a_i,\om_i)$ with $(b_i,\sigma_i)$.
\end{remark}
We now proceed to our main section, where we establish short-time existence for a large class of $\SU(2)$-flows.

\section{Parabolicity of \texorpdfstring{$\SU(2)$}{}-flows}
\label{section: flows}

As we showed in Theorem \ref{theorem: classification of second-order invariants}, all the second-order invariants of an $\SU(2)$-structure are determined by the suitable irreducible components of $\nabla a_i$ and $W^-$. Since $W^-$ lies in the $5$-dimensional module $S^2_0(\R^3)$ and the latter does not occur in $\End(\R^4)/\mathfrak{su}(2) \cong 4\R \oplus 3\R^3$, it follows that we cannot evolve $\om_i$ using $W^-$ via a second order geometric flow, see also \S\ref{sec: comparison with g2 and spin7}. On the other hand, all the components of $\nabla a_i$ lie in the $\SU(2)$-modules $\R$ and $\R^3$, hence all diffeomorphism-invariant second order flows of $\SU(2)$-structures, to highest order, can be obtained by choosing $f_0,f_1,f_2,f_3$ among $\delta(a_i)$ and  $g(da_i,\om_j)$ ($i\leq j$), and $B_{0,1},B_{0,2},B_{0,3}$ among $\pi^2_-(da_i)$ and $S^2_{0,i}(\nabla a_j)$ (each suitably interpreted as a tensor of the right type via the identifications defined in \S\ref{sec: preliminary}). In this section we shall establish short-time existence for a large class of such flows. 

\subsection{Linearisations and principal symbols} \label{section: linearisations and principal symbols}
We begin by computing the principal symbols of the differential operators in Theorem \ref{theorem: classification of second-order invariants} that lead to $\SU(2)$-flows. For convenience we have summarised the result in Appendix \ref{appendix: principal symbol}. First we recall the basic definitions following \cite{Topping2006}. 

\begin{definition}
    Let $E$ and $F$ denote smooth vector bundles over $M$ and let $L:\Gamma(E) \to \Gamma(F)$ be a linear second order partial differential operator. The \emph{principal symbol} of $L$ is the vector bundle homomorphism $\sigma(L):\pi^*(E) \to \pi^*(F)$, where $\pi:T^*M \to M$, defined as follows: given $(x,\xi)\in T^*M$, $s \in \Gamma(E)$ such that $s(x)=v$ and $f\in C^\infty(M)$ with $df(x)=\xi$, then\footnote{The extra factor of `2' in the definition of the principal symbol is merely for convenience.}
\[
\sigma(L)(x,\xi)v :=  \lim_{t \to \infty} 2 t^{-2}e^{-tf}L(e^{tf}s)(x).
\]
    When, moreover, $E=F$ and  $\langle \cdot, \cdot \rangle$ denotes any fibre metric on $E$, we say that the flow
\[
\partial_t u_t = L(u_t)
\]
    is \emph{parabolic} if there exists $\lm>0$ such that
\[
\langle \sigma(L)(x,\xi)v, v \rangle \geq \lm |\xi|^2 |v|^2,
\quad
\forall 
(x,\xi)\in T^*M \qandq v\in \Gamma(E).
\]
    We extend the above notions to a non-linear partial differential operator $\mathcal{P}$ by requiring that its linearisation $L_{\mathcal{P}}$, at any given $v \in \Gamma(E)$, satisfies the above parabolicity condition.
\end{definition}

Here we are interested in non-linear differential operators of the form
\begin{align*}
    \mathcal{P} : \mathcal{M} &\to \Om^2(M),\\
    \boldsymbol{\om} &\mapsto \mathcal{P}(\boldsymbol{\om})
\end{align*}
where $\mathcal{M} \subset \Om^2(M) \times \Om^2(M) \times \Om^2(M)$ is a bundle whose fibres are isomorphic to the 13-dimensional space $\mathrm{GL}(4,\R)/\SU(2)$, and $\mathcal{P}$ 
is a linear combination of the invariants given in 
1. and 2. of Theorem \ref{theorem: classification of second-order invariants}, up to $\SU(2)$ isomorphism. More precisely, we are viewing the invariants of Theorem \ref{theorem: classification of second-order invariants} as second order differential operators acting on $\boldsymbol{\om}$.
The reason we emphasise `up to $\SU(2)$ isomorphism' above is because, for instance, the term $\pi^2_-(da_1)$ can be viewed a tensor in $\Lm^2_-, \Sigma^2_{0,1}, \Sigma^2_{0,2}$ or $\Sigma^2_{0,3}$ (which are all equivalent as $\SU(2)$-modules) but
depending on what type of tensor we identify it with, this leads to a different geometric flow, see also the last paragraph of \S \ref{section: more general flows} below.


\begin{proposition}
\label{proposition: principal symbols}
    The principal symbols of the non-linear second order differential operators $\nabla a_i$ at the point $(x,\xi) \in T^*M$ are given by
    \[
    \sigma(L_{\nabla a_1})(x,\xi)(A) = 
    -\xi \otimes *(\xi \w \tilde{B}_{0,1})-\xi \otimes *(J_2\xi \w \tilde{B}_{0,3})+\xi \otimes *(J_3\xi \w \tilde{B}_{0,2})+2f_1 \xi \otimes \xi - f_0 \xi \otimes J_1\xi,
    \]
    \[
    \sigma(L_{\nabla a_2})(x,\xi)(A) = 
    -\xi \otimes *(\xi \w \tilde{B}_{0,2})-\xi \otimes *(J_3\xi \w \tilde{B}_{0,1})+\xi \otimes *(J_1\xi \w \tilde{B}_{0,3})+2f_2 \xi \otimes \xi - f_0 \xi \otimes J_2\xi,
    \]
    \[
    \sigma(L_{\nabla a_3})(x,\xi)(A) = 
    -\xi \otimes *(\xi \w \tilde{B}_{0,3})-\xi \otimes *(J_1\xi \w \tilde{B}_{0,2})+\xi \otimes *(J_2\xi \w \tilde{B}_{0,1})+2f_3 \xi \otimes \xi - f_0 \xi \otimes J_3\xi,
    \]
    where $L_{\nabla a_i}$ denotes the linearisation of $\nabla a_i$, $A= f_0 g +\sum_{i=1}^3 B_{0,i}+ \sum_{i=1}^3 f_i \om_i$ is the variation of $\boldsymbol{\om}$ as given in (\ref{eq:evolutionom1})-(\ref{eq:evolutionom3}) and $\tilde{B}_{0,i}:= B_{0,i}\diamond \om_i \in \Lm^2_-(M)$. 

\end{proposition}
\begin{proof}
    From Proposition \ref{torsionproposition}, we have
    $$ \nabla a_1 = -\frac{1}{2} \nabla (*d\om_1+J_2(*d\om_3)-J_3(*d\om_2)).$$
    It follows that the linearisation of the second order operator $\nabla a_1$, at the point  $\boldsymbol{\om}=(\om_1,\om_2,\om_3)$, in the direction of $A$, as given by (\ref{eq:evolutionom1})-(\ref{eq:evolutionom3}) is 
    \begin{align*}
    L_{\nabla a_1}(A) = -\frac{1}{2}\Big(
    &\nabla * d(f_0 \om_1 -f_2 \om_3+f_3 \om_2 + \tilde{B}_{0,1})+
    \nabla J_2* d(f_0 \om_3 -f_1 \om_2+f_2 \om_1 + \tilde{B}_{0,3})-\\
    &\nabla J_3* d(f_0 \om_2 -f_3 \om_1+f_1 \om_3 + \tilde{B}_{0,2})
    \Big)+  \lot (A).
    \end{align*}
    Given an arbitrary function $f$, 
    \begin{align*}
   -2 L_{\nabla a_1}(e^{tf}A) = t^2 e^{tf} df \otimes \Big(
    &  \big(-f_0 d^{c_1}f +f_2 d^{c_3}f - f_3 d^{c_2}f + *(df \w \tilde{B}_{0,1})\big)\ +\\
     &J_2\big(-f_0 d^{c_3}f +f_1 d^{c_2}f-f_2 d^{c_1}f + *(df \w \tilde{B}_{0,3})\big)\ -\\
    & J_3\big(-f_0 d^{c_2}f +f_3 d^{c_1}f -f_1 d^{c_3}f + *(df\w \tilde{B}_{0,2})\big)
    \Big)+  \lot (t)\\
    = t^2 e^{tf} df \otimes \Big(
    &  \big(-f_0 d^{c_1}f +f_2 d^{c_3}f - f_3 d^{c_2}f + *(df \w \tilde{B}_{0,1})\big)\ +\\
     &\big(+f_0 d^{c_1}f -f_1 df-f_2 d^{c_3}f + *(d^{c_2}f \w \tilde{B}_{0,3})\big)\ +\\
    & \big(+f_0 d^{c_1}f +f_3 d^{c_2}f -f_1 df - *(d^{c_3}f\w \tilde{B}_{0,2})\big)
    \Big)+ \lot (t)\\
    = t^2 e^{tf} df \otimes \Big(
    &  f_0 d^{c_1}f  -2f_1 df + *(df \w \tilde{B}_{0,1}) + *(d^{c_2}f \w \tilde{B}_{0,3}) - *(d^{c_3}f\w \tilde{B}_{0,2}) \Big)\\ &+ \lot (t)
    \end{align*}
    where we used Lemma \ref{lemma:relations} and the fact that $J_j(\tilde{B}_{0,i})=\tilde{B}_{0,i}$, since $\tilde{B}_{0,i} \in \Om^2_-(M)$. The principal symbol of $L_{\nabla a_1}$ now follows immediately from the above, and an analogous computation applies to $L_{\nabla a_2}$ and $L_{\nabla a_3}$.
\end{proof}
We can now perform a similar computation to derive the principal symbols of each $\SU(2)$-invariant second-order differential operator $\delta a_i$, $g(da_i,\om_j)$, $S^2_{0,i}(\nabla a_j)$ and $\pi^2_-(d a_i)$. However, there is a simpler way to do this, since the principal symbols of the latter terms can be deduced from those of $\nabla a_i$, by an appropriate projection map. For instance, for any $1$-form $\alpha$, we have  $d\alpha (X,Y)=(\nabla_X \alpha)(Y)-(\nabla_Y \alpha)(X)$ and thus, from the above Proposition, we immediately deduce 
\begin{equation}
    \sigma(L_{d a_1})(x,\xi)(A) = 
    -\xi \w *(\xi \w \tilde{B}_{0,1})-\xi \w *(J_2\xi \w \tilde{B}_{0,3})+\xi \w *(J_3\xi \w \tilde{B}_{0,2}) - f_0 \xi \w J_1\xi,
    \label{eq: linearisation of da1}
\end{equation}
and similarly for $\sigma(L_{d a_2})$ and $\sigma(L_{d a_3})$. Unlike the skew-symmetrisation operator, the $\SU(2)$ projection maps do depend on $\boldsymbol{\om}$. However, since these maps are only of zeroth order in $\boldsymbol{\om}$, these do not affect the principal symbols when we consider their linearisations. Thus, their components in $\om_i$ and $\Lm^2_-(M)$ can be computed using Lemmas \ref{lemma:relations} and \ref{lemma: identifications}. For instance, one can express \eqref{eq: linearisation of da1} explicitly as
    \begin{align*}
        \sigma(L_{d a_1})(x,\xi)(A) = \
        &+\frac{1}{2}\big(f_0 |\xi|^2-g(\xi \w J_1\xi,\tilde{B}_{0,1})+g(\xi \w J_2\xi,\tilde{B}_{0,2})+g(\xi \w J_3\xi,\tilde{B}_{0,3})\big)\ \om_1 \\
        &-\frac{1}{2}\big(g(\xi \w J_2\xi, \tilde{B}_{0,1})+g(\xi \w J_1\xi, \tilde{B}_{0,2})\big)\ \om_2 \\
        &-\frac{1}{2}\big(g(\xi \w J_3\xi, \tilde{B}_{0,1})+g(\xi \w J_1\xi, \tilde{B}_{0,3})\big)\ \om_3 \\
        &+\frac{1}{2}|\xi|^2 \tilde{B}_{0,1}-\frac{1}{2} f_0 \sum_i g(\xi \w J_1 \xi, \sigma_i)\sigma_i \\
        &-\sum_{i,j,k}\epsilon_{ijk}\frac{1}{4}\big(g(\xi \w J_3\xi, \sigma_i)g(\tilde{B}_{0,2},\sigma_j)\big) \sigma_k+\sum_{i,j,k}\epsilon_{ijk}\frac{1}{4}\big(g(\xi \w J_2\xi, \sigma_i)g(\tilde{B}_{0,3},\sigma_j)\big) \sigma_k
    \end{align*}
and similarly for $\sigma(L_{d a_2})$ and $\sigma(L_{d a_3})$. In the same vein, 
one can use Lemma \ref{lemma:relations} to show that
\begin{equation}
        \sigma(L_{\delta a_1})(x,\xi)(A) = -2f_1 |\xi|^2-*(\xi \w J_2 \xi \w \tilde{B}_{0,3})+*(\xi \w J_3 \xi \w \tilde{B}_{0,2}),
\end{equation}
and similarly for $\sigma(L_{\delta a_2})$ and $\sigma(L_{\delta a_3})$. 
It now remains to compute the principal symbols of the terms $S^2_{0,i}(\nabla a_j)$. Thus, we need to apply the projection maps $S^2_{0,i}$ defined in \S\ref{sec: preliminary} to the principal symbols computed in Proposition \ref{proposition: principal symbols}, to do so we appeal to the following lemma: 
\begin{lemma}
\label{lemma: symmetric tensors to forms}
    Let $\alpha,\beta$ be $1$-forms, $B_{0,i}$ be arbitrary symmetric $2$-tensors in $S^2_{0,i}$, and $\tilde{B}_{0,i} := B_{0,i}\diamond \om_i \in \Lm^2_-$, then the following relations hold:
\begin{align*}
    S^2_{0,i}(\alpha \otimes J_i\beta) \diamond \om_i 
        &= \frac{1}{2} \pi^2_-(\alpha \w \beta),\\
    S^2_{0,i}\big(\alpha \otimes B_{0,i}(\alpha^\sharp,\cdot)\big) 
        &= \frac{1}{4} |\alpha|^2 B_{0,i},\\
    \pi^2_-\big(\alpha \w *(\alpha \w \tilde{B}_{0,i})\big) 
        &= -\frac{1}{2} |\alpha|^2 \tilde{B}_{0,i},\\
    *(\alpha \w \tilde{B}_{0,i}) 
        &= - B_{0,i}(J_i(\alpha^\sharp),\cdot) = -\alpha^\sharp \ip \tilde{B}_{0,i},\\
    8 S^2_{0,i}\big(\alpha \otimes B_{0,j}(\beta^\sharp,\cdot)\big)  \diamond \om_i 
        &= -\sum_{a,b,c} \epsilon_{abc} g(\tilde{B}_{0,j},\sigma_a)g(\alpha \w J_k\beta,\sigma_b)\sigma_c + 2 g(\alpha,J_k\beta) B_{0,j},\\
    8 S^2_{0,i}\big(\alpha \otimes B_{0,k}(\beta^\sharp,\cdot)\big)  \diamond \om_i 
        &= +\sum_{a,b,c} \epsilon_{abc} g(\tilde{B}_{0,k},\sigma_a)g(\alpha \w J_j\beta,\sigma_b)\sigma_c - 2 g(\alpha,J_j\beta) B_{0,k},
\end{align*}
where $(i,j,k) \sim (1,2,3)$ in the last two expressions.
    In particular, 
    \begin{gather*}
        8 S^2_{0,i}\big(\alpha \otimes B_{0,j}(\alpha^\sharp,\cdot)\big)  \diamond \om_i = -\sum_{a,b,c} \epsilon_{abc} g(\tilde{B}_{0,j},\sigma_a)g(\alpha \w J_k\alpha,\sigma_b)\sigma_c,\\
        8 S^2_{0,i}\big(\alpha \otimes B_{0,k}(\alpha^\sharp,\cdot)\big)  \diamond \om_i = +\sum_{a,b,c} \epsilon_{abc} g(\tilde{B}_{0,k},\sigma_a)g(\alpha \w J_j\alpha,\sigma_b)\sigma_c.
    \end{gather*}
\end{lemma}
\begin{proof}
    The maps involved are all $\SU(2)$-invariant, and hence by Schur's lemma it suffices to verify the above expressions for suitable choices of $\alpha,\beta$ and $B_{0,i}$. This is a long but straightforward calculation, using the above local expressions for $\om_i$ and $\sigma_i$.
\end{proof}

Since $S^2_{0,i}(M)\cong \Lm^2_-(M)$, describing the latter principal symbols as anti-self-dual forms will facilitate comparisons. 
Using Proposition \ref{proposition: principal symbols} and Lemma \ref{lemma: symmetric tensors to forms},
\begin{align*}
   \sigma(L_{S^2_{0,1}(\nabla a_1)})(x,\xi)(A) \diamond \om_1 =\ &\frac{1}{2}\pi^2_-\Big(
    \xi \w J_1(\xi^\sharp) \ip \tilde{B}_{0,1}- \xi \w J_3(\xi^\sharp) \ip \tilde{B}_{0,3}- \xi \w J_2(\xi^\sharp) \ip \tilde{B}_{0,2} -2 f_1 \xi \w J_1 \xi \Big)\\
    =\ 
    & -\frac{1}{2} f_1 \sum_i g(\xi \w J_1 \xi, \sigma_i) \sigma_i -\frac{1}{8}\sum_{i,j,k}\epsilon_{ijk}g(\xi \w J_2\xi, \sigma_j)g(\tilde{B}_{0,2},\sigma_i) \sigma_k \\
    \ &+\frac{1}{8}\sum_{i,j,k}\epsilon_{ijk}g(\xi \w J_1\xi, \sigma_j)g(\tilde{B}_{0,1},\sigma_i) \sigma_k-\frac{1}{8}\sum_{i,j,k}\epsilon_{ijk}g(\xi \w J_3\xi, \sigma_j)g(\tilde{B}_{0,3},\sigma_i) \sigma_k,
\end{align*}
\begin{align*}
   \sigma(L_{S^2_{0,2}(\nabla a_1)})(x,\xi)(A) \diamond \om_2 =\ &\frac{1}{2}\pi^2_-\Big(
    \xi \w J_2(\xi^\sharp) \ip \tilde{B}_{0,1}+ \xi \w \xi^\sharp \ip \tilde{B}_{0,3} + \xi \w J_1(\xi^\sharp) \ip \tilde{B}_{0,2} -2 f_1 \xi \w J_2 \xi 
    +f_0 \xi \w J_3 \xi\Big)\\    
    =\ &+\frac{1}{4}|\xi|^2\tilde{B}_{0,3} +\frac{1}{4} f_0 \sum_i g(\xi \w J_3 \xi,\sigma_i)\sigma_i - \frac{1}{2} f_1 \sum_i g(\xi \w J_2 \xi,\sigma_i)\sigma_i\\
    \ &+\frac{1}{8}\sum_{i,j,k}\epsilon_{ijk}g(\xi \w J_2\xi, \sigma_j)g(\tilde{B}_{0,1},\sigma_i) \sigma_k+\frac{1}{8}\sum_{i,j,k}\epsilon_{ijk}g(\xi \w J_1\xi, \sigma_j)g(\tilde{B}_{0,2},\sigma_i) \sigma_k,
\end{align*}
\begin{align*}
   \sigma(L_{S^2_{0,3}(\nabla a_1)})(x,\xi)(A) \diamond \om_3 = &\frac{1}{2}\pi^2_-\Big(
    \xi \w J_3(\xi^\sharp) \ip \tilde{B}_{0,1}+ \xi \w (J_1\xi^\sharp) \ip \tilde{B}_{0,3} -  \xi \w \xi^\sharp \ip \tilde{B}_{0,2} -2 f_1 \xi \w J_3 \xi 
    -f_0 \xi \w J_2 \xi\Big)\\ 
    = &-\frac{1}{4}|\xi|^2\tilde{B}_{0,2}-\frac{1}{4}f_0\sum_i g(\xi \w J_2 \xi,\sigma_i)\sigma_i -\frac{1}{2}f_1\sum_i g(\xi \w J_3\xi,\sigma_i)\sigma_i\\
    &+\frac{1}{8}\sum_{i,j,k}\epsilon_{ijk}g(\xi \w J_3\xi, \sigma_j)g(\tilde{B}_{0,1},\sigma_i) \sigma_k+\frac{1}{8}\sum_{i,j,k}\epsilon_{ijk}g(\xi \w J_1\xi, \sigma_j)g(\tilde{B}_{0,3},\sigma_i) \sigma_k.
\end{align*}
To obtain the remaining $9$ terms, it suffices to take cyclic permutations of the indices $(123)$ above; this is summarised in Appendix \ref{appendix: principal symbol}. 
Having computed the principal symbols, we next investigate which $\SU(2)$-flows are parabolic.

\subsection{The Ricci flow and DeTurck's trick}
\label{sec: deturck trick}

Let us briefly recall the proof of short-time existence of the Ricci flow, since we shall apply a very similar strategy. 
There are several excellent references on the topic, cf. \cites{Chow2004, Chow2006hamilton}, but we shall follow the exposition by Topping  \cite{Topping2006}. From now on the manifold $M^4$ will be assumed to be compact, in order to appeal to the standard result that parabolic flows exist for short time and are unique, given regular enough initial data.

From \cite{Topping2006}*{Chapter 5}, the principal symbol of the  linearisation of the Ricci map, $g \mapsto \Ric(g)$, is 
\begin{equation}
    -\sigma(L_{\Ric(g)})(x,\xi)(A) = |\xi|^2 B - \xi \otimes B(\xi^{\sharp},\cdot)- B(\xi^{\sharp},\cdot)\otimes \xi +4f_0(\xi \otimes \xi),
\end{equation}
where we recall that the variation of $g$ is given by (\ref{eq:evolutiong}). In particular, we see that the Ricci map is not elliptic, since its principal symbol is not invertible. From the work of Hamilton and DeTurck, we know that the kernel of the principal symbol is precisely generated by the diffeomorphism group. 
Using the results from Appendix \ref{appendix: principal symbol} and Lemma \ref{lemma: symmetric tensors to forms}, it is not hard to verify directly that
\[
S^2_{0,i}(\sigma(L_{\Ric(g)})(x,\xi)(A)) \diamond \om_i = - \sigma(L_{\pi^2_-(d a_i)})(x,\xi)(A)
\]
and
\[
\tr(\sigma(L_{\Ric(g)})(x,\xi)(A)) = -\frac{1}{2} \sigma(L_{g(da_1,\om_1)+g(da_2,\om_2)+g(da_3,\om_3)})(x,\xi)(A).
\]
In view of Proposition \ref{prop: ricci in terms of torsion} this was indeed to be expected, so this confirms that everything is consistent. We shall now recall DeTurck's trick to prove short-time existence of the Ricci flow.

We begin by considering the \emph{gravitation map} $G: S^2(M)\to S^2(M)$, defined by $$G(H)=H-\frac{1}{2}(\tr H)g,$$ 
where $H$ is an arbitrary positive definite symmetric tensor,
and its divergence by $\delta G(H)= -\tr (\nabla H) +\frac{1}{2}d(\tr H)$. Observe that $\delta G(H)$ is a first-order differential operator acting on $g$. It is shown in \cite{Topping2006}*{Chapter 5} that the linearisation of $\mathcal{L}_{(H^{-1}\delta G(H))^\sharp} g$ is given by $$\partial_t(\mathcal{L}_{(H^{-1}\delta G(H))^\sharp} g )= -\mathcal{L}_{(\delta G(B))^\sharp} g + 
\mathrm{l.o.t}(g)
$$ 
and hence its principal symbol is 
\[
\sigma(\mathcal{L}_{(\delta G(B))^\sharp}g)(x,\xi)(A)= -2 \xi \otimes B(\xi^{\sharp},\cdot)-2 B(\xi^{\sharp},\cdot)\otimes \xi + 8f_0(\xi \otimes \xi).
\]
It follows that
\begin{equation}
    -2\sigma(L_{\Ric(g)})(x,\xi)(A) - \sigma(\mathcal{L}_{(\delta G(B))^\sharp}g)(x,\xi)(A)
    = 2 |\xi|^2 B,
    \label{eq: principal symbol of ricci deturck}
\end{equation}
and we deduce that the Ricci-DeTurck flow
\begin{equation}
    \partial_t(g(t))= -2\Ric(g(t))+\mathcal{L}_{(H^{-1}\delta G(H))^\sharp} g(t) \label{eq: ricci deturck flow}
\end{equation}
is parabolic, admitting as such a unique solution for short time, given smooth initial data $g(0)=g_0$. 
We now let $X_t=-H^{-1}\delta G(H))^\sharp$ denote a time-dependent vector field (associated to the metric solving (\ref{eq: ricci deturck flow})) and $\{\varphi_t\}$ the associated $1$-parameter family of diffeomorphisms defined by
\begin{equation}
    X_t \circ \varphi_t(p) = \frac{d}{dt} \varphi_t(p),
    \qwithq
    \varphi_0=\mathrm{Id}.
    \label{eq: flow generated by vector field}
\end{equation}
Setting $\hat{g}(t):=\varphi_t^*(g(t))$, we then have 
\[
\partial_t(\hat{g}(t))=\varphi_t^*(-2\Ric(g)+\mathcal{L}_{(H^{-1}\delta G(H))^\sharp} g+\mathcal{L}_{X_t} g) = -2\Ric(\varphi_t^*g(t)),
\]
and this gives the desired solution to the Ricci flow starting at $\hat{g}(0)=g_0$.

Since our main goal is to classify parabolic flows of $\SU(2)$-structures, it is natural to consider those which induce the Ricci flow on the metric, to highest order. From Proposition \ref{prop: ricci in terms of torsion}, it follows that such a flow is necessarily of the form:
    \begin{align}
    \frac{d}{dt} \om_1 &= + \sum_{i=1}^3 g(da_i,\om_i) \om_1 + f_3 \om_2 - f_2 \om_3 + 2\pi^2_-(da_1) +\lot (\boldsymbol{\om})\\
    \frac{d}{dt} \om_2 &= - f_3 \om_1 + \sum_{i=1}^3 g(da_i,\om_i) \om_2 + f_1 \om_3 + 2\pi^2_-(da_2) +\lot (\boldsymbol{\om})\\
    \frac{d}{dt} \om_3 &= + f_2 \om_1 - f_1 \om_2 +  \sum_{i=1}^3 g(da_i,\om_i) \om_3 + 2\pi^2_-(da_3) +\lot (\boldsymbol{\om})
\end{align}
or, more compactly,
\begin{align*}
    \partial_t (\boldsymbol{\om}) &= \Big(
\sum_{i=1}^3 g(da_i,\om_i) g 
-2 \sum_{i=1}^3 P_i(\pi^2_-(da_i)) 
+\sum_{i=1}^3 f_i \om_i
\Big) \diamond \boldsymbol{\om} +\lot (\boldsymbol{\om})\\
    &= \Big(
-2\Ric(g) 
+\sum_{i=1}^3 f_i \om_i
\Big) \diamond \boldsymbol{\om} +\lot (\boldsymbol{\om})
\end{align*}
Observe that this $\SU(2)$-flow is under-determined, since the functions $f_1,f_2,f_3$ are not prescribed. Motivated by the results in \cite{LoubeauSaEarp2019} (see also Proposition \ref{prop: harmonic flow existence} below), it is natural to ask if the Ricci flow, when coupled with the harmonic flow, 
is parabolic. We investigate more general classes of such flows in the next two sections.

\subsection{Harmonic Ricci-like flows}
\label{section: ricci harmonic like flow}
    The main goal of this section is to establish the following result:
\begin{theorem}
\label{thm: short-time existence}
    Consider the flow of $\SU(2)$-structures given by 
\begin{equation}
\label{eq: 2 parameter family of flows}
\begin{gathered}
    \partial_t(\boldsymbol{\om}) 
    = \left(-2 \Ric(g) - \mu \sum_{i=1}^3\mathcal{L}_{(J_i a_i)^\sharp}g + \lm \sum_{i=1}^3\divergence(a_i) \om_i \right) 
    \diamond \boldsymbol{\om} + \lot (\boldsymbol{\om}),\\
    \boldsymbol{\om}(0)=\boldsymbol{\om}_0,
\end{gathered}
\end{equation}
where $\mu ,\lm \in \R$ and $\boldsymbol{\om}_0$ is a smooth $\SU(2)$-triple. Then there exists a smooth solution $\boldsymbol{\om}(t)$ to (\ref{eq: 2 parameter family of flows}) provided that $9-4\sqrt{5}<\lm +\frac{\mu}{2}<9+4\sqrt{5}$. Moreover, this solution is unique.
\end{theorem}

Before proving the above theorem, we first make a few observations. If we consider the flow given by
\begin{gather}
    \partial_t(\boldsymbol{\om}) 
    = \left(-2 \epsilon \Ric(g) - \mu \sum_{i=1}^3\mathcal{L}_{(J_i a_i)^\sharp}g + \lm \sum_{i=1}^3\divergence(a_i) \om_i \right) 
    \diamond \boldsymbol{\om},
    \label{eq: 3 parameter family of flows}
\end{gather}
then by rescaling we can always assume that $\epsilon = -1, 0$ or $1$. We disregard the case $\epsilon = -1 $, since backwards heat-type flows are never parabolic. Henceforth, we will ignore lower-order terms and omit writing $\lot (\boldsymbol{\om})$ altogether, as these do not play any role in the proof of Theorem \ref{thm: short-time existence}. The family of flows given by (\ref{eq: 3 parameter family of flows}) contains three notable cases:
\begin{enumerate}
    \item If $\epsilon=\mu=0$ and $\lm=2$, then the flow reduces (up to a constant factor) to the \emph{harmonic flow} of $\SU(2)$-structures cf. \cite{LoubeauSaEarp2019}. This is the negative gradient flow of (\ref{def:energyfunctionalfulltorsion}) whereby the metric is fixed, see Proposition \ref{proposition:firstvariation}; it is sometimes also called the \emph{isometric flow} in the literature\footnote{The reader might notice an extra  factor $2$ in our negative gradient flow, compared to \cites{Dwivedi2021, DwivediLoubeauSaEarp2021}; this is due to our different inner product convention on differential forms.}. 
    \item If $\epsilon= \lm =1$ and $\mu=0$, then it corresponds to the \emph{harmonic Ricci flow}. We should point out that, since the metric is also evolving, it is not immediately clear that this flow exists.
    \item If $\epsilon=\mu=1$ and $\lm=2$, then this corresponds to the negative \emph{Dirichlet gradient flow} of (\ref{def:energyfunctionalfulltorsion}); this follows from Proposition \ref{proposition:firstvariation}, \ref{prop:divTintermsofa} and \ref{prop: relation between ricci and div T}.
\end{enumerate}
The authors of  \cite{LoubeauSaEarp2019} establish  short-time existence and uniqueness of any harmonic flow of $H$-structures by showing that it can be interpreted as a harmonic map heat flow for certain equivariant maps. 
In the case of harmonic $\mathrm{G}_2$-structures considered in \cite{Dwivedi2021}, this was proven directly by showing that the flow is parabolic. In our setup here, we can also give a direct proof, as an immediate consequence of the results in \S\ref{section: linearisations and principal symbols}:
\begin{proposition}
\label{prop: harmonic flow existence}
    The harmonic $\SU(2)$-flow exists for short time and is unique, for given smooth initial data $\boldsymbol{\om}_0$.
\end{proposition}
\begin{proof}
    If we restrict to only isometric variations of $\boldsymbol{\om}$, i.e. we set $B=0$ and $C= \sum_{i=1}^3 f_i \om_i$ in (\ref{eq:evolutionom1})-(\ref{eq:evolutionom3}), then from the principal symbols listed in Appendix \ref{appendix: principal symbol} we see that 
    $\sigma(L_{\divergence (a_i)})(x,\xi)(C)= 2f_i|\xi|^2$, for $i=1,2,3$. 
    It follows that $$\langle \sigma(L_{\sum_{i=1}^3 \divergence (a_i) \om_i})(x,\xi)(C), C \rangle = 4|\xi|^2\sum_{i=1}^3 f_i^2 = 2 |\xi|^2|C|^2.$$ Thus, the harmonic flow is immediately seen to be parabolic and the result follows.
\end{proof}
It worth pointing out that we did not have to rely on DeTurck's trick in the above proof, as the flow was already parabolic. With these observations out of the way, we now begin the proof of Theorem \ref{thm: short-time existence}. 

Consider the modified flow of (\ref{eq: 2 parameter family of flows}) given by 
\begin{gather}
    \partial_t(\boldsymbol{\om}) = (-2 \Ric(g) - \mu \sum_{i=1}^3\mathcal{L}_{(J_i a_i)^\sharp}g + \lm \sum_{i=1}^3\divergence(a_i) \om_i ) \diamond \boldsymbol{\om} + \mathcal{L}_{(H^{-1}\delta G(H))^\sharp + \mu \sum_{i=1}^3 (J_i a_i)^\sharp} \boldsymbol{\om}. 
    \label{eq: modified flow}
\end{gather}
From Lemma \ref{lemma: lie derivative of om}, we can express the latter equivalently as
\begin{align}
    \partial_t(\boldsymbol{\om}) 
    &= \big(-2 \Ric(g) +  \mathcal{L}_{(H^{-1}\delta G(H))^\sharp} g + 
    \lm \sum_{i=1}^3\divergence(a_i) \om_i
    -d(H^{-1}\delta G(H) + \mu \sum_{i=1}^3 J_i a_i)
    \big) \diamond \boldsymbol{\om}.\label{eq: we'll see}
\end{align}
For convenience we shall write
\begin{align*}
    \mathcal{P}(\boldsymbol{\om}) &= \mathcal{P}_{sym}(\boldsymbol{\om}) + \mathcal{P}_{alt}(\boldsymbol{\om}) \\
    &:= \Big(-2 \Ric(g) +  \mathcal{L}_{(H^{-1}\delta G(H))^\sharp} g \Big)+ \Big(
    \lm \sum_{i=1}^3\divergence(a_i) \om_i  - d \big(H^{-1}\delta G(H) + \mu \sum_{i=1}^3 (J_i a_i)\big)\Big),
\end{align*}
where $\mathcal{P}_{sym}$ and $\mathcal{P}_{alt}$ denote the symmetric and skew-symmetric parts respectively. 
In particular, observe that $\mathcal{P}_{sym}$ induces precisely the Ricci-DeTurck flow (\ref{eq: ricci deturck flow}),  hence  it follows from (\ref{eq: principal symbol of ricci deturck}) that
\begin{equation}
    \langle 
\sigma(L_{\mathcal{P}})(x,\xi)(A)
, B \rangle = 2 |\xi|^2 |B|^2. \label{eq: symbol symmetric}
\end{equation}
The choice of the modified flow (\ref{eq: modified flow}) was precisely intended to obtain (\ref{eq: symbol symmetric}).
Thus, it only remains to consider the skew-symmetric part.


First note that, to highest order, $d(J_i a_i) = 2 \pi^2\big(\nabla a_i (\cdot, J_i\cdot)\big)$, where $\pi^2$ denotes the projection of a 2-tensor on $\Lm^2(M)$. On the other hand, using (\ref{eq: nabla of 1 form}),
\[
\nabla a_1 (\cdot, J_1\cdot) = \frac{1}{4} (\delta a_1) \om_1 + \frac{1}{4} g(da_1, \om_3) \om_2 - \frac{1}{4} g(da_1, \om_2) \om_3 +\cdots,
\]
where $\cdots$ consists of the components of $\nabla a_1 (\cdot, J_1\cdot)$ in the complement of $\langle \om_1, \om_2, \om_3\rangle$ in $T^* \otimes T^*$; likewise we get analogous expressions for $\nabla a_2 (\cdot, J_2\cdot)$ and $\nabla a_3 (\cdot, J_3\cdot)$. Using the above together with Appendix \ref{appendix: principal symbol}, and the fact that $\Lm^2_-(M)\subset\ker(\diamond \boldsymbol{\om})$, we deduce that
\begin{align*}
  2  \sigma(L_{\sum_{i=1}^3 d (J_ia_i)})(x,\xi)(A) \diamond \boldsymbol{\om} = \
  &\Big(\sigma(L_{ \delta a_1 - g(da_2, \om_3) + g(da_3, \om_2) })(x,\xi)(A) \om_1\\
  &+\sigma(L_{ \delta a_2 - g(da_3, \om_1) + g(da_1, \om_3) })(x,\xi)(A) \om_2\\
  &+\sigma(L_{ \delta a_3 - g(da_1, \om_2) + g(da_2, \om_1) })(x,\xi)(A) \om_3 \Big) \diamond \boldsymbol{\om},\\
  = \ &\Big(\sum_{i=1}^3 \sigma(L_{ \delta a_i})(x,\xi)(A) \om_i \Big) \diamond \boldsymbol{\om}.
\end{align*}

Next we need to compute the self-dual components of the principal symbol of $d \big(H^{-1}\delta G(H)\big)$, whose linearisation, we recall, is given by $-\delta G (B) = \tr (\nabla B) -\frac{1}{2}d(\tr B)$. Thus, we have that
\begin{align*}
    \sigma\big(d (\delta G (B))\big)(x,\xi) = &-2 \xi \w \Big(f_0 g(\xi^\sharp,\cdot) +\sum_{i=1}^3 \tr (\xi \otimes B_{0,i}) \Big)\\
    = &-2 \xi \w (\sum_{i=1}^3  B_{0,i}(\xi^\sharp,\cdot))\\
    = &+*(\xi \w J_3 \xi \w \tilde{B}_{0,2}  -  \xi \w J_2 \xi \w \tilde{B}_{0,3} )\ \om_1 \\
    &+ 
    *(\xi \w J_1 \xi \w \tilde{B}_{0,3}  -  \xi \w J_3 \xi \w \tilde{B}_{0,1} )\ \om_2 \\
    &+ 
    *(\xi \w J_2 \xi \w \tilde{B}_{0,1}  -  \xi \w J_1 \xi \w \tilde{B}_{0,2} )\ \om_3 + \cdots,
\end{align*}
where $\cdots$ denotes the anti-self-dual terms, and for the last equality we used 
$$\xi \w B_{0,i}(\xi^\sharp,\cdot) \w \om_j = \xi \w J_j(J_i\xi) \w \tilde{B}_{0,i},
\qforq i,j=1,2,3.$$The latter identity is obtained as follows: since $\tilde{B}_{0,i} \w \om_j=0$, it follows that $(X \ip \tilde{B}_{0,i}) \w \om_j = - \tilde{B}_{0,i} \w (X \ip \om_j) =   \tilde{B}_{0,i} \w J_j(X^\flat)$, for any vector field $X$. The result follows from the fact that $B_{0,i}(\cdot,\cdot)= - \tilde{B}_{0,i}(J_i\cdot,\cdot)$, cf. Lemma \ref{lemma: identifications}.
Using the principal symbols  in Appendix \ref{appendix: principal symbol}, we can summarise the result of the above computations into
\begin{align*}
    \sigma(L_{\mathcal{P}_{alt}})(x,\xi)(A)  =  \ 
    &(2 \lm + \mu) f_1 |\xi|^2 +(\lambda +\frac{\mu}{2}-1)\big(*(\xi \w J_2 \xi \w \tilde{B}_{0,3} - \xi \w J_3 \xi \w \tilde{B}_{0,2}  )\big)\om_1 + \\
    &(2 \lm + \mu) f_2 |\xi|^2 +(\lambda +\frac{\mu}{2}-1)\big(*(\xi \w J_3 \xi \w \tilde{B}_{0,1} - \xi \w J_1 \xi \w \tilde{B}_{0,3}  )\big)\om_2 +\\
    &(2 \lm + \mu) f_3 |\xi|^2 +(\lambda +\frac{\mu}{2}-1)\big(*(\xi \w J_1 \xi \w \tilde{B}_{0,2} - \xi \w J_2 \xi \w \tilde{B}_{0,1}  )\big)\om_3 + \cdots,
\end{align*}
where $\cdots$ consists of the terms in $\Lm^2_-(M)$. Together with (\ref{eq: symbol symmetric}) we now have that 
\begin{align*}
\langle 
\sigma(L_{\mathcal{P}})(x,\xi)(A)
, A \rangle =\ &2 |\xi|^2 |B|^2 + 2|\xi|^2 (2 \lm + \mu) (f_1^2+f_2^2+f_3^2) \ + \\
&2 f_1 (\lambda +\frac{\mu}{2}-1)\big(*(\xi \w J_2 \xi \w \tilde{B}_{0,3} - \xi \w J_3 \xi \w \tilde{B}_{0,2}  )\big)\ +\\
&2 f_2 (\lambda +\frac{\mu}{2}-1)\big(*(\xi \w J_3 \xi \w \tilde{B}_{0,1} - \xi \w J_1 \xi \w \tilde{B}_{0,3}  )\big)\ +\\
&2 f_3 (\lambda +\frac{\mu}{2}-1)\big(*(\xi \w J_1 \xi \w \tilde{B}_{0,2} - \xi \w J_2 \xi \w \tilde{B}_{0,1}  )\big)\\
=\ &2 |\xi|^2 |B|^2 + 2|\xi|^2 (2 \lm + \mu) (f_1^2+f_2^2+f_3^2) \ + \\
&2  (\lambda +\frac{\mu}{2}-1)\big(*( \tilde{B}_{0,1} \w (f_2 \xi \w J_3 \xi - f_3 \xi \w J_2 \xi)  )\big)\ +\\
&2 (\lambda +\frac{\mu}{2}-1)\big(*( \tilde{B}_{0,2} \w (f_3 \xi \w J_1 \xi - f_1\xi \w J_3 \xi) )\big)\ +\\
&2 (\lambda +\frac{\mu}{2}-1)\big(*( \tilde{B}_{0,3} \w (f_1 \xi \w J_2\xi - f_2 \xi \w J_1 \xi) )\big).
\end{align*}

We see already that if $\lm + \frac{\mu}{2}=1$ then the flow is parabolic (in particular, this applies to the harmonic Ricci flow). Our goal is to find an optimal range of values for $\lm,\mu$ so that the flow (\ref{eq: we'll see}) stays parabolic.
Since $\SU(2)$ acts transitively on $S^3\subset \R^4$, we can identify pointwise $(\om_1,\om_2,\om_3)$ with the standard triple on $\R^4$ and $\xi$ with $|\xi| dx_{1}$. Writing $\tilde{B}_{0,1}=\sum_i b_i \sigma_i$, where $\sigma_i$ denote the standard anti-self-dual $2$-forms on $\R^4$, we have
\begin{align*}
    *\big( \tilde{B}_{0,1} \w (f_2 \xi \w J_3 \xi - f_3 \xi \w J_2 \xi) \big) &= *\big( |\xi|^2 (b_3 f_2 - b_2 f_3) dx_{1234} \big) \\
    &= |\xi|^2 (b_3 f_2 - b_2 f_3)\\
    &\geq -|\xi|^2 \sqrt{b_2^2+b_3^2} \sqrt{f_2^2+f_3^2} \\
    &\geq -|\xi|^2 \frac{|\tilde{B}_{0,1}|}{\sqrt{2}} \sqrt{f_2^2+f_3^2}\\
    &\geq -\frac{1}{2} |\xi|^2 \Big(\frac{|\tilde{B}_{0,1}|^2}{2\epsilon}+ \epsilon (f_2^2+f_3^2)\Big)
\end{align*}
where we used Cauchy-Schwarz for the first inequality and Young's for the last one: $ab \leq (a^2/\epsilon+\epsilon b^2)/2$, for any $\epsilon>0$. An analogous argument holds for the other two terms, hence, using $|{B}_{0,i}|^2=2|\tilde{B}_{0,i}|^2$,
\begin{align}
\langle 
\sigma(L_{\mathcal{P}})(x,\xi)(A)
, A \rangle 
\geq \ &2 |\xi|^2 |B|^2 + 2|\xi|^2 (2 \lm + \mu) (f_1^2+f_2^2+f_3^2)\nonumber  \\
&-|\xi|^2  |\lambda +\frac{\mu}{2}-1|\Big(
\frac{|\tilde{B}_{0,1}|^2+|\tilde{B}_{0,2}|^2+|\tilde{B}_{0,3}|^2}{2\epsilon}+ 2\epsilon (f_1^2+f_2^2+f_3^2)
\Big)\label{eq: estimating principal sym} \\
= \  &8 |\xi|^2f_0^2 + \frac{1}{2}|\xi|^2(8-\frac{1}{\epsilon}|\lm+\frac{\mu}{2}-1|)\sum_{i=1}^3|\tilde{B}_{0,i}|^2 + 2|\xi|^2(2\lm+\mu-\epsilon|\lm+\frac{\mu}{2}-1|) \sum_{i=1}^3 f_i^2.\nonumber
\end{align}
We see that the modified flow (\ref{eq: we'll see}) is parabolic for $\lm, \mu$ such that 
$|x-1|<8 \epsilon$ and $|x-1|<2x/\epsilon$, where $x:=\lm+\frac{\mu}{2}$. Observe that we still have the freedom to choose $\epsilon>0$ so as to maximise the range of values of $x$ satisfying these inequalities. It is not hard to see that the minimum value of $x$ is attained when $2 x/\epsilon = 8 \epsilon = 1-x \Leftrightarrow \epsilon=-1+\sqrt{5}/2$, and the maximum when $2 x/\epsilon = 8 \epsilon = x-1\Leftrightarrow 
 \epsilon=1+\sqrt{5}/2$. The corresponding range of values of $x$ is precisely as given in Theorem \ref{thm: short-time existence}.

To conclude existence of the flow  (\ref{eq: 2 parameter family of flows}), we argue as in DeTurck's trick: we define a $1$-parameter family of diffeomorphisms $\{\varphi_t\}$ by (\ref{eq: flow generated by vector field}), where $$X_t=-\big(H^{-1}\delta G(H)\big)^\sharp - \mu \sum_{i=1}^3 (J_i a_i)^\sharp$$ and $\vp_0=\mathrm{Id}$. Setting $\hat{\boldsymbol{\om}}(t):= \varphi^*_t(\boldsymbol{\om}(t))$ then gives a short-time solution to (\ref{eq: 2 parameter family of flows}), ie. the existence part of Theorem \ref{thm: short-time existence}.

The proof of uniqueness is also similar to the Ricci case, but we need to proceed in two steps (see Remark \ref{rem: 2 diffeomorphisms} below). The strategy is similar to the one employed in \cite{Dwivedi2023} for the $\mathrm{G}_2$ case. Suppose that there are two solutions $\boldsymbol{\om}_1(t),\boldsymbol{\om}_2(t)$ to (\ref{eq: 2 parameter family of flows}), with $\boldsymbol{\om}_1(0)= \boldsymbol{\om}_2(0)= \boldsymbol{\om}(0)$. We define vector fields
\begin{equation}
Y_u(t) := \mu \sum_{i=1}^3 \big(
(J_i)_u (a_i)_u\big)^\sharp
\end{equation}
associated to $\boldsymbol{\om}_u(t)$, for $u=1,2$. It should be understood that $(a_i)_u$ correspond to the torsion forms associated to $\boldsymbol{\om}_u(t)$, and similarly for the almost complex structures $(J_i)_u$. We denote by $\{f_u(t)\}$ the $1$-parameter family of diffeomorphisms associated to $\{Y_u(t)\}$, with $f_u(0)=\mathrm{Id}$. Let $\bar{\boldsymbol{\om}}_u(t):=f_u(t)^*{\boldsymbol{\om}}_u(t)$, then we have
\begin{align*}
    \partial_t(\bar{\boldsymbol{\om}}_u(t)) 
    &= \  f_u(t)^* \left(
    (-2 \Ric(g_u) - \mu \sum_{i=1}^3\mathcal{L}_{((J_i)_u (a_i)_u)^\sharp}g_u + \lm \sum_{i=1}^3\divergence((a_i)_u) (\om_i)_u ) \diamond \boldsymbol{\om}_u 
    + \mathcal{L}_{Y_u(t)} \boldsymbol{\om}_u \right) \\
    &= \ f_u(t)^* 
    \left(
    (-2 \Ric(g_u) + \lm \sum_{i=1}^3\divergence((a_i)_u) (\om_i)_u 
    - d(\mu \sum_{i=1}^3 (J_i)_u (a_i)_u)) 
    \diamond \boldsymbol{\om}_u
    \right)\\
    &= \ \left(-2 \Ric(\bar{g}_u) + \lm \sum_{i=1}^3\divergence((\bar{a}_i)_u) (\bar{\om}_i)_u 
- d(\mu \sum_{i=1}^3 (\bar{J}_i)_u (\bar{a}_i)_u)
    \right) \diamond \bar{\boldsymbol{\om}}_u,
\end{align*}
where we used Lemma \ref{lemma: lie derivative of om} for the second equality. Observe that the symmetric part of the latter flow is precisely the Ricci flow. Thus, following \cite{Chow2004}*{\S 4.1}, we can define another $1$-parameter of diffeomorphisms $\{\phi_u(t)\}$ by the harmonic map heat flow
\begin{equation*}
    \partial_t \phi_u(t) = \Delta_{\bar{g}_u(t),H} \phi_u(t),
    \qwithq
    \phi_1(0)=\phi_2(0)=\mathrm{Id},
\end{equation*}
where we recall that $H$ denotes a fixed background metric. Letting $\hat{\boldsymbol{\om}}_u(t):= (\phi_u(t)^{-1})^*\bar{\boldsymbol{\om}}_u(t)$, and computing as before,
\begin{align*}
    \partial_t(\hat{\boldsymbol{\om}}_u(t)) &= \ (\phi_u(t)^{-1})^* \Big(
    \left(
    -2 \Ric(\bar{g}_u) + \lambda \sum_{i=1}^3 \divergence((\bar{a}_i)_u) (\bar{\om}_i)_u 
    - d\left(\mu \sum_{i=1}^3 (\bar{J}_i)_u (\bar{a}_i)_u\right)
    \right) \diamond \bar{\boldsymbol{\om}}_u \Big)
    + \mathcal{L}_{(H^{-1}\hat{\delta} G(H))^\sharp} \hat{\boldsymbol{\om}}_u \\
    &= \ 
    \left(
    -2 \Ric(\hat{g}_u) + \lambda \sum_{i=1}^3 \divergence((\hat{a}_i)_u) (\hat{\om}_i)_u 
    - d\left(\mu \sum_{i=1}^3 (\hat{J}_i)_u (\hat{a}_i)_u\right)
    \right) \diamond \hat{\boldsymbol{\om}}_u
    + \mathcal{L}_{(H^{-1}\hat{\delta} G(H))^\sharp} \hat{\boldsymbol{\om}}_u \\
    &= \ 
    \left(
    -2 \Ric(\hat{g}_u) 
    + \mathcal{L}_{(H^{-1}\hat{\delta} G(H))^\sharp} \hat{g}_u
    + \lambda \sum_{i=1}^3 \divergence((\hat{a}_i)_u) (\hat{\om}_i)_u 
    - d\left(H^{-1}\hat{\delta} G(H) + \mu \sum_{i=1}^3 (\hat{J}_i)_u (\hat{a}_i)_u\right)
    \right) \diamond \hat{\boldsymbol{\om}}_u,
\end{align*}
where the first equality follows from the relation between the harmonic map flow and the Ricci flow, see \cite{Chow2004}*{Lemma 3.27},
and the last one is from Lemma \ref{eq: lie derivative of om}. 
Thus, we see that $\hat{\boldsymbol{\om}}_1(t)$ and $\hat{\boldsymbol{\om}}_2(t)$ are both solutions to (\ref{eq: we'll see}) with the same initial data. But we already established above that (\ref{eq: we'll see}) is parabolic, hence by uniqueness $\hat{\boldsymbol{\om}}_1(t)=\hat{\boldsymbol{\om}}_2(t)$. Since $\hat{g}_1(t)=\hat{g}_2(t)$, it also follows that $\phi_1(t)=\phi_2(t)$ from the relation between the harmonic map heat flow and the Ricci flow, see \cite{Chow2004}*{\S 4.4}, and so $\bar{\boldsymbol{\om}}_1(t)=\bar{\boldsymbol{\om}}_2(t)$. Thus, to conclude, we only need to show that $f_1(t)=f_2(t)$. Since
\begin{align*}
    0 &=\ \partial_{t}( f_u(t) \circ f_u(t)^{-1}(p))\\ 
    &=\ 
    \partial_{t}( f_u(t)) \circ f_u(t)^{-1}(p)
    + 
    (f_u(t))_*
    \partial_{t}( f_u(t)^{-1}(p)) \\
    &=\
    Y(t)(p) + (f_u(t))_*
    \partial_{t}( f_u(t)^{-1}(p)),
\end{align*}
we have
\begin{equation*}
    \partial_{t}( f_u(t)^{-1}(p))  = - (f_u(t)^{-1})_* Y(t)(p) = - \mu \sum_{i=1}^3 ((\bar{J}_i)_u(\bar{a}_i)_u)^\sharp \circ f_u(t)^{-1}(p)
\end{equation*}
i.e. $f_u(t)^{-1}$ is the diffeomorphism generated by the vector field $-\mu((\bar{J}_i)_u(\bar{a}_i)_u)^\sharp$.
As $\bar{\boldsymbol{\om}}_1(t)=\bar{\boldsymbol{\om}}_2(t)$ and $f_1(0)=f_2(0)=\mathrm{Id}$, the result follows by uniqueness. This concludes the proof of Theorem \ref{thm: short-time existence}.

\begin{remark}
\label{rem: 2 diffeomorphisms}
  The point of defining two diffeomorphisms $f_u(t)$ and $\phi_u(t)$ in succession, rather than just one, is that we relied on the parabolicity of the harmonic map heat flow, cf. \cite{Eells1964}, and also its relation to Ricci flow, cf. \cite{Chow2004}*{Lemma 3.27}. Otherwise one would have to demonstrate parabolicity of a different flow of maps and establish a corresponding relation. This parallels Remark 6.82 of \cite{Dwivedi2023}.
\end{remark}

\subsection{More general \texorpdfstring{$\SU(2)$}{}-flows}
\label{section: more general flows}

We are interested in $\SU(2)$-flows which evolve, to second order (highest), by the invariants given in Theorem \ref{theorem: classification of second-order invariants}. 
This means that each $f_i$ in (\ref{eq:evolutionom1})-(\ref{eq:evolutionom3}) can be chosen as a linear combination of the $9$ functions from Theorem \ref{theorem: classification of second-order invariants}--(1), and each $B_{0,i}$ as a linear combination of the 12 terms in the space (isomorphic to) $\mathfrak{su}(2)\cong \R^3$ from Theorem \ref{theorem: classification of second-order invariants}--(2). This gives a $(4\times 9 + 3 \times 12) = 72$-parameter family of $\SU(2)$-flows. Deriving precise conditions for parabolicity in such generality is much too difficult; in this section we explain how one can nonetheless extend our results to a wider class of flows, by suitable modifications.

\paragraph{Modification by Lie derivative terms.} 
A proof along the lines of \S\ref{section: ricci harmonic like flow} applies more generally to $\SU(2)$-flows of the form
\begin{gather}
    \partial_t(\boldsymbol{\om}) = \big(-2  \Ric(g) -\sum_{i=0}^3\sum_{j=1}^3\mu_{ij}\mathcal{L}_{(J_i a_j)^\sharp}g +\sum_{j=1}^3\lm_j \divergence(a_j) \om_j \big) \diamond \boldsymbol{\om},
    \label{eq: more general flow}
\end{gather}
where $\mu_{ij},\lm_j \in \R$, and we set $J_0:=\mathrm{Id}$ for summation convenience.
Denoting, for clarity, the vector field
$$
W:= {(H^{-1}\delta G(H))^\sharp + \sum_{i=0}^3\sum_{j=1}^3\mu_{ij}\mathcal{L}_{(J_i a_j)^\sharp}},
$$
and again using Lemma \ref{lemma: lie derivative of om} for the second equality, we can consider the modified flow given by
\begin{align*}
    \partial_t(\boldsymbol{\om}) = \big(&-2  \Ric(g) -\sum_{i=0}^3\sum_{j=1}^3\mu_{ij}\mathcal{L}_{(J_i a_j)^\sharp}g +\sum_{j=1}^3\lm_j \divergence(a_j) \om_j \big) \diamond \boldsymbol{\om} 
    +  \mathcal{L}_{W} \boldsymbol{\om}\\
    = \big(&-2 \Ric(g) +  \mathcal{L}_{(H^{-1}\delta G(H))^\sharp} g + 
    \sum_{j=1}^3\lm_j \divergence(a_j) \om_j
    -d(H^{-1}\delta G(H) + \sum_{i=0}^3\sum_{j=1}^3\mu_{ij} J_i a_j)
    \big) \diamond \boldsymbol{\om},
\end{align*}
whose symmetric part is precisely given by the Ricci DeTurck flow. Thus, it suffices to compute the principal symbol of the skew-symmetric part and extract a parabolicity range for the parameters. The calculation at this level generality is still rather cumbersome, but the reader might have specific cases of interest which could be easier to work out. It should be clear that if the $\mu_{ij}$ and $\lm_i$ are sufficiently close to the flow of Theorem \ref{thm: short-time existence}, then the modified flow of (\ref{eq: more general flow}) will also exist, as invertibility of the principal symbol is an open condition. The solution to the original flow is obtained by pulling back by the corresponding diffeomorphism as we did before. The proof of uniqueness also extends to this situation. 

\paragraph{Negative gradient flow of the weighted Dirichlet energy functional.}
In \S\ref{thm: short-time existence}, we established the existence and uniqueness of the negative gradient flow of the functional (\ref{def:energyfunctionalfulltorsion}). We shall now show how one can extend this to the weighted functional $E_{\boldsymbol{\lm}}(\boldsymbol{\om})$ defined by (\ref{eq: weighted energy functional}). From Proposition \ref{prop: first variation a1 a1} and Remark \ref{rem: key remark}, it is easy to see that its first variation formula is
\begin{align*}
    \frac{d}{dt} E_{\boldsymbol{\lm}}(\boldsymbol{\om}) 
    = 
    \int_M g \Big( 
    B, 
    &\lm_1 \big(-\frac{1}{2}g(da_1, \om_1)g
    +P_1(
    \pi^2_-(d a_1))
    -2P_2(
    S^2_{0,3}(\nabla a_1) \diamond \om_3 )
    +2P_3(
    S^2_{0,2}(\nabla a_1) \diamond \om_2 )
    \big)\\
    +&\lm_2 \big(-\frac{1}{2}g(da_2, \om_2)g
    +P_2(
    \pi^2_-(d a_2))
    -2P_3(
    S^2_{0,1}(\nabla a_2) \diamond \om_1 )
    +2P_1(
    S^2_{0,3}(\nabla a_2) \diamond \om_3 )
    \big)\\
    +&\lm_3 \big(-\frac{1}{2}g(da_3, \om_3)g
    +P_3(
    \pi^2_-(d a_3))
    -2P_1(
    S^2_{0,2}(\nabla a_3) \diamond \om_2 )
    +2P_2(
    S^2_{0,1}(\nabla a_3) \diamond \om_1 )
    \big)\Big)\vol\\
    - 2&\sum_{i=1}^3 \int_M g \Big(C, \lm_i\divergence(a_i)\om_i)
    \Big)\vol
    + \mathrm{l.o.t}(\boldsymbol{\om}).
\end{align*}
The corresponding negative gradient flow is, to highest order, given by
\begin{align}
    \partial_t(\boldsymbol{\om}) =\ \Big(
    &-2 \Ric(g) - \mu \sum_{i=1}^3\mathcal{L}_{(J_i a_i)^\sharp}g + \lm \sum_{i=1}^3\divergence(a_i) \om_i +
    \sum_{i=1}^3\big(\frac{\lm_i}{2}+\frac{\mu}{2}-1\big)g(da_i,\om_i)g \nonumber \\
    &+
    P_1\big(
    (2-\lm_1-\mu) \pi^2_-(da_1)+
    2(\lm_3-\mu) S^2_{0,2}(\nabla a_3) \diamond \om_2
    -
    2(\lm_2-\mu) S^2_{0,3}(\nabla a_2) \diamond \om_3
    \big)\nonumber\\
    &+
    P_2\big(
    (2-\lm_2-\mu) \pi^2_-(da_2)+
    2(\lm_1-\mu) S^2_{0,3}(\nabla a_1) \diamond \om_3
    -
    2(\lm_3-\mu) S^2_{0,1}(\nabla a_3) \diamond \om_1
    \big)\nonumber\\
    &+
    P_3\big(
    (2-\lm_3-\mu) \pi^2_-(da_3)+
    2(\lm_2-\mu) S^2_{0,1}(\nabla a_2) \diamond \om_1
    -
    2(\lm_1-\mu) S^2_{0,2}(\nabla a_1) \diamond \om_2
    \big)\nonumber\\
    &+ \sum_{i=1}^3 (2\lm_i-\lm)\divergence(a_i)\om_i \Big) \diamond \boldsymbol{\om}=:\mathcal{P}(\boldsymbol{\om}). 
\label{eq: weighted flow}
\end{align}
Observe that, when $\mu = \lm_1=\lm_2=\lm_3=1$ and $\lm=2$, we indeed recover the negative gradient flow of (\ref{def:energyfunctionalfulltorsion}), as discussed in \S\ref{section: ricci harmonic like flow}. Our goal here is to identify a range of values for the constants $\lm_i$ such that this flow remains parabolic. To this end we set 
$$L_1:=\sup_i\{|2-\lm_i-\mu|\},\quad 
  L_2:=\sup_i\{|\lm_i-\mu|\}, \qandq
  L_3:=\sup_i\{|2\lm_i-\lm|\}.$$
Using the principal symbols given in Appendix \ref{appendix: principal symbol}, alongside the Cauchy-Schwarz and Young inequalities, one  easily obtains:
\begin{lemma}
    The following bounds hold
\begin{align*}
    |\langle \sigma(L_{g(da_1,\om_1)g})(x,\xi)(A), A \rangle| 
        &\leq  |\xi|^2(10 f_0^2 + 2\sum_{i=1}^3|\tilde{B}_{0,i}|^2), \\
    |\langle \sigma(L_{P_1(\pi^2_-(da_1))})(x,\xi)(A), A \rangle| 
        &\leq |\xi|^2 (f_0^2 + 4|\tilde{B}_{0,1}|^2 + |\tilde{B}_{0,2}|^2 + |\tilde{B}_{0,3}|^2),\\
    |\langle \sigma(L_{P_1(S^2_{0,2}(\nabla a_3)\diamond \om_2)})(x,\xi)(A), A \rangle| 
        &\leq  |\xi|^2 \left(\frac{1}{2}f_0^2 + f_3^2 + 3|\tilde{B}_{0,1}|^2 + \frac{1}{2}(|\tilde{B}_{0,2}|^2 + |\tilde{B}_{0,3}|^2)\right),\\
    |\langle \sigma(L_{P_1(S^2_{0,3}(\nabla a_2)\diamond \om_3)})(x,\xi)(A), A \rangle| 
        &\leq  |\xi|^2 \left(\frac{1}{2}f_0^2 + f_2^2 + 3|\tilde{B}_{0,1}|^2 + \frac{1}{2}(|\tilde{B}_{0,2}|^2 + |\tilde{B}_{0,3}|^2)\right),\\
    |\langle \sigma(L_{\divergence(a_1)\om_1})(x,\xi)(A), A \rangle| 
        &\leq  |\xi|^2(6 f_1^2 + |\tilde{B}_{0,2}|^2 + |\tilde{B}_{0,3}|^2), 
\end{align*}
and similar expressions hold for cyclic permutation of the indices. 
\end{lemma}
\begin{proof}
    We prove the first one, to establish the procedure; the rest follow by analogous computations.
    \begin{align*}
    |\langle \sigma(L_{g(da_1,\om_1)g})(x,\xi)(A), A \rangle| 
    &= 4 | f_0^2 |\xi|^2 - f_0 g(\xi \wedge J_1 \xi, \tilde{B}_{0,1}) + f_0 g(\xi \wedge J_2 \xi, \tilde{B}_{0,2}) + f_0 g(\xi \wedge J_3 \xi, \tilde{B}_{0,3}) |\\
    &\leq 4 f_0^2 |\xi|^2 + 4 |\xi|^2 (|f_0| |\tilde{B}_{0,1}| + |f_0| |\tilde{B}_{0,2}| + |f_0| |\tilde{B}_{0,3}|)\\
    &\leq 4 f_0^2 |\xi|^2 + 2 |\xi|^2 (3 f_0^2 + |\tilde{B}_{0,1}|^2 + |\tilde{B}_{0,2}|^2 + |\tilde{B}_{0,3}|^2)\\
    &= |\xi|^2 (10 f_0^2 + 2 \sum_{i=1}^3 |\tilde{B}_{0,i}|^2).
    \qedhere
\end{align*}
\end{proof}
As in the proof of Theorem \ref{thm: short-time existence}, we can apply the modified DeTurck trick to (\ref{eq: weighted flow}), using the same vector field. Computing as in \S\ref{section: ricci harmonic like flow}, using the above estimates and (\ref{eq: estimating principal sym}), we find 
\begin{align*}
\langle 
\sigma(L_{\mathcal{P}})(x,\xi)(A)
, A \rangle 
\geq  \  &|\xi|^2f_0^2(8- 18 L_1 - 6L_2 )\ +\\ 
&|\xi|^2 (4-\frac{1}{2\epsilon}|\lm+\frac{\mu}{2}-1|-9L_1-16L_2-2L_3)\sum_{i=1}^3 |\tilde{B}_{0,i}|^2\ +\\
&|\xi|^2 
(2(2\lm+\mu-\epsilon|\lm+\frac{\mu}{2}-1|)-4 L_2 -6 L_3) \sum_{i=1}^3 f_i^2.
\end{align*}
In particular, setting $\mu=1$ and $\lm=1/2$, we see that
if $\lm_i$ are such $9L_1+3L_2<4$,  $9L_1+16L_2+2L_3<4$ and $2L_2+3L_3<2$, then the flow (\ref{eq: weighted flow}) exists for short time. Note that the range of values here are not necessarily sharp. \\

\paragraph{Weyl curvature terms decrease parabolicity.} 
Consider the flow given by
\begin{equation}
    \partial_t(\boldsymbol{\om}) 
    = \big(\nu g(da_1,\om_2)g-2 \Ric(g) - \mu \sum_{i=1}^3 \mathcal{L}_{(J_i a_i)^\sharp}g + \lm \sum_{i=1}^3 \divergence(a_i) \om_i \big) \diamond \boldsymbol{\om},
    \label{eq: weyl decrease parabolicity}
\end{equation}
with $\nu,\mu,\lm \in \R$. This is clearly a modified version of (\ref{eq: 2 parameter family of flows}) by a `self-dual Weyl-type' term $g(da_1,\om_2)$, in the terminology of Proposition \ref{proposition sd weyl curvature}. 

From Appendix \ref{appendix: principal symbol}, we have
\[
\langle \sigma(L_{g(d a_1,\om_2)})(x,\xi)(A) g, A \rangle 
= - 4f_0 \big(g(\xi \w J_2 \xi, \tilde{B}_{0,1})+g(\xi \w J_1 \xi, \tilde{B}_{0,2})\big).
\]
Observe that this term does not come with a `sign', ie. the flow (\ref{eq: weyl decrease parabolicity}) will exist for a smaller range of values than that in Theorem \ref{thm: short-time existence}. To illustrate this more explicitly, we again start by using
Cauchy-Schwarz and Young's inequalities:
\begin{align*}
|\langle \sigma(L_{g(d a_1,\om_2)})(x,\xi)(A) g, A \rangle | 
&= 4|f_0||g(\xi \w J_2 \xi, \tilde{B}_{0,1})+g(\xi \w J_1 \xi, \tilde{B}_{0,2})| \\
&\leq 4 |\xi|^2 |f_0|  (|\tilde{B}_{0,1}|+|\tilde{B}_{0,2}|) \\
&\leq 2 |\xi|^2 \big({2 f_0^2} + (|\tilde{B}_{0,1}|^2+|\tilde{B}_{0,2}|^2)\big).
\end{align*}
Repeating the calculations in the proof of Theorem \ref{thm: short-time existence} for the flow (\ref{eq: weyl decrease parabolicity}),
we obtain, instead of (\ref{eq: estimating principal sym}), 
\begin{align*}
\langle 
\sigma(L_{\mathcal{P}})(x,\xi)(A)
, A \rangle 
\geq \ &2 |\xi|^2 |B|^2 + 2|\xi|^2 (2 \lm + \mu) (f_1^2+f_2^2+f_3^2)  \\
&-|\xi|^2  |\lambda +\frac{\mu}{2}-1|\Big(
\frac{|\tilde{B}_{0,1}|^2+|\tilde{B}_{0,2}|^2+|\tilde{B}_{0,3}|^2}{2\epsilon}+ 2\epsilon (f_1^2+f_2^2+f_3^2)
\Big)\\
&- 2|\nu| |\xi|^2 ({2 f_0^2} + (|\tilde{B}_{0,1}|^2+|\tilde{B}_{0,2}|^2))
\\
= \  &(8 - 4 |\nu|) |\xi|^2f_0^2 +
\frac{1}{2}|\xi|^2(8-\frac{1}{\epsilon}|\lm+\frac{\mu}{2}-1| -4 |\nu|)\sum_{i=1}^2|\tilde{B}_{0,i}|^2 \\ 
&+ 
\frac{1}{2}|\xi|^2(8-\frac{1}{\epsilon}|\lm+\frac{\mu}{2}-1|)|\tilde{B}_{0,3}|^2
+ 2|\xi|^2(2\lm+\mu-\epsilon|\lm+\frac{\mu}{2}-1|) \sum_{i=1}^3 f_i^2.
\end{align*}
In order to get parabolicity, we now need $|\nu|< 2$,
 $|x-1|<(8-4|\nu|)\epsilon$ and $|x-1|<2x/\epsilon$, where $x:=\lm+\frac{\mu}{2}$. Comparing with the proof of Theorem \ref{thm: short-time existence}, we see that the range of values for $\lm+\frac{\mu}{2}$ will be smaller than that of Theorem \ref{thm: short-time existence}. This is in fact a general feature: the Weyl curvature terms reduce parabolicity, since unlike the Ricci tensor they lack a Laplacian term, cf. their principal symbols given in Appendix \ref{appendix: principal symbol}. Thus, addition of the $\SU(2)$-invariants of the form $g(da_i,\om_j)$, for $i<j$, always makes the family of weighted flows `less parabolic', in this precise sense. The crucial conclusion is that we need the Ricci curvature term to dominate the Weyl terms, in order to ensure short-time existence.

By comparison to the $\mathrm{G}_2$ case, the reader can observe in  \cite{Dwivedi2023}*{Theorem 6.76} that the Weyl curvature term, denoted by $F$ therein, reduces the parabolicity range; this is precisely the feature we described above. Unlike the $\mathrm{G}_2$ case, however, there is here also an issue of multiplicity. To explain this, consider a flow given by
\begin{equation}
        \partial_t(\boldsymbol{\om}) = \big(\nu g(da_1,\om_2)\om_1-2 \Ric(g) - \mu \sum_{i=1}^3\mathcal{L}_{(J_i a_i)^\sharp}g + \lm \sum_{i=1}^3\divergence(a_i) \om_i \big) \diamond \boldsymbol{\om}\label{eq: weyl decrease parabolicity 2}.
\end{equation}
Both in (\ref{eq: weyl decrease parabolicity}) and (\ref{eq: weyl decrease parabolicity 2}), we are modifying the flow (\ref{eq: 2 parameter family of flows}) by the same term $g(da_1,\om_2)$, but the resulting flows are very different. Similarly, adding either $S^2_{0,1}(\nabla a_1)$ or $P_2(S^2_{0,1}(\nabla a_1)\diamond \om_1)$ results in different flows, however much these terms are equivalent as $\SU(2)$-invariants. This type of phenomenon occurs because $S^2(M)\oplus \mathfrak{su}(2)^\perp \cong 4\R \oplus 3 \R^3$ contains multiple copies of some $\SU(2)$-modules; this also occurs in the $\SU(n)$ and $\Sp(n)$ cases, for $n>1$, but not in the $\mathrm{G}_2$ and $\Spin(7)$ cases, see \S\ref{sec: comparison with g2 and spin7} below. The point here is that flow families of the forms (\ref{eq: weyl decrease parabolicity}) and (\ref{eq: weyl decrease parabolicity 2}), although built from the same second-order invariants, will have different parabolicity ranges. 

\subsection{Solitons}
In this section  we define the notion of solitons for the $\SU(2)$-we considered in \S\ref{section: ricci harmonic like flow}, while leaving a more detailed study of those solutions as a suggestion for future investigations.
We begin by considering a general $\SU(2)$-flow  of the form 
\begin{equation}
    \partial_t (\boldsymbol{\om}) = \mathcal{P}(\boldsymbol{\om}),\label{eq: general flow for soliton}
\end{equation}
where $\mathcal{P}$ denotes a \emph{diffeomorphism-invariant}  partial differential operator, ie. $$\varphi^*\mathcal{P}(\boldsymbol{\om})= \mathcal{P}(\varphi^*\boldsymbol{\om}),
\quad\forall \varphi\in\Diff(M),
$$ 
and which is also \emph{homothetic scaling-invariant}, ie. $\mathcal{P}(c^2 \boldsymbol{\om})=\mathcal{P}( \boldsymbol{\om})$, for $c\in \R$. Note that in particular the Ricci map $g \mapsto \Ric(g)$ satisfies these two properties. Suppose now that $\boldsymbol{\om}_0$ is a solution to 
\begin{equation}
\mathcal{P}(\boldsymbol{\om}_0) = \mathcal{L}_{Y}\boldsymbol{\om}_0 - \nu \boldsymbol{\om}_0,\label{eq: general soliton}
\end{equation}
for some vector field $Y$ and $\nu \in \R$. Let $\boldsymbol{\om}(t):= (1- \nu t ) \varphi_t^*\boldsymbol{\om}_0$,  $X(t):=(1-\nu t)^{-1}Y$ and denote by $\varphi_t$ the $1$-parameter family of diffeomorphisms generated by $X(t)$ as in (\ref{eq: flow generated by vector field}), then we have
\begin{align*}
    \partial_t \boldsymbol{\om}(t) &= -\nu \varphi_t^*\boldsymbol{\om}_0 + (1-\nu t ) \varphi^*_t(\mathcal{L}_{X(t)}\boldsymbol{\om}_0)\\
    &= \varphi^*(\mathcal{P}(\boldsymbol{\om}_0))\\
    &= \mathcal{P}(\boldsymbol{\om}(t)).
\end{align*}
Thus, $\boldsymbol{\om}(t)$ is a solution to (\ref{eq: general flow for soliton}) which evolves only by the symmetries of the flow. Mimicking the definition of Ricci soliton cf. \cite{Topping2006}*{Definition 1.2.2}, we define:
\begin{definition}
    The data $(\boldsymbol{\om}_0,Y, \nu)$ satisfying (\ref{eq: general soliton}) is called a \textit{shrinking}, \textit{steady} or \textit{expanding} soliton for the flow (\ref{eq: general flow for soliton}) if $\nu<0$, $\nu=0$ or $\nu>0$ respectively. 
\end{definition}
From Proposition \ref{torsionproposition}, it is not hard to see that the torsion forms remain unchanged under the homothetic scaling $\tilde{\boldsymbol{\om}}:=c^2 \boldsymbol{\om}$,  i.e. $\tilde{a}_i=a_i$. Furthermore, the $a_i=a_i(\boldsymbol{\om})$ are diffeomorphism-invariant; indeed the `Bianchi-type identity' in Corollary \ref{cor:bianchitypeidentity} is a consequence of this invariance, cf. \cite{Karigiannis2009}*{\S 4.1}. Since $\tilde{J} \tilde{a}_i= {J} {a}_i$, we also have  $(\tilde{J} \tilde{a}_i)^{\tilde{\sharp}}= c^{-2}({J} {a}_i)^\sharp$, and similarly one can check that $\tilde{\delta} \tilde{a}_i= c^{-2} \delta a_i$. We deduce that the flow given by
\begin{gather}
    \partial_t(\boldsymbol{\om}) 
    = \left(-2 \Ric(g) - \mu \sum_{i=1}^3\mathcal{L}_{(J_i a_i)^\sharp}g + \lm \sum_{i=1}^3\divergence(a_i) \om_i \right) \diamond \boldsymbol{\om} + {Q}(\boldsymbol{a}),
    \label{eq: general flow soliton}
\end{gather}
where ${Q}(\boldsymbol{a})$ denotes a function in $a_i$, satisfies the properties of (\ref{eq: general flow for soliton}). This also applies to other flows, such as (\ref{eq: more general flow}). Note that in practice we will only require $Q(\boldsymbol{a})$ to be a quadratic polynomial in the torsion. Thus, a soliton of (\ref{eq: general flow soliton}) is a solution to (\ref{eq: general soliton}), where $\mathcal{P}(\boldsymbol{\om})$ is the RHS of (\ref{eq: general flow soliton}). In particular, it would interesting to find solitons when
\begin{align*}
    \mathcal{P}(\boldsymbol{\om}) 
    &= 
    \left(
    \mathrm{sym}(\mathrm{div}(T^t)) + T*T - \frac{1}{2}|T|^2g+\mathrm{div}(T)\right)
    \diamond \boldsymbol{\om}\\
    &= \sum_i \left( -\mathrm{sym}(\nabla_{J_i}a_i) 
    - \sum_{j,k}\epsilon_{ijk} a_i \odot J_ja_k
    + 2 a_i \otimes a_i - |a_i|^2 g + \mathrm{div}(a_i)\om_i
    \right) \diamond \boldsymbol{\om},
\end{align*}
corresponding to the negative gradient flow of (\ref{def:energyfunctionalfulltorsion}) 
and where the second equality follows from Proposition \ref{prop:divTintermsofa}.

\newpage
\section{Discussion of results}

\subsection{\texorpdfstring{$\mathrm{SU}(3)$, $\mathrm{G}_2$ and $\mathrm{Spin}(7)$}{} second-order invariants revisited}
\label{sec: comparison with g2 and spin7}

We now provide several illustrations for our reading of Bryant's remarks, and we explain how this approach applies to general $H$-structures. We give an alternative proof of  \cite{Dwivedi2023}*{Theorem 6.2}, which states that there are exactly 6 second order $\mathrm{G}_2$ invariants that lead to geometric flows of the underlying $\mathrm{G}_2$-structure  $3$-form. Our approach differs from the one given in \cite{Dwivedi2023}*{\S 5}, in that it uses the representation-theoretic method of \S\ref{subsect: rep theory}, following \cite{Bryant06someremarks}.
In the $\Spin(7)$ case, we show that there are exactly 4 such terms; this was also recently determined in \cite{Dwivedi2024}.
In the case of $\SU(3)$-structures, we use results in \cite{Bedulli2007} to show that  there are exactly 16 such terms, leading to a 19-parameter family of possible $\SU(3)$-flows, to second order; this doesn't seem to have been mentioned heretofore in the literature. In particular, we shall see that both the $\mathrm{G}_2$ and $\Spin(7)$ cases admit fewer  geometric flows, compared to the $\SU(2)$ case; this explains why a full classification of parabolic flows of $\SU(2)$-structures is so much more involved. 

We first recall from \cite{Fadel2022} that, in order to define a geometric flow of an $H$-structure, with $H\subset \mathrm{SO}(n)$, say determined by a (multi-)tensor $\xi$, we need to choose an endomorphism $A=A(\xi)$ and evolve the underlying tensor by 
\begin{equation}
    \partial_t \xi = A \diamond \xi.\label{equ: general flow of h structures}
\end{equation} 
Denoting the space of second-order invariants by $V_2(\mathfrak{h})$, as in \S\ref{subsect: rep theory}, let us now restrict to cases in which $A$ is at most second order in $\xi$, ie. $A \in V_2(\mathfrak{h})$, to highest order. 
Note that we also need the endomorphism $A$ to lie in the complement of   $\ker(\diamond \xi)$, 
since $H$ is the stabiliser of $\xi$ by definition;
we denote this space by $$\mathrm{End}(\R^n)_{\mathfrak{h}} := \mathrm{End}(\R^n)/{\mathfrak{h}}\cong \Sigma^2(\R^n)\oplus \mathfrak{h}^\perp,
$$
so that the assumption is $A\in 
\mathrm{End}(\R^n)_{\mathfrak{h}}  \cap V_2(\mathfrak{h})$.
To be precise,  $V_2(\mathfrak{h}) \cap \mathrm{End}(\R^n)$ refers to the subspace of $V_2(\mathfrak{h})$ consisting of those $H$-modules which also occur in $\mathrm{End}(\R^n)_{\mathfrak{h}}$, \textit{counted with their multiplicities}. This is precisely the approach we considered in this article, and we shall next apply it  to the cases $H=\mathrm{G}_2$ and $\Spin(7)$. 

\paragraph{The $\mathrm{G}_2$ case.} 
The space $V_2(\mathfrak{g}_2)$ was computed in \cite{Bryant06someremarks}*{(4.7)}. From this we easily see that
\begin{equation}
    V_2(\mathfrak{g}_2) \cap \mathrm{End(\R^7)}_{\mathfrak{g}_2} \cong \R \oplus 2 \Lm^1(\R^7) \oplus 3 \Sigma^2_0(\R^7),
\end{equation}
where, in the notation of \cite{Bryant06someremarks}, $V_{0,0}\cong\R$, $V_{1,0} \cong\Lm^1(\R^7)$ and $V_{2,0} \cong \Sigma^2_0(\R^7)$. 
Observe that there are exactly 6 irreducible $\mathrm{G}_2$-modules as expected from \cite{Dwivedi2023}*{Theorem 6.2}. Moreover, as $\mathrm{G}_2$-modules, we also know from \cite{Bryant06someremarks}*{(4.8)}  that
\begin{equation}
    V_2(\mathfrak{so}(7)) \cap \mathrm{End(\R^7)}_{\mathfrak{g}_2} \cong \R \oplus 2 \Sigma^2_0(\R^7).
\end{equation}
Hence the $\R$ component in $V_2(\mathfrak{g}_2) \cap (\mathrm{End(\R^7)})_{\mathfrak{g}_2}$ is given by the scalar curvature, and two of the three components of $\Sigma^2_0(\R^7)$ by the traceless Ricci tensor and a Weyl curvature term $W_{27}$. Clearly the third term in $\Sigma^2_0(\R^7)$ has to be given by $\mathcal{L}_{\tau_1^\sharp}g$, since it does not lie in $V_2(\mathfrak{so}(7))$; here $\tau_1$ denotes the $\mathrm{G}_2$ torsion $1$-form, as defined in \cite{Bryant06someremarks}*{(3.7)}. Among the two terms spanning $\Lm^1(\R^7)$, one has to be given by the harmonic part $\mathrm{div}(T)$; compare with Proposition \ref{prop: harmonic condition}, see also \cite{LoubeauSaEarp2019}. 
We should point out that, strictly speaking, the harmonic part lies in $\mathfrak{h}^\perp \subset \Lm^2(\R^n)$, and in the $\mathrm{G}_2$ case it so happens that $$\mathfrak{g}_2^\perp= \Lm^2_{7}(\R^7) \cong \Lm^1(\R^7),$$
see eg. \cite{Bryant06someremarks}*{(2.14)}.
The second term in $\Lm^1(\R^7)$, denoted by $\mathrm{div}(T^t)$ in \cite{Dwivedi2023}, corresponds to other $\Lm^1(\R^7)$-component of $\nabla T$, where $T$ denotes the full intrinsic torsion. 
We now explain the appearance of this term. 
As described below for general $H$, the term $\mathrm{div}(T^t)$ in fact belongs to space $$\Sigma^2(\R^7)\otimes \mathfrak{g}_2^\perp  \cong 2\Lm^1(\R^7) \oplus \Sigma^2_0(\R^7)\oplus V_{1,1}\oplus V_{0,1}\oplus  V_{3,0}.
$$
It follows that $\divergence(T^t)$ is obtained by suitably tracing $\nabla T$: whereas the trace on the first two indices yields the usual $\mathrm{div}(T)$, the trace of the first and third indices gives precisely $\mathrm{div}(T^t)$. Note that the occurrence of $\Sigma^2_0(\R^7)$ in the  decomposition corresponds to the aforementioned $\mathcal{L}_{{\tau_1}^\sharp}g$ term.
Since $\mathrm{End(\R^7)}_{\mathfrak{g}_2} \cong \R \oplus \Lm^1(\R^7) \oplus \Sigma^2_0(\R^7)$, it follows that there is a $6$-parameter family of $\mathrm{G}_2$-flows, to second order.
Thus, specialising (\ref{equ: general flow of h structures}) to flows of the $3$-form $\xi=\vp$, we get
\begin{equation}
\partial_t(\vp) = \Big(\lm_1 \mathrm{Scal}(g)g+\lm_2\Ric(g)+
\lm_3 W_{27}+ \lm_4\mathcal{L}_{\tau_1^\sharp}g + \lm_5 \divergence(T) +\lm_6 \divergence(T^t) \Big)\diamond \vp.
\end{equation}
as the general $\mathrm{G}_2$-flow, to second order, up to lower order terms; this was indeed shown in \cite{Dwivedi2023}. 

Owing to the fact that Bryant had already computed the spaces $V_1(\mathfrak{g}_2)$ and $V_2(\mathfrak{g}_2)$, and introduced the Laplacian flow in \cite{Bryant06someremarks}, we believe that he was already aware of the existence of a $6$-parameter family of $\mathrm{G}_2$-flows, albeit not explicitly mentioning it therein. The reason for singling out the Laplacian flow of closed $\mathrm{G}_2$-structures in this family is the fact that the underlying metric evolves by the Ricci flow (to highest order), see \cite{Bryant06someremarks}*{(6.15)}, and the stationary points are precisely the torsion-free $\mathrm{G}_2$-structures. More generally, Bryant's results show that, for any flow of closed $\mathrm{G}_2$-structures, the metric can only evolve by the Ricci curvature.

\paragraph{The $\mathrm{Spin}(7)$ case.}
The space $V_2(\mathfrak{spin}(7))$ was computed in \cite{UdhavFowdar}*{\S 4}. From this we easily see that
\begin{equation}
    V_2(\mathfrak{spin}(7)) \cap \mathrm{End(\R^8)}_{\mathfrak{spin}(7)} \cong \R \oplus \mathfrak{spin}(7)^\perp \oplus 2 \Sigma^2_0(\R^8),\label{equ: the spin7 v2}
\end{equation}
where, in their notation, $V_{0,0,0} \cong \R$, $V_{1,0,0}\cong \mathfrak{spin}(7)^\perp\cong \R^7$ and $V_{0,0,2} \cong \Sigma^2_0(\R^8)$. There are exactly 4 irreducible $\mathrm{Spin}(7)$-modules, which is consistent with \cite{Dwivedi2024}.
On the other hand, it is also easy to show that, as $\mathrm{Spin}(7)$-modules,
\begin{equation}
    V_2(\mathfrak{so}(8)) \cap \mathrm{End(\R^8)}_{\mathfrak{spin}(7)} \cong \R \oplus  \Sigma^2_0(\R^8).
\end{equation}
This implies that the $\R$ component in $V_2(\mathfrak{spin}(7)) \cap (\mathrm{End(\R^8)})_{\mathfrak{spin}(7)}$ is again given by the scalar curvature, and one of the two $\Sigma^2_0(\R^7)$ components is given by the traceless Ricci tensor. Clearly the second term in $\Sigma^2_0(\R^8)$ has to be given by $\mathcal{L}_{{T^1_8}^\sharp}g$, since it does not lie in $V_2(\mathfrak{so}(8))$; here $T^1_8$ denotes the $\Spin(7)$ torsion $1$-form, as defined in \cite{UdhavFowdar}*{(2.5)}. The remaining term in $\mathfrak{spin}(7)^\perp$ is again just the harmonic term $\mathrm{div}(T)$, where $T$ denotes the $\Spin(7)$ intrinsic torsion. This exhausts all the second-order invariants which can be used to evolve the $\Spin(7)$-structure $4$-form. Since $\mathrm{End(\R^8)}_{\mathfrak{spin}(7)} \cong \R \oplus \mathfrak{spin}(7)^\perp \oplus \Sigma^2_0(\R^7)$, we conclude that there is a $4$-parameter family of $\Spin(7)$ flows, to second order. Explicitly, specialising (\ref{equ: general flow of h structures}) to flows of the $4$-form $\xi=\Phi$, we get
\begin{equation}
    \partial_t(\Phi) 
    = (\lm_1 \mathrm{Scal}(g)g+\lm_2\Ric(g)+ \lm_3\mathcal{L}_{{T^1_8}^\sharp}g + \lm_4 \divergence(T) )\diamond \Phi.
\end{equation}
We observe that, while there is exactly one Weyl curvature term that can occur in a $\mathrm{G}_2$-flow, none at all can occur in the $\Spin(7)$ case. By further contrast, we saw in the $\SU(2)$ case that all of the $5$ components of $W^+$ can occur.

\paragraph{The general $H$ case.} Given $H\subset \mathrm{SO}(n)$, as in Berger's list \cite{Berger}, we always have
\begin{equation}
     \R \oplus \mathfrak{h}^\perp \oplus 2 \Sigma^2_0(\R^n) \subset V_2(\mathfrak{h}) \cap \mathrm{End(\R^n)}_{\mathfrak{h}}, 
     \label{minimal case}
\end{equation}
where, as before, the $\R$ component corresponds to the scalar curvature, and one of the two $\Sigma^2_0(\R^n)$ components is given by the traceless Ricci tensor. The remaining term in $\Sigma^2_0(\R^n)$ is of the form $\mathcal{L}_{\tau^\sharp}g$, where $\tau$ is a $1$-form component in the space of intrinsic torsion $V_1(\mathfrak{h})$, sometimes called the \emph{Lee form} in the literature, eg. $\tau=\sum_{i=1}^3 J_i a_i$, for $H=\SU(2)$.
The $\mathfrak{h}^\perp $ component corresponds as before to the harmonic term, cf. \cite{LoubeauSaEarp2019}. Note that we are not assuming that these are irreducible as $H$-modules, which will in general depend on the group $H$. In particular, we see that the $\Spin(7)$ setup actually realises the minimal instance of (\ref{minimal case}). Moreover, observe that the second-order invariants that arise from the curvature tensor lie in the sub-module
\[
V_{2}(\mathfrak{so}(n))\cong \R \oplus \Sigma^2_0(\R^n)\oplus W \subset V_{2}(\mathfrak{h}), 
\]
where $W$ denotes the space of Weyl curvature terms, and furthermore, from the proof of (\ref{eq: space of second order invariants}), we know that the remaining ones lie in 
\[
\Sigma^2(\R^n) \otimes \mathfrak{h}^\perp\cong \Sigma^2_0(\R^n) \otimes \mathfrak{h}^\perp \oplus \mathfrak{h}^\perp,
\]
where of course the harmonic term $\mathrm{div}(T)$ lies in the last summand. This provides an efficient method to extract the second-order invariants, seeing as only these arise independently from the curvature and $\mathrm{sym}(\nabla T)\diamond \xi$. The number of possible second-order flow to highest order can then be deduced from the multiplicity of each irreducible $H$-module in $V_2(\mathfrak{h})$ which also appears in $\mathrm{End}(\R^n)_{\mathfrak{h}}$.

In the cases $H=\SU(n)$ and $\Sp(n)$, there are many other terms aside from those in (\ref{minimal case}), as already illustrated by our results for $H=\SU(2)=\Sp(1)$.
Another distinctive feature of the $\SU(2)$ case, relative to $\mathrm{G}_2$ or $\Spin(7)$, is the fact that there occur $H$-modules with repeated multiplicities in $\End(\R^n)_{\mathfrak{h}}$, which is why we had to carefully distinguish between the spaces $\Sigma^2_{0,1}, \Sigma^2_{0,2}, \Sigma^2_{0,3}$ and $\Lm^2_-$, and likewise for $g,\om_1,\om_2$ and $\om_3$. This indeed explains why there are so many possible flows in this context, which also occurs in the $\SU(n)$ and $\Sp(n)$ cases. For instance, we illustrate the $\SU(3)$ setup next as a novel application of our approach. 

\paragraph{The $\SU(3)$ case.}
The space $V_{2}(\mathfrak{su}(3))$ was computed in \cite{Bedulli2007}*{\S 2.6}. From this we easily see that
\begin{equation}
    V_{2}(\mathfrak{su}(3)) \cap \End(\R^6)_{\mathfrak{su}(3)} \cong 3 \R \oplus 4 \Lm^1(\R^6) \oplus 5 S^2_+(\R^6) \oplus 4 S^2_-(\R^6),
\end{equation}
where $ \Sigma^2_0(\R^6) = S^2_+(\R^6) \oplus S^2_-(\R^6)$, as irreducible $\SU(3)$-modules of dimensions $8$ and $12$, respectively  \cite{Bedulli2007}*{\S 2.3}.
In the above decomposition we used the identifications   
$$\mathfrak{su}(3)^\perp \cong \R \oplus \Lm^1(\R^6)
\qandq
\mathfrak{su}(3)\cong S^2_+(\R^6).
$$ 
Thus, there are a total  16 possible second-order $\SU(3)$-invariants which can be used to evolve an $\SU(3)$-structure. On the other hand, as $\SO(6)$-modules,  it is not hard to compute that 
\begin{equation}
    V_{2}(\mathfrak{so}(6)) \cong  \R \oplus \Sigma^2_0(\R^6) \oplus V_{0,2,2},
\end{equation}
where $V_{0,2,2}$ is the irreducible $\SO(6)$-module of dimension $84$, corresponding to the Weyl curvature. One still has to check whether any components of the Weyl  term occur in $V_{2}(\mathfrak{su}(3)) \cap \End(\R^6)$, and indeed a calculation shows that, as $\SU(3)$-modules,
\[
V_{0,2,2} \cap \End(\R^6)_{\mathfrak{su}(3)} \cong \R \oplus \Lm^1(\R^6) \oplus S^2_+(\R^6) \oplus S^2_-(\R^6).
\]
Thus, among the 16 second-order $\SU(3)$-invariants, precisely $7$ can be computed from the Riemann curvature tensor. 
Since $$\End(\R^6)_{\mathfrak{su}(3)}\cong 2 \R \oplus \Lm^1(\R^6) \oplus S^2_+(\R^6) \oplus S^2_-(\R^6),$$
note that there is a $(2\times 3 + 4 + 5 + 4) = 19$-parameter family of $\SU(3)$-flows, to second order. 
Geometric $\SU(3)$-flows  arising from this framework will be investigated in a future work.

\subsection{Quaternionic bundle}
\label{sec: quaternionic bundle}

In \cite{Grigorian2017}, Grigorian derived the notion of 
harmonic $\mathrm{G}_2$-structures independently of \cites{Dwivedi2021,LoubeauSaEarp2019}, by interpreting the intrinsic torsion as a connection on a rank $8$ octonionic vector bundle $\mathbb{O}M$, whose natural `Coulomb gauge' condition is equivalent to geometric harmonincity. In our $\SU(2)$ case, we can define by analogy a \emph{quaternionic bundle} over $M^4$ as follows.

We consider the rank $4$ vector bundle $\mathbb{H}M:=\Lm^0 \oplus \Lm^2_+$, and define a quaternionic product on sections by
\begin{equation}
    (a,\alpha) \cdot (b, \beta) := (ab-\frac{1}{2}g(\alpha,\beta), a\beta+b\alpha+\alpha \diamond \beta),
\end{equation}
corresponding to the usual multiplication on $\mathbb{H}$, whereby $\Im(\mathbb{H})\simeq\Lm^2_+(M)$, see Lemma \ref{lemma: action of diamond on om_i}. We define an inner product on $\mathbb{H}M$ by
\begin{equation}
    \langle (a,\alpha) , (b, \beta)\rangle := ab +\frac{1}{2}g(\alpha, \beta),
\end{equation}
where the half factor stems from our convention $g(\om_i, \om_j)=2\delta_{ij}$. Since the operator $\diamond$ only depends on the metric, it is covariantly constant and hence:
\begin{proposition}
    The induced covariant derivative on $\mathbb{H}M$ is compatible with the quaternionic product, ie.
\begin{equation}
    \nabla_X\big( (a,\alpha) \cdot (b, \beta)\big)
    =\nabla_X\big( (a,\alpha)\big) \cdot (b, \beta)+(a,\alpha) \cdot \nabla_X\big((b, \beta)\big),
\end{equation}
where $\nabla_X( (c,\gamma)):=(\nabla_X c,\nabla_X\gamma)$.
\end{proposition}

Comparing with \cite{Grigorian2017}*{Proposition 6.1}, we see that in the $\mathrm{G}_2$ case the above compatibility does not hold for the octonionic bundle $\mathbb{O}M$, and in fact a term involving the intrinsic torsion arises. This seems to be due to the fact that,  exceptionally in dimension $4$, the Levi-Civita connection preserves the decomposition $\Lm^2_+\oplus \Lm^2_-$ without any holonomy constraint. 
In \cite{Grigorian2017}, the octonionic bundle $\mathbb{O}M$ is defined as $\Lm^0\oplus \Lm^1_7$, and the appearance of $\Lm^1_7$ stems from the 
$\mathrm{G}_2$-equivariant isomorphism identifying it with $\Lm^2_7\cong \mathfrak{g}_2^\perp$. Such an identification does not hold in general (certainly not in the $\SU(2)$ case), so  for general $H$-structures one should really consider $\mathfrak{h}^\perp$. In particular, it would be interesting to investigate this perspective on harmonic $H$-structures for the quaternionic structures $\Sp(n)$ and $\Sp(n)\Sp(1)$,  $n>1$.

\subsection{Ongoing and future work}

We briefly discuss some natural open questions related to our results, which might be object of investigation by the authors and the wider community in the immediate future.

\begin{enumerate}

\item One important issue which we have overlooked in this paper is the choice of lower-order terms in the flows. While for short-time existence one can restrict attention to the highest-order terms, this says nothing about long-time existence, let alone convergence. Neither did we address what the stationary points of the flows actually are expected to be; one should be able to find suitable expressions involving terms in $V_2(\mathfrak{su}(2))$ and $S^2(V_1(\mathfrak{su}(2)))$ such that the stationary points have special torsion -- most optimistically no torsion at all.

    \item 
    Our results suggest that the most interesting flows to consider are the harmonic Ricci flow and the negative gradient flow the Dirichlet energy functional \eqref{def:energyfunctionalfulltorsion}.
    In both cases, it is still unclear whether there are more critical points, aside from the torsion-free ones. Moreover, one still needs to develop a long-time theory for them, including Shi-type estimates, monotone quantities, etc. In particular, what can one say in terms of  singularity formation, solitons, etc.?

\item We are currently investigating the general flows of $\SU(n)$- and $\Sp(n)$-structures, for $n\geq 2$, together with the other authors of \cite{Fadel2022}. The present paper should be considered as a first step in the broader programme of analytical classification of such flows, according to their special stationarity and parabolicity.

    \item It would be interesting to study the flows introduced here in explicit  settings, such as  homogeneous spaces, cohomogeneity one, etc.

    \item It was recently shown in \cite{Kirill2024}*{Theorem 1.3} that there is an $\SU(2)$ functional, different from \eqref{def:energyfunctionalfulltorsion},  whose critical points are Einstein metrics. Using the $\SU(2)$ functionals  defined here in \S\ref{sec: general quadratic functionals}, can one find other such functionals capable of detecting special metrics?

\end{enumerate}

\appendix

\section{Principal symbols}
\label{appendix: principal symbol}

We summarise the principal symbols of all the second order diffeomorphism-invariants belonging to $3$-dimensional $\SU(2)$-module (which we identify with $\Lm^2_-(M)$):
\begin{align*}
   \sigma(L_{S^2_{0,1}(\nabla a_1)})(x,\xi)(A) \diamond \om_1 = \frac{1}{2}\pi^2_-\Big(
    \xi \w *(J_1\xi \w \tilde{B}_{0,1})- \xi \w *(J_3\xi \w \tilde{B}_{0,3})- \xi \w *(J_2\xi \w \tilde{B}_{0,2}) -2 f_1 \xi \w J_1 \xi \Big)
\end{align*}
\begin{align*}
   \sigma(L_{S^2_{0,2}(\nabla a_2)})(x,\xi)(A) \diamond \om_2 = \frac{1}{2}\pi^2_-\Big(
    \xi \w *(J_2\xi \w \tilde{B}_{0,2})- \xi \w *(J_1\xi \w \tilde{B}_{0,1})- \xi \w *(J_3\xi \w \tilde{B}_{0,3}) -2 f_2 \xi \w J_2 \xi \Big)
\end{align*}
\begin{align*}
   \sigma(L_{S^2_{0,3}(\nabla a_3)})(x,\xi)(A) \diamond \om_3 = \frac{1}{2}\pi^2_-\Big(
    \xi \w *(J_3\xi \w \tilde{B}_{0,3})- \xi \w *(J_2\xi \w \tilde{B}_{0,2})- \xi \w *(J_1\xi \w \tilde{B}_{0,1}) -2 f_3 \xi \w J_3 \xi \Big)
\end{align*}
\begin{align*}
   \sigma(L_{S^2_{0,2}(\nabla a_1)})(x,\xi)(A) \diamond \om_2 = \frac{1}{2}\pi^2_-\Big(
    \xi \w *(J_2\xi \w \tilde{B}_{0,1})+ \frac{1}{2}|\xi|^2 \tilde{B}_{0,3} + \xi \w *(J_1 \xi \w \tilde{B}_{0,2}) -2 f_1 \xi \w J_2 \xi 
    +f_0 \xi \w J_3 \xi\Big)
\end{align*}
\begin{align*}
   \sigma(L_{S^2_{0,3}(\nabla a_2)})(x,\xi)(A) \diamond \om_3 = \frac{1}{2}\pi^2_-\Big(
    \xi \w *(J_3\xi \w \tilde{B}_{0,2})+ \frac{1}{2}|\xi|^2 \tilde{B}_{0,1} + \xi \w *(J_2 \xi \w \tilde{B}_{0,3}) -2 f_2 \xi \w J_3 \xi 
    +f_0 \xi \w J_1 \xi\Big)
\end{align*}
\begin{align*}
   \sigma(L_{S^2_{0,1}(\nabla a_3)})(x,\xi)(A) \diamond \om_1 = \frac{1}{2}\pi^2_-\Big(
    \xi \w *(J_1\xi \w \tilde{B}_{0,3})+ \frac{1}{2}|\xi|^2 \tilde{B}_{0,2} + \xi \w *(J_3 \xi \w \tilde{B}_{0,1}) -2 f_3 \xi \w J_1 \xi 
    +f_0 \xi \w J_2 \xi\Big)
\end{align*}
\begin{align*}
   \sigma(L_{S^2_{0,3}(\nabla a_1)})(x,\xi)(A) \diamond \om_3 = \frac{1}{2}\pi^2_-\Big(
    \xi \w *(J_3 \xi \w \tilde{B}_{0,1})+ \xi \w *(J_1\xi \w \tilde{B}_{0,3}) - \frac{1}{2} |\xi|^2 \tilde{B}_{0,2} -2 f_1 \xi \w J_3 \xi 
    -f_0 \xi \w J_2 \xi\Big)
\end{align*}
\begin{align*}
   \sigma(L_{S^2_{0,1}(\nabla a_2)})(x,\xi)(A) \diamond \om_1 = \frac{1}{2}\pi^2_-\Big(
    \xi \w *(J_1 \xi \w \tilde{B}_{0,2})+ \xi \w *(J_2\xi \w \tilde{B}_{0,1}) - \frac{1}{2} |\xi|^2 \tilde{B}_{0,3} -2 f_2 \xi \w J_1 \xi 
    -f_0 \xi \w J_3 \xi\Big)
\end{align*}
\begin{align*}
   \sigma(L_{S^2_{0,2}(\nabla a_3)})(x,\xi)(A) \diamond \om_2 = \frac{1}{2}\pi^2_-\Big(
    \xi \w *(J_2 \xi \w \tilde{B}_{0,3})+ \xi \w *(J_3\xi \w \tilde{B}_{0,2}) - \frac{1}{2} |\xi|^2 \tilde{B}_{0,1} -2 f_3 \xi \w J_2 \xi 
    -f_0 \xi \w J_1 \xi\Big)
\end{align*}
    \[
    \sigma(L_{\pi^2_- (da_1)})(x,\xi)(A) = 
    \pi^2_-\Big(+\frac{1}{2}|\xi|^2 \tilde{B}_{0,1} -\xi \w *(J_2\xi \w \tilde{B}_{0,3})+\xi \w *(J_3\xi \w \tilde{B}_{0,2}) - f_0 \xi \w J_1\xi\Big),
    \]
    \[
    \sigma(L_{\pi^2_- (da_2)})(x,\xi)(A) = 
    \pi^2_-\Big(+\frac{1}{2}|\xi|^2 \tilde{B}_{0,2}-\xi \w *(J_3\xi \w \tilde{B}_{0,1})+\xi \w *(J_1\xi \w \tilde{B}_{0,3}) - f_0 \xi \w J_2\xi\Big),
    \]
    \[
    \sigma(L_{\pi^2_- (da_3)})(x,\xi)(A) = 
    \pi^2_-\Big(+\frac{1}{2}|\xi|^2 \tilde{B}_{0,3}-\xi \w *(J_1\xi \w \tilde{B}_{0,2})+\xi \w *(J_2\xi \w \tilde{B}_{0,1}) - f_0 \xi \w J_3\xi\Big),
    \]
Observe that the above comprises of 12 independent terms which is indeed consistent with 
Theorem \ref{theorem: classification of second-order invariants}. Next we summarise the principal symbols of all the second-order invariants belonging to the trivial representation of $\SU(2)$:
    \begin{align*}
        \sigma(L_{\delta a_1})(x,\xi)(A) = -2f_1 |\xi|^2-*(\xi \w J_2 \xi \w \tilde{B}_{0,3})+*(\xi \w J_3 \xi \w \tilde{B}_{0,2}).
    \end{align*}
    \begin{align*}
        \sigma(L_{\delta a_2})(x,\xi)(A) = -2f_2 |\xi|^2-*(\xi \w J_3 \xi \w \tilde{B}_{0,1})+*(\xi \w J_1 \xi \w \tilde{B}_{0,3}).
    \end{align*}
        \begin{align*}
        \sigma(L_{\delta a_3})(x,\xi)(A) = -2f_3 |\xi|^2-*(\xi \w J_1 \xi \w \tilde{B}_{0,2})+*(\xi \w J_2 \xi \w \tilde{B}_{0,1}).
    \end{align*}
    \begin{align*}
        \sigma(L_{g(d a_1,\om_1)})(x,\xi)(A) = 
        f_0 |\xi|^2-g(\xi \w J_1\xi,\tilde{B}_{0,1})+g(\xi \w J_2\xi,\tilde{B}_{0,2})+g(\xi \w J_3\xi,\tilde{B}_{0,3}) 
    \end{align*}
        \begin{align*}
        \sigma(L_{g(d a_2,\om_2)})(x,\xi)(A) = 
        f_0 |\xi|^2-g(\xi \w J_2\xi,\tilde{B}_{0,2})+g(\xi \w J_3\xi,\tilde{B}_{0,3})+g(\xi \w J_1\xi,\tilde{B}_{0,1}) 
    \end{align*}
        \begin{align*}
        \sigma(L_{g(d a_3,\om_3)})(x,\xi)(A) = 
        f_0 |\xi|^2-g(\xi \w J_3\xi,\tilde{B}_{0,3})+g(\xi \w J_1\xi,\tilde{B}_{0,1})+g(\xi \w J_2\xi,\tilde{B}_{0,2}) 
    \end{align*}
     \begin{align*}
        \sigma(L_{g(d a_1,\om_2)})(x,\xi)(A) = 
        -g(\xi \w J_2\xi, \tilde{B}_{0,1})-g(\xi \w J_1\xi, \tilde{B}_{0,2}) 
    \end{align*}
     \begin{align*}
        \sigma(L_{g(d a_2,\om_3)})(x,\xi)(A) = 
        -g(\xi \w J_3\xi, \tilde{B}_{0,2})-g(\xi \w J_2\xi, \tilde{B}_{0,3}) 
    \end{align*}
     \begin{align*}
        \sigma(L_{g(d a_3,\om_1)})(x,\xi)(A) = 
        -g(\xi \w J_1\xi, \tilde{B}_{0,3})-g(\xi \w J_3\xi, \tilde{B}_{0,1}) 
    \end{align*}
     \begin{align*}
        \sigma(L_{g(d a_1,\om_3)})(x,\xi)(A) = 
        -g(\xi \w J_3\xi, \tilde{B}_{0,1})-g(\xi \w J_1\xi, \tilde{B}_{0,3})
    \end{align*}
      \begin{align*}
        \sigma(L_{g(d a_2,\om_1)})(x,\xi)(A) = 
        -g(\xi \w J_1\xi, \tilde{B}_{0,2})-g(\xi \w J_2\xi, \tilde{B}_{0,1})
    \end{align*}
      \begin{align*}
        \sigma(L_{g(d a_3,\om_2)})(x,\xi)(A) = 
        -g(\xi \w J_2\xi, \tilde{B}_{0,3})-g(\xi \w J_3\xi, \tilde{B}_{0,2})
    \end{align*}
Observe that the above comprises of $9 = 12-3$ independent terms since  $\sigma(L_{g(d a_i,\om_j)})=\sigma(L_{g(d a_j,\om_i)})$ which is again consistent with Theorem \ref{theorem: classification of second-order invariants}. 

\bibliography{biblioG}
\bibliographystyle{plain}

\end{document}